\documentclass[12pt]{amsart}
\usepackage{amsmath}
\usepackage{amsaddr}
\usepackage{amsthm}
\usepackage{amssymb}
\usepackage{tikz}
\usepackage{bm}
\usetikzlibrary{patterns}
\usepackage{graphicx}
\usepackage{enumitem}
\usepackage{subcaption}
\usepackage{amsfonts}
\usepackage{mathrsfs}
\usepackage{euscript}
\usepackage{xparse}
\usepackage{ifthen}
\usepackage{float}
\usepackage[normalem]{ulem}
\usepackage{geometry}
\usepackage{braket}
\usepackage{cancel}
\usepackage{hyperref}
\usepackage{color}
\usepackage{mathtools}

\newtheorem{theorem}{Theorem}[section]
\usetikzlibrary{decorations.pathreplacing,calligraphy}
\usetikzlibrary{matrix,arrows.meta}
\newcommand{\mtp}{\ensuremath{\texttt{a}}}
\newcommand{\alphamtp}{\ensuremath{\texttt{b}}}
\newcommand{\ptp}{\ensuremath{\texttt{a}_*}}
\newcommand{\alphaptp}{\ensuremath{\texttt{b}_*}}

\newcommand{\Sdelta}{\ensuremath{\mathtt{B}}}

\newcommand{\reflection}
{\ensuremath{\varrho}}

\newcommand{\btau}{\boldsymbol{\tau}}
\newcommand{\brho}{\boldsymbol{\rho}}
\renewcommand{\S}{S}
\newcommand{\s}{s}

\newcommand{\Pf}{\ensuremath{\mathrm{Pf}}}
\newcommand{\Ai}{\ensuremath{\mathrm{Ai}}}

\newcommand{\W}{W}
\newcommand{\V}{V}

\newcommand{\bt}{\mathbf{t}}
\newcommand{\bv}{\mathbf{v}}
\newcommand{\bu}{\mathbf{u}}
\newcommand{\bB}{\mathbf{B}}
\newcommand{\bb}{\mathbf{b}}
\newcommand{\bD}{\mathbf{D}}

\newcommand{\bH}{\mathbf{H}}
\newcommand{\bh}{\mathbf{h}}
\newcommand{\bx}{\mathbf{x}}
\newcommand{\bz}{\mathbf{z}}
\newcommand{\bs}{\mathbf{s}}
\newcommand{\by}{\mathbf{y}}

\newcommand{\bS}{\mathbf{S}}

\newtheorem{corollary}[theorem]{Corollary}
\newtheorem{lemma}[theorem]{Lemma}
\newtheorem{remark}[theorem]{Remark}

\newtheorem{proposition}[theorem]{Proposition}

\DeclareMathOperator{\tr}{tr}

\title{Totally Asymmetric Simple Exclusion Process in the Half-Space}
\author{Xincheng Zhang}
\address{University of Toronto}

\begin{document}


\begin{abstract}
In this work, we present the multipoint probability distribution of the totally asymmetric simple exclusion process (TASEP) in a half-space, starting from a general deterministic initial condition. More precisely, let $h(t,x)$ denote the height function of TASEP at position $x$ and time $t$; we provide an explicit formula for
\begin{equation*}
\mathbb{P}(h(t,y_1)\leq s_1, \ldots, h(t,y_m)\leq s_m).
\end{equation*}
The formula presented is well suited for asymptotic analysis. By applying a $1{:}2{:}3$ scaling, we derive the probability distribution for the half-space KPZ fixed point, which is conjectured to be the universal limiting process for models in the KPZ universality class restricted to a half-space.

\end{abstract}

\maketitle
\tableofcontents

\section{Introduction}

The Kardar--Parisi--Zhang (KPZ) universality class contains a wide range of random surface growth models and planar geometry models, which are conjectured to have the same large-time behavior. Its
characteristic $1{:}2{:}3$ scaling has led to several universal limiting
objects \cite{BDJ99,J2000,PS02Airy,MQR21,DOV22,BSS24}. For
one-dimensional growth models encoded by a height function $h(t,x)$,
the conjectural universal space-time limit is the \emph{KPZ fixed point}
\cite{MQR21}. Convergence to this limit has so far been proved for only
a limited collection of models \cite{MQR21,DV22thescalingoflis,wu2023,ACB24}. For more background on the KPZ universality class, see \cite{QuaintrotoKPZ, CorKPZequation,BGintro,zygintro,BDSintro}.

Models in the KPZ universality class have a half-space version (for surface growth models) or a symmetrized version (for last-passage percolation models). In addition to sharing the bulk dynamics of the full-space models, the half-space models have an extra parameter that characterizes the behavior at the origin. As a space-time random process, the height functions $h(t,x)$ for these models are restricted to the half-space $x \geq 0$ with a parameter governing the boundary condition $h(t,x = 0)$; these models are also conjectured to converge to the same universal limiting process, \emph{the half-space KPZ fixed point} $\mathfrak{h}(\bt,\bx),\bx\geq 0$, which depends on the parameter $\brho$. 

The two goals of the paper are to derive the explicit Fredholm Pfaffian formulas for the multipoint distributions of the half-space KPZ fixed point and the half-space TASEP, starting from a general deterministic initial condition. TASEP is an exactly solvable discrete model in the KPZ universality class. We will introduce the model in the following section.

\subsection{Half-space TASEP and previous results}
The half-space TASEP with injection rate $\alpha$ is a continuous-time Markov process on $\mathbb{Z}_{\geq 0}$. Particles attempt to jump to the right at rate $1$, subject to the exclusion rule: if the destination site is already occupied, the jump is suppressed and the waiting time restarts. New particles enter at site $0$ at rate $\alpha$ if site $0$ is empty. Let $\eta: \mathbb{Z}_{\geq 0}\to \{0,1\}$ denote the occupation configuration. Here $\eta_x$ is $1$ if there is a particle at position $x$ and $0$ otherwise. For a finite-range function $f:\{0,1\}^{\mathbb{Z}_{\geq 0}}\to \mathbb{R}$, the generator is given by
\begin{equation*}
 \mathcal{L} f(\eta)=\alpha(f(1,\eta_1,\eta_2,\ldots)-f(\eta_0,\eta_1,\ldots))+\sum_{x\geq 0}\eta_x(1-\eta_{x+1})(f(\eta_{x,x+1})-f(\eta)),
 \end{equation*} 
 where $\eta_{x,x+1}$ is obtained by switching the occupation variables $\eta$ at the sites $x$ and $x+1$. There is also a height representation $h(t,x),x\geq 0$, where
 \begin{equation}\label{def:heightfunction}
  h(t,x)=  \begin{cases}
      &h(0,0)-2J_{t}-\sum_{i=0}^{x-1}(1-2\eta_t(i)),\quad x\geq 1,\\
      &h(0,0)-2J_{t}, \quad x=0.\\
  \end{cases}
  \end{equation}
  Here $J_{t}$ is the number of new particles that have entered the system up to time $t$. Notice that when a new particle jumps into the system, the height function only changes at $0$. Under this convention, the height function at a local maximum decreases by \(2\) when the particle jumps.

   TASEP is one of the solvable models in the KPZ universality class. The full-space TASEP was first solved in \cite{Schutz97} using the coordinate Bethe ansatz. The formula given in \cite{Schutz97} is not conducive to asymptotic analysis. Later, the determinant was rewritten in \cite{Sas05,BFPS07}, where a biorthogonalization problem was formulated whose solution yields the transition probability. For certain special initial conditions, the biorthogonalization problem was solved \cite{BFP06,BFP08,BFPS07}. In \cite{MQR21}, the biorthogonalization problem was solved with a general initial condition. Later, a multitime, multipoint formula was derived in \cite{JR21multipoint,Liu22multipoint} for some special initial conditions.
   
   The first studies on half-space models in the KPZ universality class concerned
last-passage percolation (LPP) in the half-quadrant $\{(i,j):i\geq j\}$ with i.i.d.\
geometric weights, where the weights on the diagonal $\{i=j\}$ carry a separate
parameter $\alpha$ \cite{BR01srp,BR01asymp,BR01aais}; equivalently, this is LPP on the
full quadrant with weights symmetric across the diagonal. The diagonal parameter
plays the role of the injection rate of the half-space TASEP: it measures how strongly
the boundary attracts the maximizing paths. In these works a phase transition in
$\alpha$ was established for the one-point distribution of the last-passage time
$G(N,N)$ on the diagonal: there is a critical value $\alpha_0$ (depending on the model),
and the large-$N$ fluctuations of $G(N,N)$ depend on the relation between \(\alpha
\) and \(\alpha_0\). In the case where $0\leq \alpha<\alpha_0$, the fluctuations are of order $N^{1/3}$ and have the Tracy--Widom GSE distribution; if $\alpha=\alpha_0$, the fluctuations are of order $N^{1/3}$ and have the Tracy--Widom GOE distribution; if $\alpha>\alpha_0$, the fluctuations are of order $N^{1/2}$ and have a Gaussian distribution. Furthermore, depending on whether the position is away from the origin or not, one can see different types of limiting distributions. It is interesting that the Tracy--Widom GSE distribution is not present in the full-space models.

  The multipoint distribution of the PNG model in the half-space, with or without nucleation at $0$, was later studied in \cite{SI03}. The fluctuations of the process near the origin give the symplectic--unitary transition in random matrix theory in the subcritical case and give the orthogonal--unitary transition in the critical case \cite{PTG99}.

  Later, exponential last-passage percolation in the first quadrant was studied in \cite{BBCS16,BBCS16fep}. The authors started from the geometric LPP problem, scaled it to exponential LPP, and then performed an asymptotic analysis. The formula is established by realizing the dynamics of the half-space TASEP as a marginal of the Pfaffian Schur process. For more details about the Pfaffian Schur process, see \cite{BR05Eynardmehtatheorem,DJPM18,SI03}. They derived the multipoint distribution of a process that interpolates between the symplectic--unitary regime and the orthogonal--unitary regime. This requires a weak scaling of the parameter $\alpha$ around its critical value, so that in the limit, the effect of boundary injection is still seen.

 From the point of view of lateral growth, all these models have been solved for the narrow-wedge initial condition at $0$. In terms of the half-space TASEP, that is equivalent to the model starting with all sites initially empty.

    For TASEP starting from the product Bernoulli initial condition, which is equivalent to the half-space stationary LPP model, a formula was derived in \cite{DPA20,DPA22}. The authors also used the connection between the model and the Pfaffian Schur process.

    Our first result gives an explicit Fredholm Pfaffian formula for the half-space TASEP. We recall the Fredholm Pfaffian notation. Let $(X,\mu)$ be a
measure space, let $\mathcal H=L^2(X,\mu)$, and set
$J=\left(\begin{smallmatrix}0&I\\-I&0\end{smallmatrix}\right)$
on $\mathcal H\oplus\mathcal H$.
For a $2\times2$ matrix kernel $K$ satisfying
$K(x,y)=-K(y,x)^{\mathsf T}$, the Fredholm Pfaffian is
\[
\Pf(J+K)_{\mathcal H}
=
1+\sum_{n=1}^{\infty}\frac{1}{n!}
\int_{X^n}
\operatorname{pf}\!\left(
\bigl[K(z_i,z_j)\bigr]_{i,j=1}^{n}
\right)
\prod_{i=1}^{n}\mu(dz_i),
\]
whenever the series converges absolutely. Here
$\operatorname{pf}$ denotes the ordinary Pfaffian of the
$2n\times2n$ block matrix. The class of operators for which this series converges and the
properties used below are given in
Appendix~\ref{app:generalizedFredholm}.
    \begin{theorem}\label{thm:introthm1}
        Fix $\alpha,t>0$. Let $h_0\in\mathcal H_{\mathrm{init}}$, where the space is
defined in \eqref{def:heightspacedef}. Let $h(t,x)$ be the half-space TASEP height function defined in \eqref{def:heightfunction}, starting from \(h_0\). Let $ y_1< y_2<\cdots < y_m\in \mathbb{Z}_{\geq0}$ and $s_1,\ldots, s_m\in \mathbb{Z}$ be such that $y_i+s_i$ is even and $|s_{i+1}-s_i| <y_{i+1}-y_i$ for each \(i\). Then
        \begin{equation}
            \mathbb{P}(h(t,y_1)\leq s_1, \ldots, h(t,y_m)\leq s_m) = \Pf(J+JK)_{L^2(\{1,\ldots,m\}\times [0,\infty))}.
        \end{equation}
        The explicit definition of the kernel is given in \eqref{eqn:defRij}--\eqref{eqn:ktilde}.
    \end{theorem}
An important feature of the formula is that information about the initial condition appears as a first-hitting density for a random
walk, which is the same as in the full-space case \cite{MQR21}.

There are many other models in the KPZ universality class that are studied in the half-space. Examples include polymer models with a wall \cite{BD21,Timo12,BBC15,BCD23,SayanZhu24}, the stochastic six-vertex model in the half-quadrant \cite{BBCW17,GGMW24}, the half-space ASEP \cite{BC24,He24current,Parekh19}, the half-space KPZ equation \cite{BKL20halfspacestatKPZequation,DKLT20,BKL22,KD20replicabetheansatz}, the half-space Macdonald process \cite{BBC20halfspaceMac}, and many algebraic structures related to half-space models \cite{IMS22,IMS23,Theo23}. For more studies related to the properties of models in a half-space, see \cite{he2022shiftinvariancehalfspace,FO24,chen2024secondclassparticlehalfline}.


TASEP in other geometries has also been studied; for example, TASEP on a periodic domain is studied in \cite{BL21,BLS22,Liao22}. 

\subsection{The $1{:}2{:}3$ scaling limit}

The formula in Theorem~\ref{thm:introthm1} is suitable for asymptotic
analysis with the KPZ scaling.  For $\varepsilon>0$, the $1{:}2{:}3$ rescaled TASEP height function  is 
    \begin{equation}\label{eqn:halfheightscaleintro}
            \bh^{\varepsilon}(\bt,\bx):=\varepsilon^{1/2}[h(2\varepsilon^{-3/2} \bt, 2\varepsilon^{-1}\bx)+\varepsilon^{-3/2}\bt],
            \end{equation}
with lattice-point corrections omitted, and we scale the boundary rate by
\begin{equation*}
\alpha^\varepsilon
=
\frac{1+\varepsilon^{1/2}\brho}{2}.
\end{equation*}
This is the characteristic $1{:}2{:}3$ scaling around the critical boundary rate $1/2$. The following is the main scaling theorem.

\begin{theorem}\label{thm:intro-fixed-point}
Fix $\bt>0$, $\brho\in\mathbb R$. Assume that \(\bh^\varepsilon(0,\bx)\) converges in the sense specified in \eqref{eqn:multiple-wedge-UC-convergence}. Then for
$0\leq\by_1<\cdots<\by_m$ and
$\bs_1,\ldots,\bs_m\in\mathbb R$,
\begin{equation*}
\lim_{\varepsilon\to0}
\mathbb P\bigl(
\bh^\varepsilon(\bt,\by_1)\leq\bs_1,\ldots,
\bh^\varepsilon(\bt,\by_m)\leq\bs_m
\bigr)
=
\Pf(J+JK^{\mathrm{fp}})
_{L^2(\{1,\ldots,m\}\times[0,\infty))},
\end{equation*}
where $K^{\mathrm{fp}}$ is the explicit matrix kernel in
\eqref{def:fixedppointRdef}--\eqref{def:fixedppointtildeKdef}.
\end{theorem}
The formula for general initial conditions gives the transition
probabilities for the half-space KPZ fixed point. To construct the
process from the TASEP scaling limit, one must also establish
tightness of the rescaled processes and verify that the limiting
transition probabilities satisfy the Chapman--Kolmogorov equations.
We will address these questions in a separate paper.

A natural question is whether our formula reduces to the formula in \cite{BBCS16} for the narrow-wedge initial condition at $0$. The answer is that our kernel does not coincide term by term with that of \cite{BBCS16}, but it can be proved that the two Pfaffians are equal by algebraic transformations of the kernel.

\subsection{Method and organization of the paper}
The main difficulty in solving the half-space TASEP (or the full-space TASEP) from a general initial condition is that there is no algebraic correspondence to the Pfaffian Schur process. The Pfaffian Schur process can only be utilized for a few initial conditions. The other route, which is how the full-space TASEP is solved, is by a biorthogonalization argument for a Sch\"utz-type formula. The Sch\"utz-type formula and a skew-biorthogonalization problem were recently established in \cite{degier2025pfaffian}. However, how to solve that skew-biorthogonalization problem in a way that is suitable for asymptotic analysis remains open.

We will solve the half-space TASEP with a guess-and-check approach. Consider a continuous-time Markov process $X_t$ on the state space $(S,\mathcal{S})$ with generator $\mathcal{L}$. For any $A \in \mathcal{S},x\in S$, the Markov transition function $P_t(x,A)$ satisfies the Kolmogorov backward equation:
\begin{equation}\label{eqn:kol}
    \begin{cases}
        (\partial_t-\mathcal{L})P_t(x,A) = 0,\\
        \lim_{t\to 0}P_t(x,A) = 1_{x\in A}.
    \end{cases}
\end{equation}
If we can find a function $\tilde{P}_t(x,A)$ that satisfies \eqref{eqn:kol} and establish uniqueness of solutions to \eqref{eqn:kol}, then the function $\tilde{P}_t(x,A)$ must be the transition function for the Markov process.

 We will use this general scheme to derive a formula for the one-time multipoint distribution of the half-space TASEP, namely 
 \begin{equation*}
     \mathbb{P}(h(t,y_1)\leq s_1,\ldots,h(t,y_m)\leq s_m).
 \end{equation*}
 The formula is largely inspired by \cite{MQR21,NQR20,BBCS16,BBCS16fep}. The first paper, in which the full-space TASEP was solved, reveals the key philosophy that ``initial conditions should appear in the formula as a hitting probability''. In the second paper, the Kolmogorov equation is verified for the formula in \cite{MQR21}. In the last two papers, the multipoint formula is given for the half-space TASEP starting from the narrow-wedge initial condition at $0$. The formula in \cite{BBCS16} can be viewed as containing information about a general final configuration. 

 Another key ingredient that facilitates the guess part is the skew-time reversal invariance property of the half-space TASEP:
 \begin{equation*}
     \mathbb{P}(h(t,x)\leq g(x),\forall x|h(0,\cdot)=f) =\mathbb{P}(h(t,x)\leq -f(x),\forall x|h(0,\cdot)=-g).
 \end{equation*}
 This property leads to the conjecture that the appearance of the initial and final configurations should be similar in the formula. Combining this idea with an investigation of \cite{BBCS16}, we found the multipoint formula for a general initial condition. In other words, the formula can be thought of as an ``ansatz'' based on the Pfaffian Schur process formula. 
 
 The paper is organized as follows. The main formula for TASEP and all definitions are presented in Section~\ref{subsec:notationandresult}. In Section~\ref{subsec:propertyofoperators}, we develop some properties of the kernel and prove that the Pfaffian in the main theorem is well defined. In Section~\ref{sec:kolm} and Section~\ref{subsec:initialcondition}, we prove that the formula we present satisfies \eqref{eqn:kol}.

 In Section~\ref{sec:scalinglimit}, we derive the scaling limit formula. Section~\ref{subsec:transformationofkernel} transforms the kernel into a form that is suitable for the asymptotic analysis. In Section~\ref{sec:pointwiselimit}, we establish the bounds and pointwise limits for each component of the kernel, and in Section~\ref{subsec:Fredholm-convergence}, we prove the convergence of the kernel in a suitable norm, which ensures the convergence of the Fredholm Pfaffian.

The paper is heavy in notation. All the definitions of variables and operators for TASEP appear in Section~\ref{subsec:notationandresult} and those for the half-space KPZ fixed point appear in Section~\ref{sec:pointwiselimit}.
 
\subsection{Acknowledgments}
 The author thanks Jeremy Quastel for the idea that one should be able to guess the formula from the perspective of solving the Kolmogorov equation. The author is also grateful to Jeremy Quastel for many discussions of the problems and the writing of the paper. The author also thanks Henry Hu for some discussions on the problem and Zongrui Yang for pointing out many related works in the literature. The author wants to thank Patrik Ferrari, Sabrina Gernholt, and Heng-Yu Chang for pointing out errors in an earlier version of this paper.
 This research was partially supported by NSERC.
 
\section{Solving the half-space TASEP}
\subsection{Notation and the main result}\label{subsec:notationandresult}
We first specify the assumptions on the height functions, and give a new parametrization of the half-space TASEP. 
We are interested in the following probability distribution:
\begin{equation*}
    \mathbb{P}(h(t,x;h_{\mathrm{init}})\leq h_{\mathrm{final}}).
\end{equation*}
Both $h_{\mathrm{init}}$ and $h_{\mathrm{final}}$ are functions $f:\mathbb{Z}_{\geq 0}\to \mathbb{Z}$ such that $f(x+1)-f(x)\in\{\pm 1\}$, representing the height functions of TASEP. There are slightly different assumptions on the types of functions $h_{\mathrm{init}}$ and $h_{\mathrm{final}}$ that are allowed. We define
\begin{equation}\label{def:heightspacedef}
        \begin{split}
            &\mathcal{H}_{\mathrm{init}} := \{h:\lim_{x\to\infty}h(x)= -\infty, h\text{ has finitely many local maxima}, h(0)\text{ is even}\},\\
            &\mathcal{H}_{\mathrm{final}} := \{h:\lim_{x\to\infty}h(x)= \infty, h\text{ has finitely many local minima},h(0)\text{ is even}\}.
        \end{split}
\end{equation}
Then we only consider $h_{\mathrm{init}}\in \mathcal{H}_{\mathrm{init}}, h_{\mathrm{final}}\in \mathcal{H}_{\mathrm{final}}$. These conditions are not too restrictive as we will see that they cover the finite-dimensional distributions. Under these assumptions, $h_{\mathrm{init}}$ is uniquely determined by the positions of the peaks $x_i$ and the heights of the peaks $h_i$. We use the notation 
\begin{equation}\label{eqn:halfpeakrepre}
(\vec{x},\vec{h})_t=(x_1,h_1;x_2,h_2;\ldots;x_n,h_n)_t,\quad 0 \leq x_1 < \cdots < x_n
\end{equation}
to denote the peak configuration at time $t$. If a tuple $(\vec{x},\vec{h})$ represents a configuration of TASEP, it satisfies the following parity and spacing constraints: $x_i+h_i$ all have the same parity; $|h_{i+1}-h_{i}|<x_{i+1}-x_i$. Similarly, $h_{\mathrm{final}
}(x)$ is uniquely determined by the positions of the troughs $y_i$ and the heights of the troughs $s_i$. We use  
\begin{equation}\label{eqn:halftroughrepre}
   \{\vec{y},\vec{s}\} =\{y_1,s_1;y_2,s_2;\ldots;y_m, s_m\},\quad 0 \leq y_1 < \cdots < y_m
\end{equation}
to denote it. We define one TASEP height function to be less than the other if the inequality holds coordinatewise. Notice that the multipoint distribution of TASEP \[\mathbb{P}_{h_0}(h(t, y_1)\leq s_1;\ldots;h(t,y_m)\leq s_m)\] can be written as $\mathbb{P}((\vec{x},\vec{h})_t\leq \{\vec{y},\vec{s}\})$, where $h_0=(\vec{x},\vec{h})_0$.

We now develop another equivalent way to parameterize the TASEP height function. Given a peak configuration $(\vec{x},\vec{h}) = (x_1,h_1;\ldots; x_n,h_n)$, we define \emph{the primordial peak} of $(\vec{x},\vec{h})$ to be the point $(x_\star,h_\star)$ where
\begin{equation}\label{def:primcoord}
 x_{\star} := \frac{h_n-h_1+x_n+x_1}{2}, \quad  h_\star := \frac{h_1+h_n+x_n-x_1}{2}.
\end{equation}
Geometrically, the primordial peak is the point of intersection between the line of slope $+1$ passing through the first peak and the line of slope $-1$ passing through the last peak; see the red point in Figure~\ref{figure:halfTASEP}.

Further, we want to define variables that record the relative positions of peaks. Define
\begin{equation}\label{def:uidicoord}
u_i := (x_{i+1}-x_i+h_{i+1}-h_i)/2 \geq 1, \quad d_i := (x_{i+1}-x_i-h_{i+1}+h_i)/2 \geq 1.
\end{equation}
Geometrically, \(u_i, d_i\) are the numbers of up-steps and down-steps between the peaks at \(x_i\) and \(x_{i+1}\), respectively. See Figure~\ref{figure:halfTASEP}.
For \(i<j\), set
\begin{equation}
    \begin{split}
        u := u_1+\cdots+u_{n-1},\quad d := d_1+\cdots+d_{n-1},\quad u_{ij}:=\sum_{k=i}^{j-1}u_k,\quad d_{ij} :=\sum_{k=i}^{j-1}d_k,
    \end{split}
\end{equation}
with the convention that \(u=d=0\) when \(n=1\).
It is easy to see that $u$ is the number of up-steps from the first peak to \emph{the primordial peak} and $d$ is the number of down-steps from \emph{the primordial peak} to the last peak; see Figure~\ref{figure:halfTASEP}. 

\begin{figure}[!htbp]
    \centering
    \begin{tikzpicture}[scale=0.4]

\draw[very thin, gray, dotted] (-2,2) grid (10,10);

\draw[thick, ->] (-2,4) -- (10,4) node[right] {$x$};
\draw[thick, ->] (-1,2) -- (-1,10) node[above] {$y$};


\draw[thick]  (-1,6) -- (0,5)--(1,6) --(2,5)-- (4,7) -- (8,3);
\draw[thick, dotted] (-1,6) -- (2,9) -- (4,7);
\node[] at (0.5,7.5){\footnotesize$u=3$};
\node[] at (3.5,7.5){\footnotesize$d=2$};

\filldraw[red] (2,9) circle (4pt);
\node[right, red,shift={(-0.5,0.5)}] at (2,9) {\small$(3,5)$};
\filldraw[black] (-1,6) circle (4pt);
\filldraw[black] (4,7) circle (4pt);
\filldraw[black] (1,6) circle (4pt);
\node[anchor=west, font=\scriptsize] at (5.3,8.2) {$
\begin{aligned}
d_1&=1,\qquad u_1=1,\\
d_2&=1,\qquad u_2=2.
\end{aligned}
$};
\end{tikzpicture}
    \caption{Configuration $(0,2;2,2;5,3)$ with \emph{the primordial peak} $(3,5)$}
    \label{figure:halfTASEP}
\end{figure} 

For a trough configuration $\{y_1,s_1;\ldots;y_m,s_m\}$, we use $u',d',u_i',d_i',u_{ij}',d_{ij}'$ to denote the variables $u,d,u_i,d_i,u_{ij},d_{ij}$ parameterized by the peak configuration $(y_1,-s_1;\allowbreak\ldots;\allowbreak y_m,-s_m)$. More specifically,
\begin{equation}\label{def:uidiprimecoord}
u'_i = (y_{i+1}-y_i-s_{i+1}+s_i)/2, \quad d'_i = (y_{i+1}-y_i+s_{i+1}-s_i)/2.
\end{equation}
The primordial trough of the configuration is defined to be
\begin{equation}\label{def:primtroughcoord}
        y_{\star} := \frac{-s_m+s_1+y_m+y_1}{2}, \quad  s_\star := \frac{s_1+s_m-y_m+y_1}{2}.
\end{equation}

We now define the following ``operator product ansatz'' kernel parameterized by the configuration $(x_1,h_1;\ldots;x_n,h_n)$, at time $t$. We define an operator $V$:
\begin{equation}\label{def:Vdef}
        V := \ptp^{-x_\star} W\ptp^{x_\star},
\end{equation}
where
\begin{equation}\label{def:Wdef}
    \begin{split}
        &W:= (I-W_0\overline{1}^tW_{1,2}\overline{1}^tW_{2,3}\cdots\overline{1}^t W_{n-1,n}\overline{1}^tW_{n+1}),\quad n>1;\\
        &W:=(I-\overline{1}^t),\quad n =1;\\
&W_{i,i+1} := \mtp^{-u_i}\ptp^{-d_i},\quad  W_0 := \mtp^{u}, \quad W_{n+1} := \ptp^{d}, 
    \end{split}
\end{equation}
where
\begin{equation}\label{def:aastardef}
\mtp :=1-2D, \quad \ptp := 1+2D.
\end{equation}
$D$ is the differential operator, in the distributional sense. $\mtp^{-1}, \ptp^{-1}$ are integral operators with the kernels
\begin{equation}\label{def:aastarinversedef}
    \mtp^{-1}(x,y) := \frac{1}{2}e^{(x-y)/2}1_{x< y}, \quad \ptp^{-1}(x,y) := \frac{1}{2}e^{(y-x)/2}1_{y< x}.
\end{equation}
Notice that $\mtp \mtp^{-1}(x,y) =\ptp \ptp^{-1}(x,y)= \delta(x-y)$, where $\delta(x)$ is the Dirac delta function at $0$. In this paper, we do not distinguish the integral operator from its integral kernel; the meaning should be clear from the context.

We use projection operators with the following multiplication kernels:
$1^{a}_b(x) = 1_{b<x <a}$. We use $\overline{1}$, $\underline{1}$, and $\overline{\underline{1}}$ whenever the endpoints are included, and $a$ or $b$ is omitted if it is $\infty$ or $-\infty$, respectively. The indices in $\W_{0},\W_{n+1}$ have no meaning; they are used simply for notational convenience.

In some scenarios one wants to bring $\ptp^{-x_\star}$ and $\ptp^{x_\star}$ into the bracket and group them with $W_0$ and $W_{n+1}$; thus we define
\begin{equation}\label{eqn:V0Vndef}
    V_0 := \ptp^{-x_\star}W_0,\qquad V_{n+1} := W_{n+1}\ptp^{x_\star}.
\end{equation}

We have not specified the spaces on which the operators act; this is a subtle issue and we will discuss it in Proposition~\ref{prop:Wequiv}.

For the configuration $(0,2;2,2;5,3)$ in Figure~\ref{figure:halfTASEP},
\[
u_1=1,\quad d_1=1,\quad u_2=2,\quad d_2=1,\qquad u=3,\quad d=2, \qquad x_\star  = 3.
\]
Thus
\begin{equation}\label{eqn:wconcreteexample}
        W=I-\mtp^3\overline{1}^{t}\mtp^{-1}\ptp^{-1}
\overline{1}^{t}\mtp^{-2}\ptp^{-1}
\overline{1}^{t}\ptp^{2}.
\end{equation}

Lastly, we want to define some variables that record the position of a primordial peak $(x_{\star},h_{\star})$ with respect to a trough $\{y_i,s_i\}, i =1,\ldots,m$. We define the following two objects, which we will call \emph{cones}. Let $C_{x,y}$ be the cone starting at $(x,y)$, opening upward; that is, 
\begin{equation}\label{def:Coneupdef}
        C_{x,y}:=\{(a,b)\in \mathbb{Z}^2: b > |a-x|+y\}.
\end{equation}
$C^{x,y}$ is the cone that starts at $(x,y)$ and opens downward; that is,
\begin{equation}\label{def:Conedowndef}
        C^{x,y} := \{(a,b)\in \mathbb{Z}^2: b < -|a-x|+y\}.
\end{equation}
Now define $l_i,l_i',r_i,r_i'\in \mathbb{Z}$ to be
\begin{equation}
    \begin{split}
        l_i& :=  (h_n+x_n-s_i+y_i)/2,\qquad r_i :=(h_1-x_1-s_i+y_i)/2,\\
    l_i' &:=(-s_i+y_i+x_n+h_n)/2,\qquad r_i' := (-s_i-y_i+x_n+h_n)/2.
    \end{split}
\end{equation}
$l_i,r_i$ are the signed diagonal distances from the primordial peak to the left and right sides of the cone $C_{0,s_i-y_i}$, respectively; $l'_i,r'_i$ are the signed diagonal distances from the trough $\{y_i,s_i\}$ to the left and right sides of the cone $C^{0,h_{\star}+x_{\star}}$, respectively.
The geometric picture is presented in Figure~\ref{fig:lrpicture}. There are several properties that are easy to observe from the picture:
\begin{itemize}
    \item $l_i-r_i = x_{\star}\geq0, \quad l_i'-r_i'=y_i\geq0$ for all $i$.
    \item $l_i = l_i'$ for all $i$; \quad $l_i-r_i\geq u$.
    \item All the variables are integers due to the parity constraints on peaks and troughs.
\end{itemize}

\begin{figure}[!htbp]
    \centering
    \begin{minipage}{0.45\textwidth}
\centering
    \begin{tikzpicture}[scale=0.25]

\draw[very thin, gray, dotted] (-10,-2) grid (10,10);

\draw[thick, ->] (-1,-2) -- (10,-2) node[right] {$x$};
\draw[thick, ->] (-1,-2) -- (-1,12) node[above] {$y$};


\draw[thick]  (-1,6) -- (0,5)--(1,6) --(2,5)-- (4,7) -- (11,0);
\draw[thick,blue]  (-1,4) -- (1,2)--(8,9);
\draw[thin,dashed, blue]  (-8,7)--(-1,0) -- (1,2);
\draw[thick, dotted] (-4,3) --(-1,6) -- (2,9) -- (4,7);
\draw[decorate,decoration={calligraphic brace, amplitude=5pt, raise=1pt}] 
    (-4,3) -- (2,9) node [midway, sloped, above=5pt, xshift=-1pt] {\tiny$l_{i}$};
    \draw[decorate,decoration={calligraphic brace, amplitude=5pt, raise=1pt}] 
     (2,9)--(5,6)  node [midway, sloped, above=5pt, xshift=-1pt] {\tiny$ r_{i}$};

\node[] at (2.3,2){\tiny$(y_i,s_i)$};
\node[] at (-2,0){\tiny$C_{0,s_i-y_i}$};

\filldraw[black] (2,9) circle (4pt);
\node[] at (3,10.5){\tiny$(x_{\star},h_{\star})$};

\end{tikzpicture}
\end{minipage}
\begin{minipage}{0.45\textwidth}
\centering
   \begin{tikzpicture}[scale=0.25]

\draw[very thin, gray, dotted] (-10,-2) grid (10,10);

\draw[thick, ->] (-1,-2) -- (10,-2) node[right] {$x$};
\draw[thick, ->] (-1,-2) -- (-1,12) node[above] {$y$};


\draw[thick]  (-1,6) -- (0,5)--(1,6) --(2,5)-- (4,7) -- (11,0);
\draw[thick,blue]  (-1,4) -- (1,2)--(8,9);
\draw[thick, dotted] (-1,6) -- (2,9) -- (4,7);
\draw[thick, dotted] (-1,12) -- (2,9) ;
\draw[thick, dotted] (-1,12) -- (-8,5) ;
\draw[thin, blue,dashed] (1,2) -- (-8,11) ;

\node[] at (2.3,1.5){\tiny$(y_i,s_i)$};

\draw[decorate,decoration={calligraphic brace, amplitude=5pt, raise=1pt}] 
      (5,6)--(1,2) node [midway, sloped, below=5pt, xshift=-1pt] {\tiny$r'_{i}$};
    \draw[decorate,decoration={calligraphic brace, amplitude=5pt, raise=1pt}] 
     (1,2)--(-5,8)  node [midway, sloped, below=5pt, xshift=-1pt] {\tiny$ l'_{i}$};
\filldraw[black] (2,9) circle (4pt);
\node[] at (3,9.5){\tiny$(x_{\star},h_{\star})$};
\node[] at (0,12){\tiny$C^{0,h_{\star}+x_{\star}}$};

\end{tikzpicture}
\end{minipage}
    \caption{Geometric picture for $l_i,r_i,l_i',r_i'$}
    \label{fig:lrpicture}
\end{figure}
$l_i',r_i'$ record the information about the positions of troughs; they appear in the formula through the following operator. 
Define
\begin{equation}\label{def:Smatrixdef}
        \mathcal{S}_i := \begin{pmatrix}
            \S_H^{l_i',r_i'}\\
            -\S^{r_i',l_i'}
        \end{pmatrix},
\end{equation}
where $\S^{n,m}$ and $\S_H^{n,m}$ are defined to be
\begin{equation}\label{def:Sijabdef}
\begin{split}
&\s^{n,m}(x,y) := \int_{C_{1/2}} e^{-(x+y)w}\frac{(1+2w)^m}{(1-2w)^n}\frac{dw}{2\pi i},\\
&\s_H^{n,m}(x,y) := \int_{\Gamma_\alpha} \frac{(-w)(1+2w)^me^{-(x+y)w}}{(1-2w)^n(2\alpha-1+2w)(2\alpha-1-2w)}\frac{dw}{2\pi i},\\
& \S^{n,m}(x,y) := \s^{n,m}(x,y)1_{x+y\geq 0},\quad \S_H^{n,m}(x,y) := \s_H^{n,m}(x,y)1_{x+y\geq 0},
\end{split}
\end{equation}
where $C_a$ is a simple, negatively oriented loop enclosing only the pole at $w=a$. For $\alpha>0$, the contour \(\Gamma_\alpha\) is defined by
\begin{equation}\label{eqn:contourchoice}
    \Gamma_\alpha :=
\begin{cases}
C_{1/2}, & \alpha=1,\\[2pt]
C_{1/2}+ \frac12 C_0, & \alpha=\frac12,\\[2pt]
C_{1/2}+ C_{(2\alpha-1)/2},
    & \alpha\notin\left\{\frac12,1\right\}.
\end{cases}
\end{equation}
Here the contours are regarded as oriented chains, so the coefficient
$\frac12$ means that the integral over $C_0$ is taken with weight
$\frac12$.

Lastly, we define the operators which contain the boundary parameter $\alpha$. For $\alpha >0$, let
\begin{equation}\label{def: alphaDdef}
    \alphamtp := 2\alpha-1-2D, \quad \alphaptp := 2\alpha-1+2D.
\end{equation}
For $\alpha>0,\alpha\neq 1/2$, define \[\normalfont\alphamtp^{-1}(x,y) := \tfrac{1}{2}e^{(2\alpha-1)(x-y)/2}1_{x< y},\quad \alphaptp^{-1}(x,y) := \tfrac{1}{2}e^{(2\alpha-1)(y-x)/2}1_{y < x}.\]
    When $\alpha = 1/2$, \[ \normalfont  \alphamtp^{-1}(x,y)  :=\tfrac{1}{4}(1_{x<y}-1_{x> y}),\quad \alphaptp^{-1}(x,y) :=\tfrac{1}{4}(-1_{x<y}+1_{x> y}).\]   
Define 
\begin{equation}\label{def:HHinversedef}
   \begin{split}
        &H(x,y):=-1_{x> y}\tfrac{1}{8}e^{(1-2\alpha)(x-y)/2} +1_{x<y}\tfrac{1}{8}e^{(1-2\alpha)(y-x)/2},\\
        &H^{-1} :=D^{-1}\alphamtp\alphaptp, \qquad D^{-1}(x,y) := -1_{x< y}.
     \end{split}  
\end{equation}
When $\alpha>1/2$, $H(x,y)$ is the integral kernel for the operator $\alphamtp^{-1}D\alphaptp^{-1}$. As the notation suggests, $\S,\S_H, H,H^{-1}$ have certain relations, proved in Lemma~\ref{lem:HHinverseonS}. The reason for the choice of the contour $\tfrac{1}{2}C_0$ when $\alpha = 1/2$ in \eqref{eqn:contourchoice} is also manifested in Lemma~\ref{lem:HHinverseonS}. Notice that \(H^* = -H\) but \((H^{-1})^*\neq -H^{-1}\), due to the choice of \(D^{-1}\). The operators $H,H^{-1}$ have the following properties, whose proofs follow from straightforward calculus calculations.

\begin{lemma}\label{lem:bbstarcancelrelation}
    $H,H^{-1}$ commute with $\normalfont\mtp$ and $\normalfont\ptp$, and the following distributional identities hold:
        \normalfont
        \begin{equation}
                \begin{split}
                \normalfont&\alphaptp H=D\alphamtp^{-1},\quad  H\alphamtp=\alphaptp^{-1}D.
        \end{split}
        \end{equation}

    \end{lemma}

Now we are ready to state the main theorem.
 \begin{theorem}\label{thm:halfmultipoint}
    Let $\alpha>0,t >0$. Let $(\vec{x},\vec{h})=(x_1,h_1;\ldots; x_n,h_n) \in \mathcal{H}_{\mathrm{init}}$ and $\{\vec{y},\vec{s}\}=\{y_1,s_1;\ldots;y_m,s_m\}\in \mathcal{H}_{\mathrm{final}}$, where $\mathcal{H}_{\mathrm{init}},\mathcal{H}_{\mathrm{final}}$ are defined in \eqref{def:heightspacedef}. Assume that the half-space TASEP with injection rate $\alpha$ starts from the initial configuration $(\vec{x},\vec{h})$. The probability that at time $t$ it is below the configuration $\{\vec{y},\vec{s}\}$ is given by
\begin{equation}\label{eqn:halfmultipoint}
\begin{split}
\mathbb{P}((x_1,h_1;\ldots; x_n,h_n)_t \leq \{y_1,s_1;\ldots;y_m,s_m\}) =\Pf(J+JK)_{L^2(\{1,\ldots,m\}\times [0,\infty))},
\end{split}
\end{equation}
where \(J_{ij}(x,y)
=
\delta_{ij}\delta(x-y)\begin{pmatrix}0 & 1\\-1&0\end{pmatrix}\) and \(K=(K_{ij})_{i,j=1}^m\) is the \(2\times2\)
matrix-valued operator kernel. $K_{ij} =R_{ij}
+\tilde{K}_{ij}$,
where 
\begin{equation}\label{eqn:defRij}
\normalfont
    R_{ij}= \begin{pmatrix}
     1_{j<i}(\ptp)^{-u'_{ji}}(\mtp)^{-d'_{ji}}&  -\mtp^{r_i'-l_j'}H\ptp^{r_j'-l_i'}\\
    0 &1_{i<j}(\mtp)^{-u'_{ij}}(\ptp)^{-d'_{ij}}
\end{pmatrix},
\end{equation}
and 
\begin{equation}\label{eqn:ktilde}
\normalfont
        \tilde{K}_{ij}= \mathcal{S}_i
(V  H^{-1}-(V H^{-1})^*  +(V H^{-1})^*H V H^{-1})\mathcal{S}_j^*J.
\end{equation}
Here \(V\), \(\mathcal S_i\), \(H\), \(H^{-1}\), and the parameters
\(u'_{ij},d'_{ij},l_i',r_i'\) are defined above. The resulting kernel
\(JK\) is skew-symmetric.
\end{theorem} 

Although the space of the Pfaffian is \(L^2(\{1,\ldots,m\}\times [0,\infty))\), all the operator compositions are in the full space. One should think of the kernel as \(1_0K1_0\).
Throughout this paper, Fredholm determinants and Fredholm Pfaffians
are understood in the generalized sense described in Appendix
\ref{app:generalizedFredholm}. In particular, the diagonal blocks of
the kernel are required to be trace class, whereas the
off-diagonal blocks are only required to be Hilbert--Schmidt, after a suitable conjugation. We say such an operator is in the \emph{the mixed Schatten class}, defined in \eqref{def:mixedschatternclass}. The
generalized objects agree with the usual Fredholm determinant and
Pfaffian whenever the operator is trace class.

It is not obvious that the kernel is a well-defined integral operator on $L^2[0,\infty)$. There are two main problems: one is that $W_0$ and $W_{n+1}$ contain differential operators, but the objects on both sides of $W_0$ and $W_{n+1}$ are not differentiable because they contain indicator functions; the other is that the $H$ operator defined in \eqref{def:HHinversedef} does not decay at $\infty$ when $0<\alpha\leq 1/2$. We will prove the following result establishing well-definedness. 

\begin{proposition}\label{prop:welldefineandtraceclass}
Let \(\mathcal H
=
L^2\bigl(\{1,\ldots,m\}\times[0,\infty)\bigr)\). Let $\theta_1,\ldots,\theta_m$ satisfy
\[
\max\left\{0,\frac{1-2\alpha}{2}\right\}
<
\theta_1<\cdots<\theta_m<\frac12,
\]
and let $\mathcal M$ be the multiplication operator
\[
(\mathcal MF)_i(x)
=
\begin{pmatrix}
M_{-\theta_i}&0\\
0&M_{\theta_i}
\end{pmatrix}
F_i(x),\quad M_{\theta}f(x) = e^{\theta x}f(x).
\]
Then the kernel $K$ in \eqref{eqn:halfmultipoint} is well defined, and
the blocks of its conjugation satisfy
\begin{equation*}
    \begin{split}
        (\mathcal MK\mathcal M^{-1})_{11},
\,
(\mathcal MK\mathcal M^{-1})_{22}
\in\mathfrak S_1(\mathcal H),\\
(\mathcal MK\mathcal M^{-1})_{12},
\,
(\mathcal MK\mathcal M^{-1})_{21}
\in\mathfrak S_2(\mathcal H).
    \end{split}
\end{equation*}
Here \(\mathfrak S_1(\mathcal H)\) is the trace-class ideal and \(\mathfrak S_2(\mathcal H)\) is the Hilbert-Schmidt ideal. In particular, the Fredholm Pfaffian is well defined in the sense of \eqref{eqn:appendixmixedFredholmPfaffian}.
\end{proposition}
We will prove the proposition after studying some properties of the kernel in the next section.

In the rest of Section 2, we will prove the main theorem. In Section~\ref{subsec:propertyofoperators}, we prove that the kernel is well defined and in the mixed Schatten class after conjugation. In Section~\ref{sec:kolm}, we prove that the formula satisfies the Kolmogorov equation and in Section~\ref{subsec:initialcondition}, we prove that the formula has the correct initial condition.
\subsection{Properties of the operators}\label{subsec:propertyofoperators}
The main purpose of this section is to establish that the kernel in \eqref{eqn:halfmultipoint} is well defined. Recall \begin{equation}\label{eqn:Wformulaagain}
    \begin{split}
        &W= I-W_0\overline{1}^tW_{1,2}\overline{1}^tW_{2,3}\cdots\overline{1}^t W_{n-1,n}\overline{1}^tW_{n+1}.\\
&W_{i,i+1} = \mtp^{-u_i}\ptp^{-d_i}, \qquad W_0 = \mtp^{u},\qquad W_{n+1} = \ptp^{d}.
    \end{split}
\end{equation}
We emphasize the following two simple facts.
\begin{enumerate}[label=\alph*.]

    \item There are the same numbers of $\ptp$ and $\ptp^{-1}$, and of $\mtp$ and $\mtp^{-1}$, in $W$, and these factors commute with one another. Thus, if the projections $\overline{1}^t$ are removed, $W =0$.
    \item All differential operators ($\mtp,\ptp$ with positive powers) occur in $W_0$ and $W_{n+1}$. $W_{i,i+1}$ only has integral operators ($\mtp,\ptp$ with negative powers). 
\end{enumerate}

We introduce notation for subproducts of $W$ for later use. Let
\begin{equation}\label{def:Wmodifydef}
    \begin{split}
        &\W_{i,j} := W_{i,i+1}\overline{1}^tW_{i+1,i+2}\cdots \W_{j-2,j-1} \overline{1}^t W_{j-1,j},\quad 1<i+1<j<n+1;\\
        &\W_{0,i} := \W_{0}\overline{1}^t\W_{1,2}\cdots \W_{i-1,i},\quad 1< i< n+1;\\
        &W_{j,n+1} := \W_{j,j+1}\overline{1}^t\cdots \W_{n-1,n}\overline{1}^t\W_{n+1},\quad 0< j<n.
    \end{split}
\end{equation}
$V_{0,i},V_{j,n+1}$ are defined similarly with $W_0, W_{n+1}$ replaced with $V_0,V_{n+1}$ (defined in \eqref{eqn:V0Vndef}), respectively. Also, one needs the following boundary conventions:
\[W_{0,1} := W_0, \quad W_{n,n+1} := W_{n+1},\quad W_{1,1}:=I,\]
with the analogous conventions for \(V\). This avoids separate treatment of the case \(n=1\).
We record the following elementary support identities for \(\mtp^{-1}\) and
\(\ptp^{-1}\).
\begin{lemma}\label{lem:fullintegraladomain}
    \normalfont
    For $n\in \mathbb{Z}_{>0}$, the following identities hold as bounded operators from $L^2(\mathbb{R})$ to $L^2(\mathbb{R})$:
        \[\overline{1}^t\mtp^{-n}\overline{1}^t = \mtp^{-n}\overline{1}^t, \quad \overline{1}^t\ptp^{-n}\overline{1}^t  = \overline{1}^t\ptp^{-n}.\]
      
    Moreover, as identities of distributional kernels,
\begin{equation*}
        \begin{split}
        &\overline{1}^t \mtp^{-n}|\delta_t\rangle=\mtp^{-n}|\delta_t\rangle,\qquad
    \langle\delta_t|\ptp^{-n}\overline{1}^t=\langle\delta_t|\ptp^{-n},\\
    &1_t \mtp^{-n}|\delta_t\rangle=0,\qquad
    \langle\delta_t|\ptp^{-n}1_t=0.
\end{split}
\end{equation*}
\end{lemma}
\begin{proof}
    \sloppy From the integral kernel representation for $\mtp^{-1}$ in \eqref{def:aastarinversedef},
    $\mtp^{-1}(x,y) = \frac{1}{2}e^{(x-y)/2}1_{x< y}$. Thus, for each $n\in \mathbb{Z}^+$, $\mtp^{-n}$  is also an integral operator and its kernel $\mtp^{-n}(x,y)$ is supported on $x\leq y$, i.e., $\mtp^{-n}(x,y) =\frac{(y-x)^{n-1}}{2^n(n-1)!}e^{(x-y)/2}\cdot 1_{x< y}$. Similarly, $\ptp^{-n}(x,y)$ is supported on $x\geq y$. Thus, all identities follow from the support of the integral kernels.
\end{proof}
The differential operators \(\mtp=1-2D\) and \(\ptp=1+2D\) have distributional kernels
\[
    \mtp(x,y)=\delta(x-y)-2\partial_x\delta(x-y),\qquad
    \ptp(x,y)=\delta(x-y)+2\partial_x\delta(x-y),
\]
which are supported on the diagonal \(\{x=y\}\). This property allows us to define the operator $\Bra{\delta_t}\mtp$ on functions with a cutoff. 

Let
\begin{equation}\label{eqn:spacedef}
    C_{>0}^{\infty}:=\{f:\mathbb{R}\to\mathbb{R}: f|_{(0,\infty)}\in C^{\infty}(0,\infty)\cap L^2(0,\infty)\},
\end{equation}
and let \((C_{>0}^{\infty})'\) denote the
space of all linear functionals on \(C_{>0}^{\infty}\).
\begin{lemma}\label{lem:fulldifferentialdomain}
For $t>0, n\in \mathbb{Z}_{>0}$, $\normalfont\Bra{\delta_t}\mtp^{n},\Bra{\delta_t}\ptp^{n}$ are well defined on $C_{>0}^\infty$ as distributions:
\[\normalfont\Bra{\delta_t}\mtp^{n}f = (\mtp^{n}f|_{(0,\infty)})(t),\quad \normalfont\Bra{\delta_t}\ptp^{n}f = (\ptp^{n}f|_{(0,\infty)})(t).\]
\end{lemma}
The proof follows from the definitions.

Lemma~\ref{lem:fullintegraladomain} and Lemma~\ref{lem:fulldifferentialdomain} are the two key facts that we will use frequently. They are used for both 
\begin{itemize}
    \item dropping an indicator when it is inherited from the operator itself; and    \item reducing an operator to zero when the restriction contradicts the operator's support.
\end{itemize}

We now perform some equivalent transformations of the operator $W$. The manipulation of the operator involves switching the order of the differential operators in $W_0$ and $W_{n+1}$ with the indicator functions, which introduces commutators. The commutators for switching $\mtp,\ptp$ and $\overline{1}^t$ are
\begin{equation}\label{eqn:commutator}
    \begin{split}
        &[\mtp,\overline{1}^t] = 2\Ket{\delta_t}\Bra{\delta_t}, \quad [\overline{1}^t,\ptp]  = 2\Ket{\delta_t}\Bra{\delta_t},
    \end{split}
\end{equation}
which follow from a direct computation. We first give a concrete example of the transformation we will perform, which should clarify the notation and the general proof. Take $W$ in \eqref{eqn:wconcreteexample}, 
\begin{equation}\label{eqn:local44}
        \begin{split}
        W&=I-\mtp^3\overline{1}^{t}\mtp^{-1}\ptp^{-1}
\overline{1}^{t}\mtp^{-2}\ptp^{-1}
\overline{1}^{t}\ptp^2\\
&=I-\mtp^2\overline{1}^{t}\ptp^{-1}
\overline{1}^{t}\mtp^{-2}\ptp^{-1}
\overline{1}^{t}\ptp^2-2\mtp^2\Ket{\delta_t}\Bra{\delta_t}\mtp^{-1}\ptp^{-1}
\overline{1}^{t}\mtp^{-2}\ptp^{-1}
\overline{1}^{t}\ptp^2.
\end{split}
\end{equation}
The first term equals $\mtp^2\overline{1}^{t}\ptp^{-1}
\mtp^{-2}\ptp^{-1}
\overline{1}^{t}\ptp^2$ by Lemma~\ref{lem:fullintegraladomain}; the second term we denote by $C_{1,0}$. Then
\begin{equation}\label{eqn:local45}
        \begin{split}
        \eqref{eqn:local44} &=I-\mtp^2\overline{1}^{t}\ptp^{-1}
\mtp^{-2}\ptp^{-1}
\overline{1}^{t}\ptp^2-C_{1,0}\\
&=I-\mtp\overline{1}^{t}\ptp^{-1}
\mtp^{-1}\ptp^{-1}
\overline{1}^{t}\ptp^2-C_{1,0}-C_{2,0}\\
&=I-\overline{1}^{t}\ptp^{-1}
\ptp^{-1}
\overline{1}^{t}\ptp^2-C_{1,0}-C_{2,0}-C_{3,0}\\
&=1_t-C_{1,0}-C_{2,0}-C_{3,0}.
\end{split}
\end{equation}
The last equality is true since all the commutators that appear in switching \(\ptp^{2}\) and \(\overline{1}^t\) are zero by Lemma~\ref{lem:fullintegraladomain}.
$C_{i,0}$ are the commutator terms, whose definitions should be self-evident from the manipulation. For each \(C_{i,0}\), one can still switch $\ptp$ with indicators; thus, we do so in each such term. For example,
\begin{equation}
        \begin{split}
        C_{1,0}&=2\mtp^2\Ket{\delta_t}\Bra{\delta_t}\mtp^{-1}\ptp^{-1}
\overline{1}^{t}\mtp^{-2}\ptp^{-1}
\overline{1}^{t}\ptp^2\\
&=2\mtp^2\Ket{\delta_t}\Bra{\delta_t}\mtp^{-1}\ptp^{-1}
\overline{1}^{t}\mtp^{-2}
\overline{1}^{t}\ptp+4\mtp^2\Ket{\delta_t}\Bra{\delta_t}\mtp^{-1}\ptp^{-1}
\overline{1}^{t}\mtp^{-2}\ptp^{-1}
\Ket{\delta_t}\Bra{\delta_t}\ptp.
\end{split}
\end{equation}
We define $C_{1,1}$ to be the second term, and the first term equals \[2\mtp^2\Ket{\delta_t}\Bra{\delta_t}\mtp^{-1}\ptp^{-1}
\mtp^{-2}
\overline{1}^{t}\ptp\] by Lemma~\ref{lem:fullintegraladomain}. Thus,
\begin{equation}
        \begin{split}
        C_{1,0}&=2\mtp^2\Ket{\delta_t}\Bra{\delta_t}\mtp^{-1}\ptp^{-1}
\mtp^{-2}
\overline{1}^{t}\ptp+C_{1,1}\\
&=2\mtp^2\Ket{\delta_t}\Bra{\delta_t}\mtp^{-1}
\mtp^{-2}
\overline{1}^{t}+C_{1,1}+C_{1,2}\\
&= C_{1,1}+C_{1,2}.
\end{split}
\end{equation}
The last equality also follows from Lemma~\ref{lem:fullintegraladomain}. Performing similar calculations for all the terms in \eqref{eqn:local45}, one derives
$W = 1_t-\sum_{i=1}^u\sum_{j=1}^d C_{i,j}$.

Now for a general $W$, for $0<i\leq u,0<j\leq d$, we define $C_{i,j}$ to be the commutator that appears when $i$ factors of $\mtp$ and $j$ factors of $\ptp$ are moved across the indicators. Explicitly,
\begin{equation}\label{eqn:Cijdef}
        C_{i,j} = 4\mtp^{u-i}\ket{\delta_t}\bra{\delta_t}\mtp^{i-1}W_{1,n}\ptp^{j-1}\ket{\delta_t}\bra{\delta_t} \ptp^{d-j}.
\end{equation}
This is a slight abuse of notation in exchange for readability. It should be understood that $\mtp^{i-1},\ptp^{j-1}$ are moved across the indicators in sequential order, thus canceling the terms in $W_{1,n}$, and the resulting middle integral kernel is evaluated at $x=t,y=t$, yielding a constant. We define
\begin{equation}\label{eqn:cijdef}
         c_{i,j}= \bra{\delta_t}\mtp^{i-1}W_{1,n}\ptp^{j-1}\Ket{\delta_t}.
 \end{equation} 
 Notice that since all \(\mtp^{\pm 1},\ptp^{\pm 1}\) are translation invariant, \(c_{i,j}\) actually does not depend on \(t\). In this paper, we only need the information that $c_{i,j}$ is a constant independent of \(t\).
We have the following form of $W$.
\begin{proposition}\label{prop:Wequiv}
    \normalfont
    For $n\geq 1$, recall the definition of \(W\) in \eqref{def:Wdef}.
    \begin{equation}\label{eqn:Wequiv}
        \begin{split}
            W &= 1_t - 4\sum^u_{i=1}\sum_{j=1}^d c_{i,j}\mtp^{u-i}\ket{\delta_t}\bra{\delta_t} \ptp^{d-j}.
        \end{split}
    \end{equation}
    The two sides are equal as operators from $\mathcal{S}(\mathbb{R})$ to $\mathcal{S}'(\mathbb{R})$, where $\mathcal{S}(\mathbb{R})$ is the Schwartz space and $\mathcal{S}'(\mathbb{R})$ is its dual, the space of tempered distributions.

    Thus, $\W$ is a well-defined operator from $C_{>0}^{\infty}$ to $(C_{>0}^{\infty})'$, understood as the right-hand side of \eqref{eqn:Wequiv}, where $(C_{>0}^{\infty})'$ is the space of linear functionals on $C_{>0}^{\infty}$ defined in \eqref{eqn:spacedef}. The distributional supports of \(W(x,y)\) and \(V(x,y)\) are \(x\geq t/2, y\geq t/2\).
  \end{proposition}
  
    \begin{proof}
        We first move all factors of $\mtp$ across the first indicator function. Using \eqref{eqn:commutator}, we have 
        \begin{equation}\label{eqn:local1}
            \begin{split}
                \W =I-\sum_{i=1}^u C_{i,0}-\overline{1}^t\mtp^{u}\W_{1,n}\overline{1}^t\ptp^{d}=I-\sum_{i=1}^uC_{i,0}-\overline{1}^t\ptp^{-d}\overline{1}^t\ptp^{d}.
            \end{split}
        \end{equation}
        The second equality follows from repeatedly using $\mtp^{u}$ to cancel the $\mtp^{-u_i}$ in each $W_{i,i+1}$, and applying Lemma~\ref{lem:fullintegraladomain} to get rid of the indicator functions.
        
        Next, we move $\ptp^d$ to the left of $\overline{1}^t$:
        \begin{equation}
            \begin{split}
                \eqref{eqn:local1} &=  1_t - \sum^u_{i=1}\sum_{j=1}^d C_{i,j}-2\sum_{i=1}^u \mtp^{u-i}\ket{\delta_t}\bra{\delta_t}\mtp^{-u+i-1}\overline{1}^t-2\sum_{j=1}^d \overline{1}^t \ptp^{-d+j-1}\ket{\delta_t}\bra{\delta_t}\ptp^{d-j}.
            \end{split}
        \end{equation}
        The second term consists of the commutator terms from each $C_{i,0}$; the third term is what one gets when all the $\ptp^{d}$ factors are moved across $\overline{1}^t$ in $C_{i,0}$; the fourth term consists of the commutators from $\overline{1}^t\ptp^{-d}\overline{1}^t\ptp^{d}$. The third and fourth terms are $0$ by Lemma~\ref{lem:fullintegraladomain}.

         All the subsequent claims about \(W\) follow directly. The claim about \(V\) follows from its definition \eqref{def:Vdef}, Lemma~\ref{lem:fullintegraladomain} and Lemma~\ref{lem:fulldifferentialdomain}.

    \end{proof}
    The following lemma establishes the relations among \(H,H^{-1},S,S_H\).
    \begin{lemma}\label{lem:HHinverseonS}
    Recall the definitions of $H,H^{-1}$ in \eqref{def:HHinversedef}. For $\alpha>0, t>0$,
    \begin{equation}
            \begin{split}
                &1_tH^{-1} S_{H}^{n,m}1_0 = 1_t S^{n,m}1_0\text{ for any }n,m \in \mathbb{Z},\\
            &1_tHS^{n,m}1_0 = 1_tS_H^{n,m}1_0 \text{ for } n>m\geq 0.
            \end{split}
    \end{equation}
    These equalities are understood as identities of distributional kernels. The corresponding identities also
hold after replacing the right projection $1_0$ by
$\ket{\delta_0}$.
\end{lemma}
\begin{proof}
    Recall that $H^{-1}= D^{-1}\alphamtp\alphaptp$. Since we define $D^{-1}(x,y) = -1_{y> x}$, 
\[1_t H^{-1} S_{H}^{n,m}1_0=1_t D^{-1}1_t \alphamtp\alphaptp S_{H}^{n,m}1_0.\]
     By a slight modification of Lemma~\ref{lem:aactonS},
    \begin{equation}
            (1_t \alphamtp\alphaptp S_{H}^{n,m}1_0)(y,z)= \int_{\Gamma_\alpha}\frac{dw}{2\pi i}\frac{(-w)(1+2w)^m}{(1-2w)^n}e^{-(y+z) w},\qquad \text{ for } y>t,z>0.
    \end{equation}
    \(\Gamma_\alpha\) is defined in \eqref{eqn:contourchoice}. When \(\alpha \neq 1\), since the integrand does not have a pole at $w= (2\alpha-1)/2$, we can change the contour $\Gamma_\alpha$ to $C_{1/2}$. Then we can evaluate the $D^{-1}$ operator with Fubini's theorem. We get for $x>t,z>0$,
    \begin{equation}
             (1_tH^{-1} S_{H}^{n,m}1_0)(x,z) = -\int_x^\infty \,dy\int_{C_{1/2}}\frac{dw}{2\pi i}\frac{(-w)(1+2w)^m}{(1-2w)^n}e^{-(y+z) w} =s^{n,m}(x,z),
    \end{equation}
    which equals $1_t S^{n,m}1_0$.

    For the second equality, using the exact definition, for $x>t,z>0$, $(HS^{n,m})(x,z)$ equals
    \begin{equation}\label{eqn:local566}
           \begin{split}
               &\frac{1}{8}\int_{C_{1/2}} \frac{dw}{2\pi i} \bigg(\frac{e^{-(x+z)w}}{(2\alpha-1)/2+w}+\frac{e^{-(x+z)w}-e^{-(2\alpha-1)(x+z)/2}}{w-(2\alpha-1)/2}\bigg)\frac{(1+2w)^m}{(1-2w)^n} \\
               &=\int_{C_{1/2}}\frac{dw}{2\pi i}\frac{(-w)(1+2w)^me^{-(x+z)w}}{(1-2w)^n(2\alpha-1-2w)(2\alpha-1+2w)}\\
               &-\frac{e^{-(2\alpha-1)(x+z)/2}}{8}\int_{C_{1/2}}\frac{dw}{2\pi i} \frac{(1+2w)^m}{(1-2w)^n(w-(2\alpha-1)/2)}.
               \end{split}
    \end{equation}
    When $\alpha\neq 1$, the second term is $\frac{e^{-(2\alpha-1)(x+z)/2}}{8}\int_{C_{(2\alpha-1)/2}}\frac{dw}{2\pi i} \frac{(1+2w)^m}{(1-2w)^n(w-(2\alpha-1)/2)}$ because $n>m$ and hence there is no residue at $\infty$. Then by directly computing the residue, one can check that it equals
    \[\int_{(1-\frac{1}{2}1_{\alpha=1/2})C_{(2\alpha-1)/2}}\frac{dw}{2\pi i}\frac{(-w)(1+2w)^me^{-(x+z)w}}{(1-2w)^n(2\alpha-1-2w)(2\alpha-1+2w)}.\]
    Thus, together with the first term, we have proved $1_t HS^{n,m}1_0 =1_t S_H^{n,m}1_0$.

    When $\alpha=1$, the second term in \eqref{eqn:local566} is $0$ since $n>m$, and the first term is already $S^{n,m}_H(x,z)$. So, we also have the desired result.
\end{proof}
This is the reason we need to halve the residue at $(2\alpha-1)/2$ when $\alpha = 1/2$. 

  We are ready to prove Proposition~\ref{prop:welldefineandtraceclass}.
  
    \begin{proof}[Proof of Proposition~\ref{prop:welldefineandtraceclass}]
        In $\tilde{K}_{ij}$, we want to move $H^{-1}$ to make it act on $\mathcal{S}_i,\mathcal{S}_j^*$. Thus, the kernel has the form
\begin{equation}
        \tilde{K}_{ij} =1_0\mathcal{S}_i\begin{pmatrix}
           (H^{-1})^* & 1
        \end{pmatrix}
        \normalfont
\begin{pmatrix}
 -V^* &  V^*HV\\
0 &  V
\end{pmatrix}
\begin{pmatrix}
           1\\
            H^{-1}
        \end{pmatrix}\mathcal{S}_j^*J1_0.
\end{equation}
Multiply $H^{-1},(H^{-1})^*$ into $\mathcal{S}_i,\mathcal{S}_j$, we denote the middle matrix by $\mathcal{V}$. From Proposition~\ref{prop:Wequiv}, it should be understood that \(\mathcal{V} = 1_{t/2}\mathcal{V}1_{t/2}\). We have
\begin{equation}\label{eqn:local54}
        \begin{split}
            \tilde{K}_{ij} &= 1_0\begin{pmatrix}
            \S_H (H^{-1})^*&\S_H \\
            -\S (H^{-1})^*&-\S
        \end{pmatrix}1_{t/2}\mathcal{V}1_{t/2}\begin{pmatrix}
            \S &  \S_H\\
            H^{-1}\S &  H^{-1}\S_H
        \end{pmatrix}1_0.
        \end{split}
        \end{equation}
    Notice that the superscripts $l_i',r_i',l_j',r_j'$ are omitted since they are irrelevant for the current proof. Conjugating the kernel by $\mathcal{M}$ and applying Lemma~\ref{lem:HHinverseonS} gives 
    \begin{equation}\label{eqn:local59}
            (\mathcal{M} \tilde{K}\mathcal{M}^{-1})_{i,j} = 1_0 \begin{pmatrix}
            M_{-\theta_i}\S& M_{-\theta_i}\S_H \\
            -M_{\theta_i}\S (H^{-1})^*&-M_{\theta_i}\S
        \end{pmatrix}1_{t/2}\mathcal{V}1_{t/2}\begin{pmatrix}
        \S M_{\theta_j} &  \S_H M_{-\theta_j}\\
             H^{-1}\S M_{\theta_j}&  \S M_{-\theta_j}
        \end{pmatrix}1_0.
    \end{equation}
    For every fixed pair of integers $n,m$, there exist $C,N$ such that
        \begin{equation}\label{eqn:local60}
                \begin{split}
                     &|\S^{n,m}(x,y)|\leq 1_{x+y\geq 0}C(1+x+y)^N e^{-(x+y)/2}, \\
                     & |\S^{n,m}_H(x,y)|\leq 1_{x+y\geq 0}C(1+x+y)^N e^{-\min(2\alpha-1,1)(x+y)/2}.
                \end{split}
        \end{equation}
        The same first bound holds for the kernels obtained from $\S^{n,m}$ by applying $H^{-1},\allowbreak\mtp^{\pm 1},\allowbreak\ptp^{\pm 1}$ with the indicators on both sides, since these operations only multiply the residue integrand by rational functions and do not introduce a new pole inside the contour. Using \eqref{eqn:local60}, we have 
        \begin{equation}
                \begin{split}
                    &\lVert1_0 M_{\theta} \S M_{\eta}1_t \rVert_2 <\infty \text{ for any }\theta,\eta <1/2,\\
            &\lVert 1_tM_{-\eta} \S_H M_{-\theta}1_0  \rVert_2 <\infty \text{ for any } \theta,\eta >-\min(2\alpha-1,1)/2.
                \end{split}
        \end{equation}

        By Proposition~\ref{prop:Wequiv}, the term $V = \ptp^{-x_\star}W\ptp^{x_\star}$ has the form 
        \begin{equation}
                V = \ptp^{-x_\star}\bigg( 1_t - \sum^u_{i=1}\sum_{j=1}^d 4c_{i,j}\mtp^{u-i}\ket{\delta_t}\bra{\delta_t} \ptp^{d-j}\bigg)\ptp^{x_\star}.
        \end{equation}
        Thus there are two types of terms that come from composition with $V$ in $\mathcal{V}$:
        \begin{equation}\label{eqn:local62}
                \mathtt{S}1_t\mathtt{S},\qquad \mathtt{S}\ptp^{-x_\star}\mtp^{u-i}\ket{\delta_t}\bra{\delta_t} \ptp^{d-j}\ptp^{x_\star}\mathtt{S},
        \end{equation}
        where $\mathtt{S}$ can be any of $\S$, $\S_H$, or $H^{-1}\S$. We will use the following inequality for the trace norm:
        \[\lVert 1_0\mathtt{S}1_t\mathtt{S}1_0 \rVert_1\leq \lVert 1_0\mathtt{S}1_t \rVert_2 \lVert 1_t\mathtt{S}1_0 \rVert_2. \]
        
The problematic term is $\S_H$ since $-(2\alpha-1)/2\geq 0$ when $0<\alpha\leq 1/2$. Observe that $\S_H$ is always paired with $\S$ in \eqref{eqn:local59}. Thus, choosing $-\min(2\alpha-1,1)/2<\eta<1/2$ gives
        \begin{equation}
               \lVert M_{-\theta_i} \S 1_t \S_H M_{-\theta_j}  \rVert_1\leq \lVert M_{-\theta_i} \S M_{\eta}1_t \rVert_2 \lVert 1_t M_{-\eta}\S_H M_{-\theta_j}  \rVert_2 <\infty.
        \end{equation}
        The verification of all other terms is similar; thus we can see that the terms in \eqref{eqn:local62} have finite trace norm.

        We still need to discuss the terms of the form 
        \[1_0\mathtt{S}\mtp^{x_\star}W^*\mtp^{-x_\star}H\ptp^{-x_\star} W\ptp^{x_\star}\mathtt{S}1_0,\]
        which have three types of terms (up to taking adjoints):
        \begin{equation}\label{eqn:local63}
             \begin{split}
                 &1_0\mathtt{S}\mtp^{x_\star}1_t\mtp^{-x_\star}H\ptp^{-x_\star} 1_t\ptp^{x_\star}\mathtt{S}1_0, \qquad 1_0\mathtt{S}\mtp^{x_\star}1_t\mtp^{-x_\star}H\ptp^{-x_\star} \mtp^{u-i}\ket{\delta_t}\bra{\delta_t} \ptp^{d-j}\ptp^{x_\star}\mathtt{S}1_0\\
                 &1_0\mathtt{S}\mtp^{x_\star}\mtp^{d-j'}\ket{\delta_t}\bra{\delta_t} \ptp^{u-i'}\mtp^{-x_\star}H\ptp^{-x_\star} \mtp^{u-i}\ket{\delta_t}\bra{\delta_t} \ptp^{d-j}\ptp^{x_\star}\mathtt{S}1_0.
             \end{split}
        \end{equation}
        First notice that for fixed $x_\star$, there exist $C,N>0$ such that
        \begin{equation}
               |(\mtp^{-x_\star}H\ptp^{-x_\star})(x,y)|\leq C(1+x+y)^N e^{-\min(2\alpha-1,1)|x-y|/2}.
        \end{equation}
        Thus $\lVert1_t M_{-\eta}\mtp^{-x_\star}H\ptp^{-x_\star}M_{-\eta} 1_t\rVert_2<\infty$ for $1/2>\eta>\max(0,(1-2\alpha)/2)$.
        Observe that in \eqref{eqn:local59}, the $\mathtt{S}$ in \eqref{eqn:local63} can only be $H^{-1}\S$ or $\S$; thus the exponential part of $\mathtt{S}$ has enough decay; we can bound it in the same way:
        \begin{multline}
            \lVert 1_0M_{-\theta_i}\S\mtp^{x_\star}1_t\mtp^{-x_\star}H\ptp^{-x_\star} 1_t\ptp^{x_\star}\S M_{-\theta_j}1_0 \rVert_1\leq \\
            \lVert 1_0M_{-\theta_i}\S \mtp^{x_\star}M_{\eta}1_t\rVert_2 \lVert 1_tM_{-\eta}\mtp^{-x_\star}H\ptp^{-x_\star} M_{-\eta}1_t \rVert_2 \lVert 1_tM_{\eta}\ptp^{x_\star}\S M_{-\theta_j}1_0 \rVert_2 <\infty.
        \end{multline}
        For other choices of $\mathtt{S}$ the proof is the same.
         For the second and third terms in \eqref{eqn:local63}, notice that $\mtp,\ptp$ commute with $H$, and $x_\star\geq u-i$ for $i=1,\ldots,u$; thus all the operators are well defined. Both terms are rank-one operators; by checking the $L^2$ norm, one can see that they also have finite trace norm.

        Now we estimate the operators in $R_{ij}$. When $j<i$, the $11$ entry of $(\mathcal{M}R\mathcal{M}^{-1})_{ij}$ is
        \begin{equation}
                \begin{split}
                M_{-\theta_i}\mtp^{-d_{ji}'}(\ptp)^{-u_{ji}'}M_{\theta_j}.
        \end{split}
        \end{equation}
        Choosing $\eta$ such that $\theta_j<\eta<\theta_i$ and using Lemma~\ref{lem:fullintegraladomain} gives
        \begin{equation}\label{eqn:local6555}
                \lVert 1_0M_{-\theta_i}\mtp^{-d_{ji}'}1_0(\ptp)^{-u_{ji}'}M_{\theta_j}1_0 \rVert_1 \leq \lVert 1_0 M_{-\theta_i}\mtp^{-d_{ji}'}M_{\eta}1_0 \rVert_2 \lVert 1_0 M_{-\eta}(\ptp)^{-u_{ji}'}M_{\theta_j}1_0\rVert_2.
        \end{equation}
        The Hilbert--Schmidt norm can be bounded explicitly:
        \begin{equation}
                \begin{split}
                     \lVert 1_0 M_{-\eta}\mtp^{-d_{ji}'}M_{\theta_j}1_0\rVert_2^2\leq C\int_0^\infty dx\int_0^\infty dz z^{2d_{ji}'-2}e^{2(\theta_j-\eta)x}e^{-(1-2\eta)z}<\infty.
                \end{split}
        \end{equation}
        Thus, \eqref{eqn:local6555} is finite. The proof for the $22$ entry in  $(\mathcal{M}R\mathcal{M}^{-1})_{ij}$ when $i<j$ is exactly the same. 
        
        It remains to consider the $12$ entry of
$(\mathcal MR\mathcal M^{-1})_{ij}$:
\[
-1_0
M_{-\theta_i}
\mtp^{r_i'-l_j'}
H
\ptp^{r_j'-l_i'}
M_{-\theta_j}
1_0.
\]
Notice that
        \begin{equation}
                r_i'-l_j' = -(y_i+y_j+s_i-s_j)/2\leq 0,\quad r_j'-l'_i
                = -(y_i+y_j-s_i+s_j)/2\leq 0. 
        \end{equation}
Choose $\max\{0,(1-2\alpha)/2\}<\eta_i<\theta_i$ for all $i=1,\ldots,m$.
We factor the operator as
\[
\begin{split}
&1_0
M_{-\theta_i}
\mtp^{r_i'-l_j'}
H
\ptp^{r_j'-l_i'}
M_{-\theta_j}
1_0\\
&\qquad=
\bigl(
1_0M_{-\theta_i}\mtp^{r_i'-l_j'}M_{\eta_i}1_0
\bigr)
\bigl(1_0
M_{-\eta_i}HM_{-\eta_j}1_0
\bigr)
\bigl(1_0
M_{\eta_j}\ptp^{r_j'-l_i'}M_{-\theta_j}1_0
\bigr).
\end{split}
\]
The indicators are inserted due to Lemma~\ref{lem:fullintegraladomain}. The middle factor is Hilbert--Schmidt. Indeed, the explicit kernel of
$H$ gives
\[
\begin{split}
\left\lVert
1_0M_{-\eta_i}HM_{-\eta_j}1_0
\right\rVert_2^2
&\leq
C\int_0^\infty\int_0^\infty
e^{-2\eta_i x-2\eta_j y}
e^{(1-2\alpha)|x-y|}
\,dx\,dy <\infty.
\end{split}
\]

Using the explicit formulas for \(\mtp, \ptp\), one can see that the first and third factors are bounded on $L^2([0,\infty))$. 
Consequently,
\[
\begin{split}
&
\left\lVert
1_0
M_{-\theta_i}
\mtp^{r_i'-l_j'}
H
\ptp^{r_j'-l_i'}
M_{-\theta_j}
1_0
\right\rVert_2<\infty.
\end{split}
\]
\end{proof}
\begin{remark}
    When \(r_i'-l_j' = r'_j-l_i'=0\), the \(12\) entry is \(-1_0M_{-\theta_i}HM_{-\theta_j}1_0\), which is not trace class but only Hilbert--Schmidt.
 \end{remark} 

\subsection{Kolmogorov equation}\label{sec:kolm}
\subsubsection{Overview of the argument}\label{subsec:kolm}
In this section, we will prove
\begin{proposition}\label{prop:kolm}
    Let $\Pf(J+JK)$ be the object defined in \eqref{eqn:halfmultipoint}. Let $\mathcal{L}$ be the generator of the half-space TASEP with parameter $\alpha$. Then
    \begin{equation*}
        (\partial_t-\mathcal{L})\Pf(J+JK)=0.
    \end{equation*}
\end{proposition}

Recall that $\tilde{K}_{ij}$ in \eqref{eqn:ktilde} is given by 
\begin{equation}\label{eqn:ktildeforKol}
\normalfont
        \tilde{K}_{ij}= \mathcal{S}_i(V  H^{-1}-(V H^{-1})^*  +(V H^{-1})^*H V H^{-1})\mathcal{S}_j^*J,
\end{equation}
and 
$V$ is defined in \eqref{def:Vdef}.

Notice that the variable $t$ appears $n$ times in the indicator functions in $V$; we define $\partial_{t,i}V$ to be the expression obtained when the derivative hits the $i$-th indicator function. So, $\partial_t V = \sum_{i=1}^n\partial_{t,i}V$. 

Regarding the generator acting on the operator $
V$, we have
\begin{equation*}
    \mathcal{L}V = \sum_{i=1}^n \bigg(1+(\alpha-1) 1_{x_i=0}\bigg)(V^{\downarrow x_i}-V),
\end{equation*}
where $V^{\downarrow x_i}$ is the kernel parameterized by the configuration obtained from a flip at $x_i$ from $(x_1,h_1;\ldots;x_n,h_n)$. A flip means the one-step update in which a local maximum becomes a local minimum. We define
\begin{equation}\label{def:lxidef}
    \mathcal{L}_{x_i} V:=\begin{cases}
        V^{\downarrow x_i}-V \text{ if } x_i\neq0;\\
        \alpha(V^{\downarrow x_i}-V) \text{ if } x_i = 0.
    \end{cases}
\end{equation}
Thus, $\mathcal{L}V = \sum_{i=1}^n \mathcal{L}_{x_i}V$. Notice that we have $0\leq x_1<x_2 <\cdots<x_n$; thus only $x_1$ can possibly be zero.

To discuss the main proof, we state four lemmas that will be proved in the next section. These four lemmas are all algebraic facts about computing the action of the generator on the Pfaffian or the kernel.

\begin{lemma}\label{lem:halfkolondet}Let \(K\) be the kernel defined in \eqref{eqn:halfmultipoint} and assume \(I+K\) is invertible. Then
    \begin{equation*}
    (\partial_t-\mathcal{L})\Pf(J+JK) = \frac{1}{2}\Pf(J+JK)\tr((I+K)^{-1}(\partial_t-\mathcal{L})K).
\end{equation*}
\end{lemma}
A priori, one does not know whether \(I+K\) is invertible. However, this issue can be treated with the same method as in \cite[Lemma~4.11]{MQR24}. Thus, it does not affect the proof and we will not elaborate on this technicality.
\begin{lemma}\label{lem:halflleiibniz}
    For $x_i\geq 0$,        \begin{equation}\label{eqn:lleiibniz}
            \mathcal{L}_{x_i} (V^*H V) = (\mathcal{L}_{x_i}V^*)H V+V^*H  (\mathcal{L}_{x_i}V).
        \end{equation}
\end{lemma}
\begin{lemma}\label{lem:halfkolonV}
     
    \begin{equation}\label{eqn:wflip0term}
    \normalfont
       (\partial_{t,i}-\mathcal{L}_{x_i})V = \begin{cases}
       0,\quad x_i\neq 0;\\
    \alphamtp\ptp^{-1}V_0\ket{\delta_t}\bra{\delta_t} V_{1,n+1},\quad x_i=0.
   \end{cases}
   \end{equation}
 
\end{lemma}
\begin{lemma}\label{lem:halfVstarDVreduce}
For any $\alpha>0$, when $x_1=0$,
    \begin{multline}\label{eqn:halfVstarDVreduce}
        (H^{-1})^*((\partial_{t,1}-\mathcal{L}_{x_1})V^*)H V H^{-1}+(H^{-1})^*V^*H (\partial_{t,1}-\mathcal{L}_{x_1})V H^{-1}\\
         = (H^{-1})^*(\partial_{t,1}-\mathcal{L}_{x_1})V^*-(\partial_{t,1}-\mathcal{L}_{x_1})V H^{-1} .
    \end{multline}
\end{lemma}

Now, we are ready to discuss the proof of Proposition~\ref{prop:kolm}, with the help of these lemmas, which we will prove in the next section.
\begin{proof}[Proof of Proposition~\ref{prop:kolm}]  
By Lemma~\ref{lem:halfkolondet}, all we need to show is $(\partial_t-\mathcal{L})K=0$. In \eqref{eqn:ktildeforKol}, all dependence on $t$ and the initial configuration is in the middle bracket
\[(V  H^{-1}-(V H^{-1})^*  +(V H^{-1})^*H V H^{-1}).\]
The parameters in $\mathcal{S}_i, \mathcal{S}_j$ and $R_{ij}$ remain unchanged under the action of the generator. This is because $l_i',r_i'$ only depend on the initial configuration through $x_n+h_n$, which is unchanged under the action of the generator. $u_{ij}',d_{ij}'$ only depend on the final configuration and thus also remain unchanged. So, it boils down to show
\begin{equation}\label{eqn:local56}
    (\partial_t-\mathcal{L})(V  H^{-1}-(V H^{-1})^*  +(V H^{-1})^*H V H^{-1})=0.
\end{equation}
By Lemma~\ref{lem:halflleiibniz}, 
    \begin{equation*}
    \begin{split}
        \eqref{eqn:local56} &= (\partial_t-\mathcal{L})V  H^{-1}-(H^{-1})^*(\partial_t-\mathcal{L})V^*   \\
        &+(H^{-1})^*V^* H (\partial_t-\mathcal{L})V H^{-1}+(H^{-1})^*(\partial_t-\mathcal{L})V^* H V H^{-1}.
    \end{split}
    \end{equation*}
By Lemma~\ref{lem:halfkolonV}, all the terms are zero except those involving $\partial_{t,1}-\mathcal{L}_{x_1}$ when $x_1 = 0$. When \(x_1=0\), by Lemma~\ref{lem:halfVstarDVreduce}, 
\begin{equation}\label{eqn:local57}
\begin{split}
\eqref{eqn:local56} &= (\partial_t-\mathcal{L})V  H^{-1}-(H^{-1})^*(\partial_t-\mathcal{L})V^*   \\
        &-(\partial_t-\mathcal{L})V H^{-1}+(H^{-1})^*(\partial_t-\mathcal{L})V^*=0.
\end{split}
\end{equation}

\end{proof}
\subsubsection{Proofs of the lemmas}
\begin{proof}[Proof of Lemma~\ref{lem:halfkolonV}]
$\partial_{t,i}$ acting on $V$ changes the $i$-th indicator function in $V$ to $\ket{\delta_t}\bra{\delta_t}$. For $\mathcal{L}_{x_i}V$, recalling its definition in \eqref{def:lxidef}, we want to compute the difference between $V$ after a flip at $x_i$ and the original $V$. We consider several cases. We first assume that \(x_i\neq 0\).

 \begin{figure}[H]{\raggedright\noindent\textbf{Type 1:}\par}
\centering
\begin{tikzpicture}[scale=0.7]
\draw[thick, ->] (-3,-1) -- (3,-1) node[right] {$x$};
\draw[thick, ->] (-2,-2) -- (-2,2) node[above] {$y$};
 \node[shift={(0,0.2)}] at (0,0) {$(x_i,h_i)$};
 \node at (2.5,-2) {$\cdots$};
  \node at (-2.5,-2) {$\cdots$};
  \draw (-2,-2) -- (0,0) -- (2,-2);
  \draw[->, thick] (4,-1.5) -- node[above] {} (5,-1.5);
  \node at (0,-3) {Before flip};

  \begin{scope}[xshift=8cm]
    \draw[thick, ->] (-3,-1) -- (3,-1) node[right] {$x$};
     \node[shift={(-0.1,0.5)}] at (-1,-1) {\footnotesize$(x_i,h_i)$};
     \node[shift={(0.5,0.5)}] at (1,-1) {\footnotesize$(x_{i+1},h_{i+1})$};
\draw[thick, ->] (-2,-2) -- (-2,2) node[above] {$y$};
    \draw (-2,-2) -- (-1,-1) --(0,-2) -- (1,-1)--(2,-2);
    \draw[dashed] (-1,-1) -- (0,0) -- (1,-1);
     \node at (2.5,-2) {$\cdots$};
  \node at (-2.5,-2) {$\cdots$};
    \draw[->] (0,-0.4) -- (0,-1.6) node[midway, right] {};
     \node at (0,-3) {After flip};
  \end{scope}

\end{tikzpicture}
\end{figure}
Assume the configuration before the flip has a kernel $\V_{0,i}\overline{1}^t \V_{i,n+1}$; then the configuration after the flip at $x_i$ has a kernel $\V_{0,i}\mtp \overline{1}^t\mtp^{-1}\ptp^{-1}\overline{1}^t \ptp\V_{i,n+1}$. Taking the difference, we have the following.
\begin{equation*}
\begin{split}
&\V_{0,i}(\mtp\overline{1}^{t}\mtp^{-1}\ptp^{-1}\overline{1}^t\ptp-\overline{1}^t)\V_{i,n+1}\\
&=\V_{0,i}\bigg(2\overline{1}^t\ptp^{-1}\Ket{\delta_t}\Bra{\delta_t}+2\Ket{\delta_t}\Bra{\delta_t}\mtp^{-1}\overline{1}^t+4\Ket{\delta_t}\Bra{\delta_t}\mtp^{-1}\ptp^{-1}\Ket{\delta_t}\Bra{\delta_t}\bigg)\V_{i,n+1}.
\end{split}
\end{equation*}
The three terms in the bracket are the commutator terms for switching $\mtp, \ptp$ with the indicators.
By Lemma~\ref{lem:fullintegraladomain}, $\overline{1}^t\ptp^{-1}\Ket{\delta_t}=0,\Bra{\delta_t}\mtp^{-1}\overline{1}^t=0$, so the first two terms in the bracket are $0$. The third term is $\ket{\delta_t}\bra{\delta_t}$, since  $\Bra{\delta_t}\mtp^{-1}\ptp^{-1}\Ket{\delta_t}=\frac{1}{4}$ by an explicit kernel calculation. Thus $\mathcal{L}_{x_i}$ acting on $V$ also changes the $i$-th indicator function in $V$ to $\Ket{\delta_t}\Bra{\delta_t}$, giving $\partial_{t,i}V$. 

\begin{figure}[H]{\raggedright\noindent\textbf{Type 2:}\par}
\centering
\begin{tikzpicture}[scale=0.7]
 \draw (-2,-2) -- (0,0) -- (1,-1)--(2,0)--(3,1);
  \draw[->, thick] (4,-1.5) -- node[above] {} (5,-1.5);
   \node[shift={(0,0.2)}] at (0,0) {$(x_i,h_i)$};
  \draw[thick, ->] (-3,-1) -- (3,-1) node[right] {$x$};
\draw[thick, ->] (-2,-2) -- (-2,2) node[above] {$y$};
\node at (-2.5,-2.5) {$\cdots$};
\node at (3.5,1.25) {$\cdots$};
  \begin{scope}[xshift=8cm]
    \draw[thick, ->] (-3,-1) -- (3,-1) node[right] {$x$};
\draw[thick, ->] (-2,-2) -- (-2,2) node[above] {$y$};
 \node[shift={(0,0.5)}] at (-1,-1) {$(x_i,h_i)$};
    \draw (-2,-2) -- (-1,-1) --(0,-2) -- (1,-1)--(2,0)--(3,1);
    \draw[dashed] (-1,-1) -- (0,0) -- (1,-1);
    \node at (-2.5,-2.5) {$\cdots$};
\node at (3.5,1.25) {$\cdots$};
    \draw[->] (0,-0.4) -- (0,-1.6) node[midway, right] {}; 
  
  \end{scope}
\end{tikzpicture}
\end{figure}

This corresponds to $\V_{0,i}\overline{1}^t\V_{i,n+1} \to  \V_{0,i}\mtp\overline{1}^t\mtp^{-1}\V_{i,n+1}$.
We take the difference and get
\begin{equation}\label{eqn:42}
\begin{split}
\V_{0,i}\big(\mtp \overline{1}^t\mtp^{-1}-\overline{1}^t\big)\V_{i,n+1} = \V_{0,i}\big(2\Ket{\delta_t}\Bra{\delta_t}\mtp^{-1}\big)\V_{i,n+1}.
\end{split}
\end{equation}
Recall that \(V_{i,n+1} = W_{i,i+1}\overline{1}^t\cdots\overline{1}^tV_{n+1}\). All we need to show is the following fact: \begin{equation}\label{eqn:43}
    2\Ket{\delta_t}\Bra{\delta_t}\mtp^{-1}W_{i,i+1} \overline{1}^t=\Ket{\delta_t}\Bra{\delta_t}W_{i,i+1}\overline{1}^t .
\end{equation} Note that in the graph on the left, there is only one down-step from the peak at $x_i$ to the peak at $x_{i+1}$, which means $d_i=1$, so $
W_{i,i+1}\overline{1}^t = \ptp^{-1}\mtp^{-u_i}\overline{1}^t = \ptp^{-1}\overline{1}^t\mtp^{-u_i}\overline{1}^t$.  Then taking the difference of the terms in \eqref{eqn:43}, we have
\begin{equation}\label{eqn:45}
\begin{split} &2\Ket{\delta_t}\Bra{\delta_t}\mtp^{-1}\ptp^{-1}\overline{1}^t\mtp^{-u_i}\overline{1}^t-\Ket{\delta_t}\Bra{\delta_t}\ptp^{-1}\overline{1}^t \mtp^{-u_i}\overline{1}^t\\
&= \Ket{\delta_t}\Bra{\delta_t}(2\mtp^{-1}-1)\ptp^{-1}\overline{1}^t \mtp^{-u_i}\overline{1}^t \\
&= \Ket{\delta_t}\Bra{\delta_t}\mtp^{-1}\overline{1}^t \mtp^{-u_i}\overline{1}^t=0.
\end{split}  
\end{equation}
The third equality follows from Lemma~\ref{lem:fullintegraladomain}. Thus, \eqref{eqn:42} is $\V_{0,i}\Ket{\delta_t}\Bra{\delta_t}\V_{i,n+1}$, which is what we want.

\begin{figure}[H]{\raggedright\noindent\textbf{Type 3:}\par}
\centering
\begin{tikzpicture}[scale=0.7]
\draw[thick, ->] (-3,-1) -- (3,-1) node[right] {$x$};
\draw[thick, ->] (-2,-2) -- (-2,2) node[above] {$y$};
 \draw (-3,1) -- (-2,0) -- (-1,-1)-- (0,0) -- (2,-2);
  \node at (-3.5,1.25) {$\cdots$};
   \node at (2.5,-2.25) {$\cdots$};
  \node[shift={(0,0.2)}] at (0,0) {$(x_i,h_i)$};
  \draw[->, thick] (4,-1.5) -- node[above] {} (5,-1.5);
  \begin{scope}[xshift=8cm]
    \draw[thick, ->] (-3,-1) -- (3,-1) node[right] {$x$};
     \node[shift={(0,0.5)}] at (1,-1) {$(x_i,h_i)$};
\draw[thick, ->] (-2,-2) -- (-2,2) node[above] {$y$};
    \draw (-3,1) -- (-2,0) -- (-1,-1)-- (0,-2)--(1,-1) -- (2,-2);
    \draw[dashed] (-1,-1) -- (0,0) -- (1,-1);
      \node at (-3.5,1.25) {$\cdots$};
   \node at (2.5,-2.25) {$\cdots$};
    \draw[->] (0,-0.4) -- (0,-1.6) node[midway, right] {}; 
  
  \end{scope}
\end{tikzpicture}
\end{figure}
The proof for this type is very similar to that for Type 2 and is thus omitted.

\begin{figure}[H]{\raggedright\noindent\textbf{Type 4:}\par}
\centering
\begin{tikzpicture}[scale=0.7]
\draw[thick, ->] (-3,-1) -- (3,-1) node[right] {$x$};
\draw[thick, ->] (-2,-2) -- (-2,2) node[above] {$y$};
 \draw (-3,1) -- (-2,0) -- (-1,-1)-- (0,0) -- (1,-1)--(3,1);
  \node at (-3.5,1.25) {$\cdots$};
   \node at (3.5,1.25) {$\cdots$};
  \draw[->, thick] (4,-1.5) -- node[above] {} (5,-1.5);
  \node[shift={(0,0.2)}] at (0,0) {$(x_i,h_i)$};
  \begin{scope}[xshift=8cm]
    \draw[thick, ->] (-3,-1) -- (3,-1) node[right] {$x$};
\draw[thick, ->] (-2,-2) -- (-2,2) node[above] {$y$};
    \draw (-3,1)  -- (0,-2) --(3,1);
    \draw[dashed] (-1,-1) -- (0,0) -- (1,-1);
    \draw[->] (0,-0.4) -- (0,-1.6) node[midway, right] {}; 
    \node at (-3.5,1.25) {$\cdots$};
   \node at (3.5,1.25) {$\cdots$};
  \end{scope}
\end{tikzpicture}
\end{figure}
In this case,  $\V_{0,i}\overline{1}^t\V_{i,n+1}\to   \V_{0,i}\V_{i,n+1}$. The difference is
$\V_{0,i}1_t\V_{i,n+1}$. We will show that it is equal to $\V_{0,i}\Ket{\delta_t}\Bra{\delta_t}\V_{i,n+1}$. Since there is only one up-step on the left and one down-step on the right, we have $W_{i-1,i}=\mtp^{-1}\ptp^{-d_{i-1}}$ and $W_{i,i+1}=\ptp^{-1}\mtp^{-u_{i}}$. We can write 
\begin{equation*}
\overline{1}^t\W_{i-1,i} = \overline{1}^t \ptp^{-d_{i-1}}\overline{1}^t\mtp^{-1}, \quad \W_{i,i+1}\overline{1}^t=\ptp^{-1}\overline{1}^t\mtp^{-u_i}\overline{1}^t.
\end{equation*}
Then $\overline{1}^tW_{i-1,i}1_tW_{i,i+1}\overline{1}^t = \overline{1}^t \ptp^{-d_{i-1}}\overline{1}^t\mtp^{-1}1_t\ptp^{-1}\overline{1}^t\mtp^{-u_i}\overline{1}^t$. Using the integral kernel definitions of $\mtp^{-1},\ptp^{-1}$, we have
\begin{equation}\label{eqn:50}
\begin{split}
(\overline{1}^t\mtp^{-1}1_t\ptp^{-1}\overline{1}^t)(x,z) &= 1_{x\leq t,z\leq t}\int_t^\infty dy \tfrac{1}{4}1_{x < y}e^{\frac{1}{2}(x-y)} 1_{z < y}e^{\frac{1}{2}(z-y)}\\
&=1_{x\leq t,z\leq t}\tfrac{1}{4}e^{-t+\frac{1}{2}x+\frac{1}{2}z}.
\end{split}
\end{equation}
The indicator functions $1_{x\leq y}$ and $1_{z\leq y}$ in the integration can be dropped since we know $y>t$ and $x\leq t,z\leq t$. Now notice that \eqref{eqn:50} is the integral kernel of the rank-one operator $\mtp^{-1}\Ket{\delta_t}\Bra{\delta_t}\ptp^{-1}$, because $(\mtp^{-1}\Ket{\delta_t})(x) = 1_{x< t}\tfrac{1}{2}e^{-t/2+x/2}$ and $(\Bra{\delta_t}\ptp^{-1})(z)=1_{z< t}\tfrac{1}{2}e^{-t/2+z/2}$, which is what we want.

It remains to check the flips at $x_1$ and $x_n$. If $x_1\neq 0$, both cases follow by arguments very similar to those above, so we omit the details.

Now we investigate a flip at $x_1 = 0$. There are two types.

\begin{figure}[H]{\raggedright\noindent\textbf{Type 1:}\par}
\centering
\resizebox{0.8\textwidth}{!}{
\begin{tikzpicture}
\draw[thick, ->] (-3,-1) -- (3,-1) node[right] {$x$};
 \node[shift={(0.8,0)}] at (-2,1) {$(x_1,h_1)$};
\draw[thick, ->] (-2,-2) -- (-2,2) node[above] {$y$};
\node at (3,0.25) {$\cdots$};
 \draw (-2,1) -- (-1,0) -- (0,-1) --(1,-2)-- (2,-1)--(3,0);
  \draw[->, thick] (4,-1.5) -- node[above] {} (5,-1.5);
  \begin{scope}[xshift=8cm]
    \draw[thick, ->] (-3,-1) -- (3,-1) node[right] {$x$};
\draw[thick, ->] (-2,-2) -- (-2,2) node[above] {$y$};
 \node[shift={(0.8,0)}] at (-1,0) {$(x_1,h_1)$};
 \node at (3,0.25) {$\cdots$};
    \draw (-2,-1) -- (-1,0) -- (0,-1) --(1,-2)-- (2,-1)--(3,0);
    \draw[dashed] (-2,1) -- (-1,0);
    \draw[->] (-1.8,0.5) -- (-1.8,-0.2) node[midway, right] {}; 
  
  \end{scope}
\end{tikzpicture}}
\end{figure}
This corresponds to the kernel change: \[    \V_0\overline{1}^tV_{1,n+1}     \to  \ptp^{-1}\V_0\overline{1}^t \ptp V_{1,n+1}.\] Moving $\ptp$ across $\overline{1}^t$ and commuting it with $\V_0$ cancels the leading $\ptp^{-1}$, so their difference is the commutator term \[2\ptp^{-1}\V_0\ket{\delta_t}\bra{\delta_t}\V_{1,n+1}.\]
Since $V = I-\V_0\overline{1}^t\V_{1,n+1}$, this gives
$\mathcal{L}_{x_1}V = \alpha(V^{\downarrow x_1}-V)= -2\alpha\ptp^{-1}\V_0\ket{\delta_t}\bra{\delta_t} \V_{1,n+1}$,
while $\partial_{t,1}V = -\V_0\ket{\delta_t}\bra{\delta_t} \V_{1,n+1}$. Therefore
\[(\partial_{t,1}- \mathcal{L}_{x_1})V = (2\alpha\ptp^{-1}-I)\V_0\ket{\delta_t}\bra{\delta_t} \V_{1,n+1}=\alphamtp\ptp^{-1}\V_0\ket{\delta_t}\bra{\delta_t} \V_{1,n+1}.\]

\begin{figure}[H]{\raggedright\noindent\textbf{Type 2:}\par}
    \centering
    \resizebox{0.8\textwidth}{!}{
\begin{tikzpicture}
 \draw (-2,1) -- (-1,0) -- (0,1) --(1,2)-- (2,1);
  \draw[->, thick] (3,1.5) -- node[above] {} (4,1.5);
  \draw[thick, ->] (-3,-1) -- (2,-1) node[right] {$x$};
  \node at (2,0.75) {$\cdots$};
   \node[shift={(0.8,0)}] at (-2,1) {$(x_1,h_1)$};
\draw[thick, ->] (-2,-2) -- (-2,2) node[above] {$y$};
  \begin{scope}[xshift=8cm]
    \draw[thick, ->] (-3,-1) -- (2,-1) node[right] {$x$};
\draw[thick, ->] (-2,-2) -- (-2,2) node[above] {$y$};
\node[shift={(0.8,0)}] at (1,2) {$(x_1,h_1)$};
    \draw (-2,-1) -- (-1,0) -- (0,1) --(1,2)-- (2,1);
     \node at (2,0.75) {$\cdots$};
    \draw[dashed] (-2,1) -- (-1,0);
    \draw[->] (-1.8,0.5) -- (-1.8,-0.4) node[midway, right] {}; 
  
  \end{scope}
\end{tikzpicture}}
\end{figure}
In this case, we have the following kernel change:
\[\V_0\overline{1}^t\W_{1,2}\overline{1}^t \V_{2,n+1} \to \V_0W_{1,2}\overline{1}^t\V_{2,n+1}.\]
This kernel transformation is less transparent than the previous ones, so we
record how the parameters change. Recall $W_{1,2}=\ptp^{-1}\mtp^{-u_1}$ (there is a
single down-step since $d_1=1$), $\V_0=\ptp^{-x_\star}\mtp^{u}$ and
$\V_{n+1}=\ptp^{d}\ptp^{x_\star}$. After the flip at the origin the peak at $x_1=0$
disappears, the new first peak is $(x_2,h_2)=(u_1+1,\,h_1+u_1-1)$, and the remaining
peaks are unchanged. Substituting this into \eqref{def:primcoord} and the
definitions of $u,d,l_i,r_i$ gives
\[
u\to u-u_1,\quad d\to d-1,\quad x_\star\to x_\star+1,\quad h_\star\to h_\star-1,
\quad l_i\to l_i,\quad r_i\to r_i-1 ;
\]
geometrically, \emph{the primordial peak} moves one step down along the line of
slope $-1$ through the last peak (see Figure~\ref{figure:halfTASEP}), so $l_i$ and
$l_i',r_i'$, which depend only on the last peak, are unchanged. Consequently
$\V_{n+1}=\ptp^{d+x_\star}$ is unchanged, and the new boundary factor is
$\ptp^{-x_\star-1}\mtp^{u-u_1}$, which is
$\V_0W_{1,2}=\ptp^{-x_\star}\mtp^{u}\ptp^{-1}\mtp^{-u_1}$. This is exactly the new $\V_0$ term corresponding to the configuration in the right-hand diagram.

Now we compute \[\V_0W_{1,2}\overline{1}^t\V_{2,n+1}-\V_0\overline{1}^tW_{1,2}\overline{1}^t\V_{2,n+1}=\V_01_tW_{1,2}\overline{1}^t\V_{2,n+1}.\]
Notice that there is only one $\ptp^{-1}$ in $W_{1,2}$. By switching the order of $\ptp^{-1}$ in $W_{1,2}$ with $1_t$, we have
\begin{equation*}
    \begin{split}
\V_01_tW_{1,2}\overline{1}^t\V_{2,n+1}
&=2\ptp^{-1}\V_0\Ket{\delta_t}\bra{\delta_t}W_{1,2}\overline{1}^t\V_{2,n+1}+\V_0\ptp^{-1}1_t\mtp^{-u_1}\overline{1}^t \V_{2,n+1}.
    \end{split}
\end{equation*}
The second term is $0$ since $1_t \mtp^{-u_1}\overline{1}^t = 0$;
 thus we have the desired result.

 \end{proof}
\begin{proof}[Proof of Lemma~\ref{lem:halflleiibniz}]
        Let $c = 1+(\alpha-1)1_{x_i=0}$. By definition,\begin{equation}\label{eqn:ldefinition}
        \begin{split}
            &\mathcal{L}_{x_i} (\V^*H \V) =c\bigg(( \V^*H  \V)^{\downarrow x_i}-( \V^*H  \V)\bigg)\\
            &= c\bigg((\V^*)^{\downarrow x_i}H \V^{\downarrow x_i}- (\V^*)^{\downarrow x_i}H \V+ (\V^*)^{\downarrow x_i}H \V-\V^*H \V\bigg).
        \end{split}
    \end{equation}
    The last two terms are $ (\mathcal{L}_{x_i}\V^*) H  \V$. The first two terms are $ (\V^*)^{\downarrow x_i} H  (\mathcal{L}_{x_i}\V)$. So the difference between the right-hand side of equation~\eqref{eqn:lleiibniz} and equation~\eqref{eqn:ldefinition}  is
    \begin{equation}\label{eqn:58}
        \begin{split}
          & \V^* H  (\mathcal{L}_{x_i}\V)- (\V^*)^{\downarrow x_i} H  (\mathcal{L}_{x_i}\V) = -\frac{1}{c}(\mathcal{L}_{x_i}\V^*) H  (\mathcal{L}_{x_i}\V).
        \end{split}
    \end{equation} Recall from the proof of Lemma~\ref{lem:halfkolonV} that $\mathcal{L}_{x_i}\V = c(\V^{\downarrow x_i}-\V)$ is a rank-one operator for all $x_i$, so we write it as $\Ket{h_i}\Bra{g_i}$ for some auxiliary factors $h_i$ and $g_i$. Here $h_i$ is local to this proof and does not denote the peak height. Then we have $(\mathcal{L}_{x_i}\V^*) = (\mathcal{L}_{x_i}\V)^* = \Ket{g_i}\bra{h_i}$. For
\(x_i>0\), \(\ket{h_i} = -V_{0,i}\ket{\delta_t}\). For \(x_1=0\), we have \(  \ket{h_1}=-2\alpha\ptp^{-1}V_0\ket{\delta_t}\). In either case, \(\bra{h_i}H \Ket{h_i}\) is absolutely convergent and
\[\eqref{eqn:58}=-\frac{1}{c}\Ket{g_i}\bra{h_i}H \Ket{h_i}\bra{g_i} =0,\] since $H$ defined in \eqref{def:HHinversedef} is an antisymmetric operator. This shows that $\mathcal{L}_{x_i}$ also satisfies the Leibniz rule.
    \end{proof}

\begin{proof}[Proof of Lemma~\ref{lem:halfVstarDVreduce}] We start with the first term on the left-hand side of \eqref{eqn:halfVstarDVreduce}: 
\begin{equation}\label{eqn:local64}
         (H^{-1})^*(\partial_{t,1}-\mathcal{L}_{x_1})\V^*H \V H^{-1} .
 \end{equation} 
         By Lemma~\ref{lem:halfkolonV}, 
\begin{equation}\label{eqn:local65}
        \begin{split}
            \eqref{eqn:local64} &= (H^{-1})^*\V^*_{1,n+1}\ket{\delta_t}\bra{\delta_t} \V_0^*\alphaptp\mtp^{-1}H\V H^{-1}\\
            &=(H^{-1})^*\V^*_{1,n+1}\ket{\delta_t}\bra{\delta_t} \V^*_0\alphaptp\mtp^{-1}H (I-\V_0\overline{1}^t\V_{1,n+1})H^{-1}.
        \end{split}
\end{equation}
       When the first peak is at $0$, $x_\star = u$. So, \[\V_0^* = \mtp^{-u}\ptp^{u},\quad \V_0 = \ptp^{-u}\mtp^u.\]
   Substituting these expressions into equation~\eqref{eqn:local65}, we claim that the part containing $\V_0\overline{1}^t\V_{1,n+1}$ is $0$, because 
   \begin{equation*}
       \begin{split}
&(H^{-1})^*\V^*_{1,n+1}\ket{\delta_t}\bra{\delta_t} \V^*_0\alphaptp\mtp^{-1}H \V_0\overline{1}^t\V_{1,n+1}H^{-1} \\
&=(H^{-1})^*\V^*_{1,n+1}\ket{\delta_t}\bra{\delta_t}\mtp^{-u}\ptp^{u}\alphaptp\mtp^{-1}H\ptp^{-u}\mtp^u\overline{1}^t\V_{1,n+1}H^{-1}\\
&=(H^{-1})^*\V^*_{1,n+1}\ket{\delta_t}\bra{\delta_t}\alphamtp^{-1}\mtp^{-1}D \overline{1}^t\V_{1,n+1}H^{-1}.
       \end{split}
   \end{equation*}
   The second equality follows from Lemma~\ref{lem:bbstarcancelrelation}. By Lemma~\ref{lem:fullintegraladomain}, the integral operator $(\alphamtp^{-1}\mtp^{-1}D )(x,y)$ is supported on $y\geq x$; thus $\bra{\delta_t}\alphamtp^{-1}\mtp^{-1}D \overline{1}^t=0$. Notice that $\bra{\delta_t}\alphamtp^{-1}\mtp^{-1}D \overline{1}^t=0$ is also true when $0<\alpha \leq \tfrac{1}{2}$ for the same support reason.
   So,
   \begin{equation}\label{eqn:local66}
    \eqref{eqn:local65} = (H^{-1})^*\V^*_{1,n+1}\ket{\delta_t}\bra{\delta_t} \V^*_0\alphaptp\mtp^{-1}HH^{-1} =(H^{-1})^*\V^*_{1,n+1}\ket{\delta_t}\bra{\delta_t} \V^*_0\alphaptp\mtp^{-1} .
   \end{equation}
   The second equality is true since \(\alphaptp\mtp^{-1}HH^{-1}=\alphaptp\mtp^{-1}\). So, this is the same as $(H^{-1})^*(\partial_{t,1}-\mathcal{L}_{x_1})\V^* $ when $x_1=0$. Similarly, the second term on the left-hand side of \eqref{eqn:halfVstarDVreduce} is 
\begin{equation}\label{eqn:62}
    \begin{split}
         &(H^{-1})^*\V^* H \alphamtp\ptp^{-1}\V_0\ket{\delta_t}\bra{\delta_t}  \V_{1,n+1}H^{-1} \\
         &= (H^{-1})^*(I- \V^*_{1,n+1}\overline{1}^t\V^*_0) H \alphamtp\ptp^{-1}\V_0\ket{\delta_t}\bra{\delta_t}  \V_{1,n+1}H^{-1}\\
        &=-\alphamtp\ptp^{-1}\V_0\ket{\delta_t}\bra{\delta_t}  \V_{1,n+1}H^{-1},
    \end{split}
\end{equation}
which is $-(\partial_{t,1}-\mathcal{L}_{x_1})VH^{-1}$. The extra sign in the last equality comes from the fact that \((H^{-1})^*H = -I\) when acting on \(\alphamtp\ptp^{-1}\V_0\ket{\delta_t}\).

\end{proof}

\begin{proof}[Proof of Lemma~\ref{lem:halfkolondet}]
    By Lemma \eqref{lem:appendixmixeddifferentiation}, \begin{equation}\label{eqn:local68partialt}
\begin{split}
    \frac{d}{dt}\Pf(J+JK)
&=\frac12\Pf(J+JK)\tr
\bigl((I+K)^{-1}\partial_t K\bigr)\\
&=\sum_{k=1}^n \frac12\Pf(J+JK)\tr
\bigl((I+K)^{-1}\partial_{t,k} K\bigr).
\end{split}
\end{equation}
   
 Now we check what $K^{\downarrow x_k}-K$ is. Let $c = 1+(\alpha-1)1_{x_k=0}$. Recall that $K^{\downarrow x_k}-K = \frac{1}{c}\mathcal{L}_{x_k}K$. By Lemma~\ref{lem:halfkolonV}, $\frac{1}{c}\mathcal{L}_{x_k}\V$ is a rank-one operator for all $x_k\geq 0$. Write 
 \begin{equation*}
     \frac{1}{c}\mathcal{L}_{x_k}\V H^{-1} = \Ket{f_k}\bra{g_k},\quad \frac{1}{c}(H^{-1})^*\mathcal{L}_{x_k}\V^* = \Ket{g_k}\bra{f_k},
 \end{equation*}
 for some functions $f_k,g_k$. Recall that $\frac{1}{c}\mathcal{L}_{x_k}K_{ij}$ is \begin{equation}\label{eqn:localkol69}
        \begin{split}
            \begin{pmatrix}
        \mathcal{S}_i & \mathcal{S}_i
    \end{pmatrix}
\begin{pmatrix}
    -\Ket{g_k}\bra{f_k} &\Ket{g_k}\bra{f_k}HVH^{-1} +(H^{-1})^*V^*H\Ket{f_k}\bra{g_k}\\
    0 & \Ket{f_k}\bra{g_k}
\end{pmatrix}\begin{pmatrix}
            \mathcal{S}^*_jJ\\
            \mathcal{S}^*_jJ
        \end{pmatrix}.
        \end{split}
    \end{equation} 
Define
\begin{equation}
    \mathcal{G}_k^i =\begin{pmatrix}
        \mathcal{S}_i & \mathcal{S}_i
    \end{pmatrix} \begin{pmatrix}
        -\ket{g_k} & (H^{-1})^*V^*H\ket{f_k}\\
        0 & \ket{f_k}
    \end{pmatrix},\mathcal{F}_k^j= 
    \begin{pmatrix}
        \bra{f_k} & -\bra{f_k}HVH^{-1}\\
        0 & \bra{g_k}
    \end{pmatrix}\begin{pmatrix}
            \mathcal{S}^*_jJ\\
            \mathcal{S}^*_jJ
        \end{pmatrix}.
\end{equation}
Then 
\begin{equation}
    \eqref{eqn:localkol69} = \mathcal{G}_k^i\mathcal{F}_k^j.
\end{equation}
Define the following column vector and row vector:
\begin{equation*}
\mathcal{G}_k = \begin{pmatrix}
        \mathcal{G}^1_k\\
        \mathcal{G}^2_k\\
        \vdots\\
        \mathcal{G}^m_k
    \end{pmatrix}, \quad 
    \mathcal{F}_k = \begin{pmatrix}
        \mathcal{F}_k^1, \mathcal{F}_k^2,\ldots, \mathcal{F}_k^m
    \end{pmatrix}.
\end{equation*}
Then $\frac{1}{c}\mathcal{L}_{x_k}K$ is $\mathcal{G}_k\mathcal{F}_k$. Notice that $\mathcal{G}_k$ is an $m \times 2$
block matrix; thus, we write $\mathcal{G}_k = \begin{pmatrix}
     G^1_k & G^2_k
 \end{pmatrix}$, where the two entries denote the columns of $\mathcal{G}_k$; similarly, we write $\mathcal{F}_k=\begin{pmatrix}
     F^1_k\\ F^2_k
 \end{pmatrix}$. It is straightforward to check the following relations: 
 \begin{equation}\label{eqn:local72relation}
     (F_k^1)^* = -JG_k^2, \quad (F_k^2)^*= JG_k^1.
 \end{equation}
Using these definitions, we obtain
\begin{equation}\label{eqn:local74}
    \begin{split}
        \mathcal{L}_{x_k}\sqrt{\det(I+K)} &=c\bigg( \sqrt{\det(I+K^{\downarrow x_k})}-\sqrt{\det(I+K)}\bigg)\\
        &=c\bigg(\sqrt{\det(I+K+\mathcal{G}_k\mathcal{F}_k)}-\sqrt{\det(I+K)}\bigg).
    \end{split}
\end{equation}
By Lemma~\ref{lem:appendixfiniterankperturbation},
\begin{equation}\label{eqn:local73}
    \begin{split}
        \det(I+K+\mathcal{G}_k\mathcal{F}_k)&=\det(I+K)\det(I+(I+K)^{-1}\mathcal{G}_k\mathcal{F}_k)\\
        &=\det(I+K)\det(I+\mathcal{F}_k(I+K)^{-1}\mathcal{G}_k).
    \end{split}
\end{equation}
The second determinant is an ordinary $2\times 2$ determinant:
\begin{equation*}
    \det\begin{pmatrix}
        1+F_k^1(I+K)^{-1}G_k^1 & F_k^1(I+K)^{-1}G_k^2\\
        F_k^2(I+K)^{-1}G_k^1 & 1+F_k^2(I+K)^{-1}G_k^2
    \end{pmatrix}.
\end{equation*}
Using the fact that $((I+K)^{-1})^* = -J(I+K)^{-1}J$ and \eqref{eqn:local72relation}, we can see that both $F_k^1(I+K)^{-1}G_k^2$ and $F_k^2(I+K)^{-1}G_k^1$ are antisymmetric. Because both objects are scalars, they are $0$. Further, we can see that $(F_k^1(I+K)^{-1}G_k^1)^*=F_k^2(I+K)^{-1}G_k^2$, and since they are both scalars, the diagonal values are the same. 

So,
\begin{equation*}
    \begin{split}
        \eqref{eqn:local73} = \det(I+K) \bigg(1+\frac{1}{2}(F_k^1(I+K)^{-1}G_k^1+F_k^2(I+K)^{-1}G_k^2)\bigg)^2.
    \end{split}
\end{equation*}
So,
\begin{equation*}
    \begin{split}
        \eqref{eqn:local74} &= \frac{c}{2}\sqrt{\det(I+K)}(F_k^1(I+K)^{-1}G_k^1+F_k^2(I+K)^{-1}G_k^2)\\
        &=\frac{c}{2}\sqrt{\det(I+K)}\tr (I+K)^{-1}\mathcal{G}_k\mathcal{F}_k=\frac{1}{2}\sqrt{\det(I+K)}\tr(I+K)^{-1}\mathcal{L}_{x_k}K.
    \end{split}
\end{equation*}
The second equality is due to the fact that the trace of the operator $\left(\begin{smallmatrix}
    \ket{f_1}\bra{g_1} & M_{21}\\
    M_{12} & \ket{f_2}\bra{g_2}
\end{smallmatrix}\right)$ is equal to $\bra{g_1}\ket{f_1}+\bra{g_2}\ket{f_2}$. Summing over \(k\) gives the desired result.
\end{proof}

\subsection{Initial condition}\label{subsec:initialcondition}
\subsubsection{Kernel manipulation}\label{subsubsec:absorblem}
To verify the initial condition, we need to write the kernel in \eqref{eqn:ktilde} in a different form. In particular, we first want to compute
\begin{equation}
        \mathcal{S}_i\begin{pmatrix}
            \mtp^{x_\star} & \ptp^{-x_\star}
        \end{pmatrix} \text{ and } \begin{pmatrix}
            \mtp^{-x_\star} \\ \ptp^{x_\star}
        \end{pmatrix} \mathcal{S}^*_j J,
\end{equation}
where $\mathcal{S}_i = \begin{pmatrix}
    S_H^{l_i',r_i'}\\
    -S^{r_i',l_i'}
\end{pmatrix}$ is defined in \eqref{def:Smatrixdef}. Since the following discussion does not distinguish between $S_H$ and $S$, we define
\begin{equation}\label{def:mathttSdef}
        \mathtt{S} \text{ to be either } S_H\text{ or } S,\quad  \mathtt{s}\text{ to be either } s_H \text{ or } s.
\end{equation}
(Recall the definition in \eqref{def:Sijabdef}.)
 It is easy to verify that $\mtp, \ptp$ act on the operator $\mathtt{s}^{n,m}$  by changing the indices:
  \[\ptp \mathtt{s}^{n,m} = \mathtt{s}^{n-1,m},\quad \mtp \mathtt{s}^{n,m} = \mathtt{s}^{n,m+1}.\] For $\mathtt{S}^{n,m}$, one needs to track the terms when $\mtp,\ptp$ hit the indicator functions.
 Define $\Sdelta^{n,m}(x,y) := 2\mathtt{s}^{n,m}(x,y)\cdot \delta_0(x+y)$. For $n,m>0$,
   
    \begin{equation}\label{eqn:fullahitS}
        \normalfont\ptp\mathtt{S}^{n,m}=\mathtt{S}^{n,m}\mtp = \mathtt{S}^{n-1,m}+\Sdelta^{n,m}.
    \end{equation}
     $\Sdelta^{n,m}$ is the operator arising when the differential operator hits the indicator function. The following lemma illustrates when $\Sdelta^{n,m}$ will be $0$.

\begin{lemma}\label{lem:aactonS}
    For any $-t_1<t_2$ and any $n,m\in \mathbb{Z}$, the following equalities hold as identities of distributional kernels:
    \begin{equation*}
\normalfont
        \bra{\delta_{t_2}}\mtp\mathtt{S}^{n,m}1_{t_1} = \bra{\delta_{t_2}}\mathtt{S}^{n,m+1}1_{t_1} ,\quad \bra{\delta_{t_2}}\ptp\mathtt{S}^{n,m}1_{t_1} = \bra{\delta_{t_2}}\mathtt{S}^{n-1,m}1_{t_1} .
    \end{equation*}
\end{lemma}
\begin{proof}
   We prove the second equality; the first follows similarly. By \eqref{eqn:fullahitS},
   \begin{equation*}
     \bra{\delta_{t_2}}\ptp\mathtt{S}^{n,m}1_{t_1} =  \bra{\delta_{t_2}}(\mathtt{S}^{n-1,m}+\Sdelta^{n,m})1_{t_1} .
   \end{equation*}
   The term $\bra{\delta_{t_2}}\Sdelta^{n,m}1_{t_1}$ is $0$ since the Dirac delta function $\delta_0(t_2+x)=0$ for $x>t_1$.
\end{proof}
The next lemma is about how $\mtp^{-1}$ acts on $\mathtt{S}^{n,m}$.
\begin{lemma}\label{lem:ainverseactonS}
\normalfont For all $n,m\in \mathbb{Z}$,
$(\mtp^{-1}\mathtt{S}^{n,m})(x,z)1_{x\geq -z}=\mathtt{S}^{n,m-1}(x,z)$. 
\end{lemma}

\begin{proof}
When $\mathtt{S}^{n,m}=\S_H^{n,m}$, by definition,
\begin{equation}
\begin{split}
&(\mtp^{-1}\S_H^{n,m})(x,z)1_{x\geq -z}\\
&=
\int_x^\infty \frac12 e^{(x-y)/2}
\int_{\Gamma_\alpha}
\frac{(-w)(1+2w)^m e^{-(y+z)w}}
{(1-2w)^n(2\alpha-1+2w)(2\alpha-1-2w)}
\frac{dw}{2\pi i}\,dy\\
&=
\int_{\Gamma_\alpha}
\frac{(-w)(1+2w)^m}
{(1-2w)^n(2\alpha-1+2w)(2\alpha-1-2w)}
\left(
\int_x^\infty
\frac12 e^{(x-y)/2}e^{-(y+z)w}\,dy
\right)
\frac{dw}{2\pi i}\\
&=
\S_H^{n,m-1}(x,z).
\end{split}
\end{equation}
The second equality follows from Fubini's theorem. Since the contour does not enclose the pole at $w=-1/2$, the identity holds for any $m \in \mathbb{Z}$. 

For $\mathtt{S} = \S$, the proof is the same.
\end{proof}

Combining Lemma~\ref{lem:aactonS} and Lemma~\ref{lem:ainverseactonS}, we have the following fact:
\begin{lemma}\label{lem:aabsorbtoS}
    Let $k \in \mathbb{Z}_{>0},i,j,p,q \in \mathbb{Z}$. We have
    \begin{equation*}
      \begin{split}
            &\normalfont1_0 \mathtt{S}^{i,j}\ptp^{-k}WH^{-1}\ptp^{k}\mathtt{S}^{p,q}1_0 =  1_0 \mathtt{S}^{i,j-k}WH^{-1}\mathtt{S}^{p-k,q}1_0,\\
          &\normalfont1_0 \mathtt{S}^{i,j}\mtp^{k}WH^{-1}\mtp^{-k}\mathtt{S}^{p,q}1_0 =  1_0 \mathtt{S}^{i-k,j}WH^{-1}\mathtt{S}^{p,q-k}1_0.
      \end{split}
    \end{equation*}
\end{lemma}
\begin{proof}
    From Proposition~\ref{prop:Wequiv}, we see that the distributional kernel $W(x,y)$ is supported on $x\geq t/2$ and $y\geq t/2$. By the definition of $H^{-1}$ in \eqref{def:HHinversedef}, the distributional kernel $H^{-1}(x,y)$ is supported on $y\geq x$. Thus, $(WH^{-1})(x,y)$ is supported on $x\geq t/2$ and $y\geq t/2$. Then by Lemma~\ref{lem:aactonS} and Lemma~\ref{lem:ainverseactonS}, we have the desired results.
\end{proof}

Now using the definition of \(V\) in \eqref{def:Vdef}, we transform the kernel \(\tilde{K}_{ij}\) to the following form:
\begin{equation}
        \tilde{K}_{ij} =1_0\mathcal{S}_i\begin{pmatrix}
           \mtp^{x_\star}(H^{-1})^* & \ptp^{-x_\star}
        \end{pmatrix}
        \normalfont
\begin{pmatrix}
- W^* &  W^*\mtp^{-x_\star}H\ptp^{-x_\star} W\\
0 &  W 
\end{pmatrix}
\begin{pmatrix}
           \mtp^{-x_\star}\\
            H^{-1}\ptp^{x_\star}
        \end{pmatrix}\mathcal{S}_j^*J1_0.
\end{equation}
Denote the middle matrix by $\mathcal{W}$. By Lemma~\ref{lem:HHinverseonS} and Lemma~\ref{lem:aabsorbtoS}, 
\begin{equation}\label{eqn:ktildeinitial}
        \begin{split}
            \tilde{K}_{ij} &= 1_0\begin{pmatrix}
            \S^{l_i'-x_\star,r_i'}&\S_H^{l_i',r_i'-x_\star}\\
            -\S^{r_i'-x_\star,l_i'}(H^{-1})^* &-\S^{r_i',l_i'-x_\star}
        \end{pmatrix}\mathcal{W}\begin{pmatrix}
            \S^{r_j',l_j'-x_\star} &  \S_H^{l_j',r_j'-x_\star}\\
            H^{-1}\S^{r_j'-x_\star,l_j'} &  \S^{l_j'-x_\star,r_j'}
        \end{pmatrix}1_0\\
&= 1_0\begin{pmatrix}
            \S^{r_i,r_i'}&\S_H^{l_i',r_i'-l_i+r_i}\\
            -\S^{r_i'-l_i+r_i,l_i'}(H^{-1})^* &-\S^{r_i',r_i}
        \end{pmatrix}\mathcal{W}\begin{pmatrix}
            \S^{r_j',r_j} &  \S_H^{l_j',r_j'+r_j-l_j}\\
            H^{-1}\S^{r_j'-l_j+r_j,l_j'} &  \S^{r_j,r_j'}
        \end{pmatrix}1_0.
        \end{split}
\end{equation}
The second equality is true because $l_i = l_i'$ and $l_i-r_i = x_{\star}$ for all $i$.

\subsubsection{Overview of the argument}
In this subsection, we will prove that our formula has the correct initial condition. More precisely, we show
\begin{proposition}\label{prop:initialcondition}
Given the initial condition $(\vec{x},\vec{h})$, let $h(x)$ be the height function of the half-space TASEP associated with $(\vec{x},\vec{h})$. Let $F(t,h) = \Pf(J+JK)$ be defined in \eqref{eqn:halfmultipoint}. Then 
\begin{equation}\label{eqn:initialconditiontoprove}
    \lim_{t\to 0}F(t,h) = 1_{(\vec{x},\vec{h})\leq \{\vec{y},\vec{s}\}} = \prod_{i=1}^m 1_{h(y_i)\leq s_i}.
\end{equation}
\end{proposition}
We first state three lemmas, proved below, and then use them to prove this proposition.

The first lemma is about how the kernel reduces when there are ineffective peaks or troughs. 
\begin{figure}[!htbp]
    \centering
    \begin{minipage}{0.45\textwidth}
\centering
    \begin{tikzpicture}[scale=0.25]

\draw[very thin, gray, dotted] (-10,-2) grid (10,10);

\draw[thick, ->] (-10,-2) -- (10,-2) node[right] {$x$};
\draw[thick, ->] (-10,-4) -- (-10,10) node[above] {$y$};


\draw[thick] (-10,3) -- (-7,6) -- (-5,4) -- (-4,5)--(-1,2) --(1,4)-- (5,0);
\draw[thick,dotted] (-7,6) -- (-4,9) -- (1,4);

\draw[thick,blue] (-10,1) -- (-8,-1) -- (-6,1)--(-5,0)--(1,6)--(3,4)--(5,6);

\node[right, black,shift={(0.1,0)}] at (-4,9) {\tiny$(x_{\star},h_{\star})$};
\node[right, blue,shift={(0.1,0.1)}] at (3,4) {\tiny$(y_{m},s_{m})$};
\filldraw[blue] (3,4) circle (4pt);
\filldraw[black] (-4,9) circle (4pt);

\end{tikzpicture}
\end{minipage}
$\longrightarrow$
\begin{minipage}{0.45\textwidth}
\centering
     \begin{tikzpicture}[scale=0.25]

\draw[very thin, gray, dotted] (-10,-2) grid (10,10);

\draw[thick, ->] (-10,-2) -- (10,-2) node[right] {$x$};
\draw[thick, ->] (-10,-4) -- (-10,10) node[above] {$y$};


\draw[thick] (-10,3) -- (-7,6) -- (-5,4) -- (-4,5)--(-1,2) --(1,4)-- (5,0);
\draw[thick,dotted] (-7,6) -- (-4,9) -- (1,4);

\draw[thick,blue] (-10,1) -- (-8,-1) -- (-6,1)--(-5,0)--(1,6);

\node[right, black,shift={(0.1,0)}] at (-4,9) {\tiny$(x_{\star},h_{\star})$};
\node[right, blue,shift={(0.1,0.1)}] at (-5,0) {\tiny$(y_{m-1},s_{m-1})$};
\filldraw[blue] (-5,0) circle (4pt);
\filldraw[black] (-4,9) circle (4pt);

\end{tikzpicture}
\end{minipage}
    \caption{Trough reduction}
    \label{fig:troughreduction}
\end{figure}

\begin{figure}[!htbp]
    \centering
    \begin{minipage}{0.45\textwidth}
\centering
    \begin{tikzpicture}[scale=0.25]

\draw[very thin, gray, dotted] (-10,0) grid (10,10);

\draw[thick, ->] (-10,0) -- (10,0) node[right] {$x$};
\draw[thick, ->] (-10,-4) -- (-10,10) node[above] {$y$};


\draw[thick] (-10,5) -- (-7,8) -- (-5,6) -- (-4,7)--(-1,4) --(1,6)-- (5,2);

\draw[thick,blue] (-10,3) -- (-8,1) -- (-6,3)--(-5,2)--(0,7);
\draw[thick,blue,dotted]  (-8,1) -- (-7,0)--(-5,2);


\node[right, blue,shift={(0.1,0)}] at (-7,1) {\tiny$(y_{\star},s_{\star})$};
\node[right, black,shift={(0.1,0.1)}] at (1,6) {\tiny$(x_{n},h_{n})$};
\filldraw[blue] (-7,0) circle (4pt);
\filldraw[black] (1,6) circle (4pt);

\end{tikzpicture}
\end{minipage}
$\longrightarrow$
\begin{minipage}{0.45\textwidth}
\centering
     \begin{tikzpicture}[scale=0.25]

\draw[very thin, gray, dotted] (-10,0) grid (10,10);

\draw[thick, ->] (-10,0) -- (10,0) node[right] {$x$};
\draw[thick, ->] (-10,-4) -- (-10,10) node[above] {$y$};


\draw[thick] (-10,5) -- (-7,8) -- (-5,6) -- (-4,7)--(1,2) ;

\draw[thick,blue] (-10,3) -- (-8,1) -- (-6,3)--(-5,2)--(0,7);
\draw[thick,blue,dotted]  (-8,1) -- (-7,0)--(-5,2);


\node[right, blue,shift={(0.1,0)}] at (-7,1) {\tiny$(y_{\star},s_{\star})$};
\filldraw[blue] (-7,0) circle (4pt);
\node[right, black,shift={(0.1,0.1)}] at (-4,7) {\tiny$(x_{n-1},h_{n-1})$};
\filldraw[black] (-4,7) circle (4pt);

\end{tikzpicture}
\end{minipage}
    \caption{Peak reduction}
    \label{fig:peakreduction}
\end{figure}

\begin{lemma}\label{lem:kernelreduction}
    Let $K$ be the kernel defined in \eqref{eqn:halfmultipoint}. We denote $K$ by $K^{(\vec{x},\vec{h})}_{\{\vec{y},\vec{s}\}}$ to explicitly show its dependence on the initial and final configurations. For $n>1$, if the $n$-th peak is outside the cone $C_{y_{\star},s_{\star}}$, then $\det(I+K^{(\vec{x},\vec{h})}_{\{\vec{y},\vec{s}\}}) =\det(I+ K^{(x_1,h_1;\ldots; x_{n-1},h_{n-1})}_{\{\vec{y},\vec{s}\}})$. For $m>1$, if the $m$-th trough is outside the cone $C^{x_{\star},h_{\star}}$, then $\det(I+K^{(\vec{x},\vec{h})}_{\{\vec{y},\vec{s}\}}) = \det(I+K^{(\vec{x},\vec{h})}_{\{y_1,s_1;\ldots;y_{m-1},s_{m-1}\}})$. The reductions leave the indicator
$1_{(\vec{x},\vec{h})\leq\{\vec{y},\vec{s}\}}$ unchanged.
    
\end{lemma}
 Notice that any troughs outside the cone $C^{x_{\star},h_{\star}}$ do not contribute to the probability and any peaks outside the cone $C_{y_{\star},s_{\star}}$ also do not contribute, since $\mathbb{P}((\vec{x},\vec{h})_t\leq \{\vec{y},\vec{s}\})$ depends only on the exponential random variables in the box formed by $C^{x_{\star},h_{\star}}$ and $C_{y_{\star},s_{\star}}$. See Figures~\ref{fig:troughreduction} and~\ref{fig:peakreduction} for an illustration. We will prove the lemma in Section~\ref{subsubsec:kernelreduction}. 

The second lemma explicitly gives the eigenfunction when there is a trough in the cone  $C^{x_{\star},h_{\star}}$. 
\begin{lemma}\label{lem:eigenfunction}
    Suppose the configuration satisfies the following two conditions:
    \begin{enumerate}
        \item the last trough $\{y_k,s_k\}$ lies in the cone $C^{x_{\star},h_{\star}}$;
        \item either $n=1$, or the $n$-th peak $(x_n,h_n)$ lies in the cone $C_{y_{\star},s_{\star}}$.
    \end{enumerate}
     Let $F_k = \left(\begin{smallmatrix}0\\
    -1_0\S^{r_k',0}\ket{\delta_0}
  \end{smallmatrix}\right)$. Then
  \begin{equation}\label{eqn:eigenfunction}
      \lim_{t\to 0^+}K(k,\cdot;k,\cdot)F_k = -F_k, \quad \lim_{t\to 0^+} K(i,\cdot;k,\cdot)F_k = 0,
  \end{equation}
  for any $0<i<k$.
\end{lemma}
\begin{remark}\label{rmk:eigenfunctionhypotheses}
    We record the inequalities supplied by conditions (1) and (2) and indicate
    where the hypotheses enter the proof. Condition (1) gives
    \begin{equation*}
        r_k' = \tfrac{1}{2}(h_n+x_n-s_k-y_k)\geq 1,\qquad
        r_k = \tfrac{1}{2}(h_1-x_1-s_k+y_k)\geq 1,
    \end{equation*}
    so in particular $F_k\neq 0$ and $l_k=r_k+x_{\star}\geq 1$. Condition (2)
    gives $h_n-x_n>s_{\star}-y_{\star}=s_k-y_k$ when $n>1$, while $d=0$ when
    $n=1$; in both cases
    \begin{equation*}
        r_k-d = \tfrac{1}{2}\big((h_n-x_n)-(s_k-y_k)\big)\geq 1.
    \end{equation*}
    In the proof below, $r_k'\geq 1$ justifies the application of Lemma
    \ref{lem:ibp}; $r_k\geq 1$ and $l_k\geq 1$ justify the
     applications of Lemma~\ref{lem:astarinverseactonS} and Lemma
    \ref{lem:HHinverseonS}; and $r_k>d$ is precisely the hypothesis of Lemma
    \ref{lem:wtoindicator}.
\end{remark}
We observe that for fixed $x,y>0$ the kernel of $K$
converges pointwise as $t\to0^+$, and we \emph{define} $K_0$ as the integral
operator with the limiting kernel.

\begin{lemma}\label{lem:kernelcontinuityatzero}
    The limit above exists, $\mathcal{M}K_0\mathcal{M}^{-1}\in
    \mathfrak{S}_{1|2}$, and, as $t\to 0^+$,
    \begin{equation*}
        \big\lVert \mathcal{M}K\mathcal{M}^{-1}
        -\mathcal{M}K_0\mathcal{M}^{-1}\big\rVert_{1|2}\longrightarrow 0.
    \end{equation*}
    In particular, $\det(I+K)\to\det(I+K_0)$ by Lemma
    \ref{lem:appendixmixeddeterminantcontinuity}.
\end{lemma}

\begin{proof}
    The operators $R_{ij}$ do not depend on $t$, so it suffices to treat
    $\tilde{K}_{ij}$. Each entry of $\mathcal{M}\tilde{K}\mathcal{M}^{-1}$ is
    exhibited in the proof of Proposition
    \ref{prop:welldefineandtraceclass} as a product of two Hilbert--Schmidt
    factors. The kernels of these factors are, for $t>0$, finite sums of
    products and compositions of $\s$ functions whose arguments depend
    continuously on $t$. Moreover, they are
    dominated, uniformly for $t\in(0,1]$, by the bound provided in the proof of Proposition
    \ref{prop:welldefineandtraceclass}. By
    dominated convergence, each factor converges in Hilbert--Schmidt norm.     Since $\lVert A_tB_t-A_0B_0\rVert_1\leq
    \lVert A_t-A_0\rVert_2\lVert B_t\rVert_2
    +\lVert A_0\rVert_2\lVert B_t-B_0\rVert_2$, each product converges in
    trace norm, and hence $\mathcal{M}\tilde{K}\mathcal{M}^{-1}$ converges in
    $\lVert\cdot\rVert_{1|2}$. The limit is the integral operator with the
    pointwise limiting kernel, which therefore lies in $\mathfrak{S}_{1|2}$.
\end{proof}

In the remainder of the section, we give the proof of Proposition~\ref{prop:initialcondition}.

\begin{proof}[Proof of Proposition~\ref{prop:initialcondition}]

We first apply Lemma~\ref{lem:kernelreduction} until it cannot be applied. This gives reduced initial and final configurations; it should be understood that all the notation in the following discussion refers to the reduced configurations. 

Lemma~\ref{lem:kernelreduction} cannot be applied for one of the following two reasons:
\begin{enumerate}
    \item the rightmost peak is in the cone $C_{y_{\star},s_{\star}}$ and the rightmost trough is in the cone $C^{x_{\star},h_{\star}}$;
    \item there is only one peak and one trough left.
\end{enumerate}
Then one makes the observation that after all the reductions, the right-hand side of \eqref{eqn:initialconditiontoprove} is simply determined by the following criterion:
\[1_{(\vec{x},\vec{h})\leq \{\vec{y},\vec{s}\}} = \begin{cases}
    1 \text{ if $(x_{\star},h_{\star}) \notin C_{y_{\star},s_{\star}}$}\\
    0 \text{ if $(x_{\star},h_{\star}) \in C_{y_{\star},s_{\star}}$}
\end{cases}\]
Now all we need to show is that the left-hand side of \eqref{eqn:initialconditiontoprove} has the same value.

If $(x_{\star},h_{\star})$ is not in \(C_{y_{\star},s_{\star}}\), we must be in case \((2)\), in which the kernel becomes a $2\times 2$ matrix operator:
 \begin{equation}\label{eqn:local2by2kernel}
\normalfont
\begin{split}
\tilde{K}_{11} = \begin{pmatrix}
            \S^{r,r'}&\S_H^{l',r'-l+r}\\
            -\S^{r'-l+r,l'}(H^{-1})^* &-\S^{r',r}
        \end{pmatrix}\mathcal{W}\begin{pmatrix}
            \S^{r',r} &  \S_H^{l',r'+r-l}\\
            H^{-1}\S^{r'-l+r,l'} &  \S^{r,r'}
        \end{pmatrix}.
\end{split}
\end{equation}
If $(x_{\star},h_{\star})$ is not in $C_{y_{\star},s_{\star}}$, we have either $r \leq 0$  or $r'\leq 0$ in the kernel \eqref{eqn:local2by2kernel}. If $r \leq 0$, then $r+r'-l\leq 0$ since $r'\leq l$ is always true by definition. In that case, both $\S^{r,r'},\S^{r+r'-l,l}$ are $0$, so $\tilde{K}_{11}= 0$. If $r' \leq 0$, then $r+r'-l\leq 0$ since $r\leq l$ is always true by definition. In that case, both $\S^{r',r},\S^{r+r'-l,l}$ are $0$. Thus, only the \(12\) entry of \(\tilde{K}_{11}\) is nonzero. In both cases, the whole kernel in the Fredholm determinant is upper triangular with identities along the diagonal; thus, the determinant is $1$, which is what we want.

If $(x_{\star},h_{\star})$ is in $C_{y_{\star},s_{\star}}$, we want to show $\lim_{t\to 0^+}F(t,h)=0$. By Lemma
\ref{lem:kernelcontinuityatzero}, Lemma
\ref{lem:appendixmixeddeterminantcontinuity}, and Corollary~\ref{cor:appendixinitialconditionproperties}(3), it is enough to present an eigenfunction for the limiting kernel $K_0$ with eigenvalue $-1$. After the reductions, the hypotheses of Lemma~\ref{lem:eigenfunction} hold. Let
\begin{equation}\label{eqn:lastelementineigenfunction}
         \begin{split}
         F_k&:=\begin{pmatrix}
         0\\
         -1_0\S^{r_k',0}\ket{\delta_0}
 \end{pmatrix}.
 \end{split}
 \end{equation}
By Lemma~\ref{lem:eigenfunction},
$F =
     \begin{pmatrix}
         0& 0 &\cdots & 0 & F_k^{\mathsf T}
     \end{pmatrix}^{\mathsf T}$   
     is the eigenfunction with eigenvalue \(-1\). Thus, we have proved the desired result.

\end{proof}

In the following subsections, we prove Lemma~\ref{lem:kernelreduction} in Section~\ref{subsubsec:kernelreduction}; we prove Lemma~\ref{lem:eigenfunction} in Section~\ref{subsubsec:eigenfunction}. 

\subsubsection{Kernel reduction}\label{subsubsec:kernelreduction}
\begin{lemma}[Kernel reduction]\label{lem:absorbterm}
     Let $t>0$ and $i,j,r,l \in \mathbb{Z}$. Recall that $W_{n-1,n} = {\normalfont\mtp}^{-u_{n-1}}{\normalfont\ptp^{-d_{n-1}}}$. Define $W_{\mathrm{new}} := (I-{\normalfont\mtp}^{u-u_{n-1}}\overline{1}^t\W_{1,n-1}\overline{1}^t{\normalfont\ptp^{d-d_{n-1}}})$. This is the operator $W$ parameterized by the configuration $(x_1,h_1;\ldots;x_{n-1},h_{n-1})$. If $r \leq d$,
    \begin{equation*}
        \begin{split}
            1_{0}\mathtt{S}^{i,j}(I-{\normalfont\mtp}^u\overline{1}^t\W_{1,n}\overline{1}^t{\normalfont\ptp^d})H^{-1}\mathtt{S}^{r,l}1_{0}&=1_{0}\mathtt{S}^{i-u_{n-1},j}W_{\mathrm{new}}H^{-1}\mathtt{S}^{r,l-u_{n-1}}1_{0}.
        \end{split}
    \end{equation*}
    For the operator $\normalfont (H^{-1})^*W^*\mtp^{-x_\star}H\ptp^{-x_\star} WH^{-1}$, if \(i,r\leq d\),
    \begin{equation}\label{eqn:whwreduction}
            \begin{split}
                 &\normalfont1_{0}\mathtt{S}^{i,j}(H^{-1})^*W^*\mtp^{-x_\star}H\ptp^{-x_\star} WH^{-1}\mathtt{S}^{r,l}1_{0}\\
                 &=\normalfont1_{0}\mathtt{S}^{i,j-u_{n-1}}(H^{-1})^*W_{\mathrm{new}}^*\mtp^{-x_\star+u_{n-1}}H\ptp^{-x_\star+u_{n-1}} W_{\mathrm{new}}H^{-1}\mathtt{S}^{r,l-u_{n-1}}1_{0}.
            \end{split}
    \end{equation}

\end{lemma}
\begin{remark}[Graphical view]  $r \leq d$ means that the rightmost peak is outside the cone $C_{y_\star,s_\star}$; then the kernel is equivalent to the kernel with the rightmost peak removed.

\begin{figure}[!htbp]
    \centering
    \begin{minipage}{0.45\textwidth}
\centering
    \begin{tikzpicture}[scale=0.25]

\draw[very thin, gray, dotted] (-10,0) grid (10,10);

\draw[thick, ->] (-10,0) -- (10,0) node[right] {$x$};
\draw[thick, ->] (-10,-4) -- (-10,10) node[above] {$y$};


\draw[thick] (-10,5) -- (-7,8) -- (-5,6) -- (-4,7)--(-1,4) --(1,6)-- (5,2);
\draw[thick,black,dotted]  (-7,8) -- (-4,11)--(1,6);

\draw[thick,blue] (-10,3) -- (-8,1) -- (-6,3)--(-5,2)--(0,7);
\draw[thick,blue,dotted]  (-8,1) -- (-7,0)--(-5,2);
\draw[decorate,decoration={calligraphic brace, amplitude=7pt, raise=3pt}] 
    (-4,11) --(1,6) node [midway, sloped, above=10pt, xshift=-1pt] {\tiny$d=5$};  
\draw[decorate,decoration={calligraphic brace, amplitude=5pt, raise=2pt}] 
     (0,7)--(-4,11) node [midway, sloped, below=4pt, xshift=-1pt] {\tiny$r=4$}; 


\node[right, blue,shift={(0.1,0)}] at (-7,1) {\tiny$(y_{\star},s_{\star})$};
\node[right, black,shift={(0.1,0.1)}] at (1,6) {\tiny$(x_{n},h_{n})$};
\filldraw[blue] (-7,0) circle (4pt);
\filldraw[black] (1,6) circle (4pt);
\filldraw[black] (-4,11) circle (4pt);

\end{tikzpicture}
\end{minipage}
$\longrightarrow$
\begin{minipage}{0.45\textwidth}
\centering
     \begin{tikzpicture}[scale=0.25]

\draw[very thin, gray, dotted] (-10,0) grid (10,10);

\draw[thick, ->] (-10,0) -- (10,0) node[right] {$x$};
\draw[thick, ->] (-10,-4) -- (-10,10) node[above] {$y$};


\draw[thick] (-10,5) -- (-7,8) -- (-5,6) -- (-4,7)--(1,2) ;

\draw[thick,blue] (-10,3) -- (-8,1) -- (-6,3)--(-5,2)--(0,7);
\draw[thick,blue,dotted]  (-8,1) -- (-7,0)--(-5,2);


\node[right, blue,shift={(0.1,0)}] at (-7,1) {\tiny$(y_{\star},s_{\star})$};
\filldraw[blue] (-7,0) circle (4pt);
\node[right, black,shift={(0.1,0.1)}] at (-4,7) {\tiny$(x_{n-1},h_{n-1})$};
\filldraw[black] (-4,7) circle (4pt);

\end{tikzpicture}
\end{minipage}
    \caption{Peak reduction detail}
    \label{fig:peakreduction2}
\end{figure}

\end{remark}
\begin{proof}
We first discuss the case in which $\mathtt{S} = \S_H$. 
    \begin{equation}\label{eqn:local89}
\begin{split}
1_{0}\S^{i,j}(I-\mtp^u\overline{1}^t\W_{1,n}\overline{1}^t\ptp^d)H^{-1}\S_H^{r,l}1_{0}&=1_{0}\S^{i,j}(I-\mtp^u\overline{1}^t\W_{1,n}(I-1_t)\ptp^d)\S^{r,l}1_{0}\\
&=1_{0}\S^{i,j}(I-\mtp^u\overline{1}^t\W_{1,n}\ptp^d)\S^{r,l}1_{0}.
\end{split}
\end{equation}
The first equality follows from Lemma~\ref{lem:HHinverseonS}, which is applicable since \(W(x,y)\) is supported on \(x\geq t/2,y\geq t/2\). The second equality is true because
\begin{equation*}
W_{n-1,n}1_{t}\ptp^{d}\S^{r,l}1_0=\sum_{i=0}^{d-1}W_{n-1,n}1_{t}\ptp^{d-i-1}\Sdelta^{r-i,l}1_0+\W_{n-1,n}1_tS^{r-d,l}1_0=0.
\end{equation*}
The summation is $0$ since $\Sdelta^{r-i,l}(x,y)=0$ for $x>t,y>0$. The second term is $0$ since $S^{r-d,l}$ is $0$ when $r\leq d$.
This means that when $ r\leq d$, the indicator function after $W_{n-1,n}$ can be dropped and $\ptp^d$ can cancel the terms in $W_{n-1,n}$, giving $ W_{n-1,n}\ptp^{d} = \mtp^{-u_{n-1}}\ptp^{d-d_{n-1}}$. So, 
\begin{equation}\label{eqn:local90}
    \eqref{eqn:local89} =1_0\S^{i,j}(I-\mtp^{u}\overline{1}^t\W_{1,n-1}\overline{1}^t\ptp^{d-d_{n-1}}\mtp^{-u_{n-1}})\S^{r,l}1_0.
\end{equation}
Then applying Lemma~\ref{lem:aabsorbtoS} to bring $\mtp^{u_{n-1}}$ to the first $S$ and $\mtp^{-u_{n-1}}$ to the second $S$, we get the desired result.

If $\mathtt{S} = \S$, one can easily see that the same proof still works since $H^{-1}$ commutes with $\mtp$ and $\ptp$, and $H^{-1}(x,y)$ is supported on $x\leq y$.

For the term $\mathtt{S}^{i,j}(H^{-1})^*W^*\mtp^{-x_\star}H\ptp^{-x_\star} WH^{-1}\mathtt{S}^{r,l}$, if \(i,r\leq d\), notice that equation~\eqref{eqn:local89} is applicable from both sides; thus it equals
\begin{equation}
       1_{0}\mathtt{S}^{i,j}(H^{-1})^*\ptp^{-u_{n-1}}W_{\mathrm{new}}^*\ptp^{u_{n-1}}\mtp^{-x_\star}H\ptp^{-x_\star} \mtp^{u_{n-1}}W_\mathrm{new}\mtp^{-u_{n-1}}H^{-1}\mathtt{S}^{r,l}1_{0}. 
\end{equation}
By Lemma~\ref{lem:aabsorbtoS} and the facts that $H$ commutes with $\mtp,\ptp$ and that \(x_\star \geq u\), we derive the desired result.

\end{proof}
Now we can prove the reduction lemma, Lemma~\ref{lem:kernelreduction}.
\begin{proof}[Proof of Lemma~\ref{lem:kernelreduction}] 
We first prove the reduction for peaks.
Recall the kernel $\tilde{K}_{ij}$ defined in \eqref{eqn:ktildeinitial}. We want to show that if the $n$-th peak is not in the cone $C_{y_{\star},s_{\star}}$, \[\det(I+K^{(\vec{
x},\vec{h})}_{\{\vec{y},\vec{s}\}})=\det(I+K^{(x_1,h_1;\ldots;x_{n-1},h_{n-1})}_{\{\vec{y},\vec{s}\}}).\]
We first consider the case in which $x_n> y_{\star}$. We list how the parameters change when the last peak is deleted from the configuration:
\begin{equation}\label{eqn:local87}
        \begin{split}
        &r_i \to r_i;\quad l_i \to l_i-u_{n-1}; \\
        & r_i' \to r_i'-u_{n-1};\quad l_i' \to l_i'-u_{n-1};\\
        & W^{(\vec{x},\vec{h})} \to W^{(x_1,h_1;\ldots;x_{n-1},h_{n-1})}.
\end{split}
\end{equation}
This can be easily seen by tracing the definitions; see Figure~\ref{fig:peakreduction}. Now, \eqref{eqn:local87} is what we want to prove.

If the $n$-th peak is not in the cone $C_{y_{\star},s_{\star}}$ and $x_n>y_{\star}$, that means $d\geq r_j$ for all $j$. In formula \eqref{eqn:ktildeinitial}, observe that $W$ acts on terms $ H^{-1}\S^{r_j'-l_j+r_j,l_j'} $ and $\S^{r_j,r_j'}$ from the left (the second row in the last matrix in \eqref{eqn:ktildeinitial}), and acts on terms $\S_H^{l_i',r_i'-l_i+r_i}$ and $-\S^{r_i',r_i}$ from the right (the second column in the first matrix in \eqref{eqn:ktildeinitial}). Since $r_i'\leq l_i$, $r_i+r_i'-l_i \leq d$. Thus, we can apply Lemma~\ref{lem:absorbterm} to both terms. So,
\begin{equation}
        \begin{split}
        &H^{-1}\S^{r_j'-l_j+r_j,l_j'} \to H^{-1}\S^{r_j'-l_j+r_j,l_j'-u_{n-1}},\quad  \S^{r_j,r_j'} \to \S^{r_j,r_j'-u_{n-1}}\\
        &\S_H^{l_i',r_i'-l_i+r_i} \to \S_H^{l_i'-u_{n-1},r_i'-l_i+r_i},\quad -\S^{r_i',r_i}\to -\S^{r_i'-u_{n-1},r_i}.
\end{split}
\end{equation}
This is exactly the change we want in \eqref{eqn:local87}. 

For $(H^{-1})^*W^*$, the discussion is exactly the same: it acts on the first row in the last matrix in \eqref{eqn:ktildeinitial} and acts on the first column in the first matrix in \eqref{eqn:ktildeinitial}. Using Lemma~\ref{lem:absorbterm}, we derive the desired change in the kernel.

For $(H^{-1})^*W^*HWH^{-1}$, it acts on the second row in the last matrix in \eqref{eqn:ktildeinitial} and acts on the first column in the first matrix in \eqref{eqn:ktildeinitial}. The condition of Lemma~\ref{lem:absorbterm} is satisfied, so the lemma is applicable, and the kernel after the transformation is what we want.

Now we consider the other case in which $x_n\leq y_{\star}$. Since $(x_n,h_n)$ is the rightmost peak, we have $x_i\leq y_{\star}$ for all $i=1,\ldots,n$ and they are all outside the cone. This implies that $r'_i = (-s_i-y_i+x_n+h_n)/2\leq 0$ for all $i$.  By definition, $\S^{r_i',r_i} = 0,\S^{r_i+r_i'-l_i,l_i}=0$ for all $i$, so the second row of the first matrix in \eqref{eqn:ktildeinitial} and the first column of the third matrix in \eqref{eqn:ktildeinitial} are $0$. Thus, only the $12$ entry of $\tilde{K}_{ij}$ is nonzero. Combining this with $R_{ij}$, we can conclude that the kernel $K$ has the form
\[K(j,\cdot;i,\cdot) = \begin{pmatrix}
    \ast & \ast\\
    0 & 0
\end{pmatrix} \text{ and } K(i,\cdot;j,\cdot) = \begin{pmatrix}
    0 & \ast\\
    0 & \ast
\end{pmatrix},\]
 for all $i< j$. This implies that $\det(I+K^{(\vec{x},\vec{h})}_{\{\vec{y},\vec{s}\}})=1$. Now one can perform the same analysis on $\det(I+K^{(x_1,h_1;\ldots;x_{n-1},h_{n-1})}_{\{\vec{y},\vec{s}\}})$, which is also $1$.

Now we discuss the reduction for troughs. If the last trough $\{y_m,s_m\}$ is outside the cone $C^{x_{\star},h_{\star}}$, there are two cases: either $y_m \leq x_{\star}$ or $y_m> x_{\star}$. 

Consider first the case $y_m > x_{\star}$. We also have $s_m\geq h_{\star}+x_{\star}-y_m$. In this case, $r'_{m} = (h_{n}-s_m+x_{n}-y_m)/2\leq 0$ and $r_m+r_m'-l_m \leq 0$.   By definition, $\S^{r_m',r_m} = 0,\S^{r_m+r_m'-l_m,l_m}=0$. So, for the kernel $\tilde{K}_{mj}$, the second row in the first matrix in \eqref{eqn:ktildeinitial} is $0$; thus $\tilde{K}_{mj}$ is in the form $ \begin{pmatrix}
    \ast & \ast\\
    0 & 0
\end{pmatrix}$. By definition, $R_{mj}$ is also in this form; thus $K_{mj} = \begin{pmatrix}
    \ast & \ast\\
    0 & 0
\end{pmatrix}$. Similarly, $K_{jm} = \begin{pmatrix}
    0 & \ast\\
    0 & \ast
\end{pmatrix}$ for all $1\leq j \leq m$, so the determinant reduces:
\begin{equation}
    \det((I+K)^m_{i=1,j=1}) = \det((I+K)^{m-1}_{i=1,j=1}).
\end{equation}

In the other case, $y_m \leq x_{\star}$. Since $\{y_m,s_m\}$ is the last trough, we have $y_i\leq x_{\star}$ for all $i=1,\ldots, m$, which implies that $r_i=(h_1-x_1-s_i+y_i)/2\leq 0$ and $r_i+r_i'-l_i\leq 0$ for all $i$. By definition, $\S^{r_i,r'_i} = 0,\S^{r_i+r_i'-l_i,l_i}=0$. Thus, $\tilde{K}_{mj} =0$ and $\tilde{K}_{jm} =0$. The identities $K_{mj} = \begin{pmatrix}
    \ast & \ast\\
    0 & 0
\end{pmatrix}$ and $K_{jm} = \begin{pmatrix}
    0 & \ast\\
    0 & \ast
\end{pmatrix}$ hold for all $1\leq j \leq m$ as well, and we have the trough reduction.
\end{proof}

\subsubsection{Eigenfunction}\label{subsubsec:eigenfunction}
The goal of this section is to prove Lemma~\ref{lem:eigenfunction}. We start with a property of the operator $\S^{n,m}$.

\begin{lemma}\label{lem:halfS0lemma}
Recall the definition of $s^{n,m}$ in \eqref{def:Sijabdef}.
\begin{equation*} 
    \begin{split}
        &\s^{n,m}(x,-x) = 0 \text{ if } n-2\geq m\geq 0, \quad \s^{n,n-1}(x,-x) = (-1)^{n+1}/2.
    \end{split}
\end{equation*}
\end{lemma}
\begin{proof}
   Evaluating the residue of the integrand gives  
\begin{equation}\label{eqn:Sformula}
            \s^{n,m}(x,y) =  (-1)^{n+1}\frac{2^{m-n}}{(n-1)!}\sum_{i=0}^{n-1\wedge m}\binom{n-1}{i}\frac{m!}{(m-i)!}(-x-y)^{n-1-i}e^{-\frac{1}{2}(x+y)}.
    \end{equation}
    When $x=-y$ and $n-2\geq m$, the power of $(-x-y)$ is always positive; hence every term vanishes. When $m= n-1$, it can easily be seen that the value is $(-1)^{n+1}/2$.

\end{proof}

The proof of Lemma~\ref{lem:eigenfunction} is a direct calculation. The following three lemmas divide the verification into three steps. 
\begin{lemma}\label{lem:ibp}
For any $t>0$, $k,m \in \mathbb{Z}$, and $n\in \mathbb{Z}_{>0}$,
    \begin{equation}
\begin{split}
    &1_t\S^{k,m}1_0\S^{n,0}\ket{\delta_0} = 1_t\S^{k,m-n}\ket{\delta_0},\\
    &1_t\S_H^{k,m}1_0\S^{n,0}\ket{\delta_0} = 1_t\S_H^{k,m-n}\ket{\delta_0}.
\end{split}
    \end{equation}
\end{lemma}
\begin{proof}
    By the definition of $\S^{i,j}$ in \eqref{def:Sijabdef}, for $x >t>0$,
    \begin{equation}\label{eqn:local101ibo}
        \begin{split}
          (\S^{k,m}1_0\S^{n,0}\ket{\delta_0})(x)=\int_0^\infty dy s^{k,m}(x,y)s^{n,0}(y,0).
        \end{split}
    \end{equation}
    Using integration by parts,
    \begin{equation}\label{eqn:local97}
        \eqref{eqn:local101ibo} = \sum_{i=0}^{n-1}2\s^{k,m-1-i}(x,0)s^{n-i,0}(0,0)  +\int_0^\infty dy s^{k,m-n}(x,y)s^{0,0}(y,0).
    \end{equation}
    The second term is $0$ since $s^{0,0}=0$. By Lemma~\ref{lem:halfS0lemma}, the only nonzero term in the sum is the one with $i=n-1$, in which case $\s^{n-i,0}(0,0)=1/2$; thus $\eqref{eqn:local101ibo} =\s^{k,m-n}(x,0)$. Since $x>t$, we have the desired result. The proof for $\S_H$ is the same. 
          
\end{proof}

\begin{lemma}\label{lem:wtoindicator}
Recall the definition of $W$ in \eqref{def:Wdef}. When $d<k$,
    $\lim_{t\to0}\W     \S^{k,0}\ket{\delta_0} = \lim_{t\to 0}1_t \S^{k,0}\ket{\delta_0}$.
\end{lemma}
This lemma reduces the complicated $W$ operator to an indicator function. It shows that after the reduction described in Lemma~\ref{lem:kernelreduction}, the actual configuration does not affect the initial condition (all that matters is $(x_{\star},h_{\star})$ being above $(y_{\star},s_{\star})$). \begin{proof}
    Using equation~\eqref{eqn:Wequiv},
    \[W = 1_t - \sum^u_{i=1}\sum_{j=1}^d 4c_{i,j}\mtp^{u-i}\ket{\delta_t}\bra{\delta_t} \ptp^{d-j}.\]
    By Lemma~\ref{lem:aactonS}, $\bra{\delta_t} \ptp^{d-j}\S^{k,0}\ket{\delta_0} = \bra{\delta_t} \S^{k-d+j,0}\ket{\delta_0}$. By Lemma~\ref{lem:halfS0lemma}, as $t\to 0$, it is $0$ for $1 \leq j\leq d$. The coefficients \(c_{i,j}\) are independent of \(t\). Thus, what is left is $1_t$.
\end{proof}
\begin{lemma}\label{lem:astarinverseactonS}
For $n\geq m+1\geq 1$, $\normalfont \ptp^{-1}\S^{n,m}=\S^{n+1,m}$.
\end{lemma}
\begin{proof}
     For $n\geq m+1\geq 1,x\geq -z$,
\begin{equation}
\begin{split}
(\ptp^{-1}\S^{n,m})(x,z)&=\int_{-z}^x dy\int_{C_{1/2}} \frac{dw}{2\pi i}\frac{1}{2}e^{(y-x)/2} e^{-(y+z)w}\frac{(1+2w)^m}{(1-2w)^n}\\
&=\int_{C_{1/2}} \frac{dw}{2\pi i} e^{-(x+z)w}\frac{(1+2w)^m}{(1-2w)^{n+1}}-e^{-(x+z)/2}\int_{C_{1/2}}  \frac{dw}{2\pi i} \frac{(1+2w)^m}{(1-2w)^{n+1}}.
\end{split}
\end{equation}
When $n\geq m+1\geq 1$, the second term is $0$. When $x<-z$, the integral is $0$; thus the result is $\S^{n+1,m}$.
\end{proof}
Notice that this identity for $\ptp^{-1}$ holds for $\S$, but not for $\S_H$.

\begin{proof}[Proof of Lemma~\ref{lem:eigenfunction}]
    Recall $\tilde{K}_{ij}$ in \eqref{eqn:ktilde}:  \begin{equation}
        \tilde{K}_{ij} =1_0\mathcal{S}_i\begin{pmatrix}
           \mtp^{x_\star} & \ptp^{-x_\star}
        \end{pmatrix}
        \normalfont
\begin{pmatrix}
-(H^{-1})^* W^* & (H^{-1})^* W^*\mtp^{-x_\star}H\ptp^{-x_\star} WH^{-1}\\
0 &  W H^{-1}
\end{pmatrix}
\begin{pmatrix}
           \mtp^{-x_\star}\\
            \ptp^{x_\star}
        \end{pmatrix}\mathcal{S}_j^*J1_0,
\end{equation}
where $\mathcal{S}_i = \begin{pmatrix}
    \S_H^{l_i',r_i'}\\
    -\S^{r_i',l_i'}
\end{pmatrix}$ and $\mathcal{S}_j^*J = \begin{pmatrix}
    \S^{r_j',l_j'} &\S_H^{l_j',r_j'}
\end{pmatrix}$. Multiplying out the middle three matrices gives
\begin{equation}\label{eqn:local99}
        \ptp^{-x_\star}W\ptp^{x_\star} H^{-1}-(H^{-1})^*\mtp^{x_\star} W^*\mtp^{-x_\star} +(H^{-1})^* \mtp^{x_\star}W^*\mtp^{-x_\star}H\ptp^{-x_\star} W\ptp^{x_\star}H^{-1}.
\end{equation}
In this proof only, set $\mathcal{W}:=\ptp^{-x_\star}W\ptp^{x_\star} H^{-1}$. Then \eqref{eqn:local99} has the form
\[\mathcal{W}-\mathcal{W}^*+\mathcal{W}^*H\mathcal{W}.\]

Recall that $F_k = \begin{pmatrix}
    0\\
    -1_0\S^{r_k',0}\Ket{\delta_0}
\end{pmatrix}$, so by Lemma~\ref{lem:ibp}, $1_t\mathcal{S}_k^*J1_0F_k = -1_t\S_H^{l_k',0}\Ket{\delta_0}$. The reason for the leading $1_t$ is that all three terms in \eqref{eqn:local99} have distributional kernels supported on $x\geq t/2,y\geq t/2$. 
Applying \eqref{eqn:local99} to $ -1_t\S_H^{l_k',0}\Ket{\delta_0}$ gives
\begin{equation}\label{eqn:local102}
        \begin{split}
        &\mathcal{W}(-1_t\S_H^{l_k',0}\Ket{\delta_0}) =-\ptp^{-x_\star}W\ptp^{x_\star}\S^{l_k',0}\Ket{\delta_0}\\
        & = -\ptp^{-x_\star}W\S^{r_k,0}\Ket{\delta_0}=-\ptp^{-x_\star}1_0\S^{r_k,0}\Ket{\delta_0}=-\S^{l_k,0}\Ket{\delta_0}.
\end{split}
\end{equation}
The first equality uses Lemma~\ref{lem:HHinverseonS}; the second equality uses Lemma~\ref{lem:aactonS}; the third equality uses Lemma~\ref{lem:wtoindicator} after taking $t\to 0$; the fourth equality uses Lemma~\ref{lem:astarinverseactonS}. The conditions of all these lemmas are satisfied, as explained in Remark~\ref{rmk:eigenfunctionhypotheses}.

For the other two terms in \eqref{eqn:local99}, notice that 
\begin{equation}
        \begin{split}
            -\mathcal{W}^*(-\S_H^{l_k',0}\Ket{\delta_0})+\mathcal{W}^*H\mathcal{W}(-\S_H^{l_k',0}\Ket{\delta_0})&=\mathcal{W}^*(\S_H^{l_k',0}\Ket{\delta_0})-\mathcal{W}^*H(\S^{l_k',0}\Ket{\delta_0})\\
            &=\mathcal{W}^*(\S_H^{l_k',0}\Ket{\delta_0})-\mathcal{W}^*(\S_H^{l_k',0}\Ket{\delta_0})=0.
        \end{split}
\end{equation}
The first equality follows from \eqref{eqn:local102}; the second equality follows from Lemma~\ref{lem:HHinverseonS}. So altogether, the result is $-\S^{l_k,0}\Ket{\delta_0}$ after \eqref{eqn:local99} is applied.

Lastly, by Lemma~\ref{lem:ibp},
\begin{equation}\label{eqn:local104}
    1_0\mathcal{S}_i(-\S^{l_k,0}\Ket{\delta_0}) = \begin{pmatrix}
    1_0\S_H^{l_i',r_i'}(-\S^{l_k,0}\Ket{\delta_0})\\
    -1_0\S^{r_i',l_i'}(-\S^{l_k,0}\Ket{\delta_0})
\end{pmatrix}=\begin{pmatrix}
    -1_0\S_H^{l_i',r_i'-l_k}\Ket{\delta_0}\\
    1_0\S^{r_i',l_i'-l_k}\Ket{\delta_0}
\end{pmatrix}.
\end{equation}

For the other part of the kernel $R_{ik}$,
\begin{equation}\label{eqn:local105}
    R_{ik}F_k = \begin{pmatrix}
       1_0\mtp^{r_i'-l'_k}H\ptp^{r_k'-l'_i}1_0\S^{r'_k,0}\ket{\delta_0}\\
    -1_{i<k}1_0(\mtp)^{-u'_{ik}}(\ptp)^{-d'_{ik}}1_0\S^{r'_k,0}\ket{\delta_0}
\end{pmatrix}=\begin{pmatrix}
        1_0\S_H^{l_i',r_i'-l_k'}\ket{\delta_0}\\
        -1_{i<k} 1_0\S^{r_i',l_i-l_k}\ket{\delta_0}
    \end{pmatrix}.
\end{equation}
The second equality uses Lemma~\ref{lem:astarinverseactonS}, Lemma~\ref{lem:HHinverseonS} and Lemma~\ref{lem:aactonS}. Now it is easy to see that if $i<k$, \eqref{eqn:local104} and \eqref{eqn:local105} cancel, and when $i=k$, we have $K_{kk}F_k=-F_k$.

\end{proof}
 Lastly, we prove the main theorem.
 \begin{proof}[Proof of Theorem~\ref{thm:halfmultipoint}]
Fix $h_0\in\mathcal H_{\mathrm{init}}$ and
$g=\{\vec y,\vec s\}\in\mathcal H_{\mathrm{final}}$.
Let $F(t,h)$ denote the right-hand side of
\eqref{eqn:halfmultipoint}, with initial configuration $h$ and
final configuration $g$.
Propositions~\ref{prop:welldefineandtraceclass},
\ref{prop:kolm}, and~\ref{prop:initialcondition} show that $F$
is well defined, satisfies $\partial_tF=\mathcal LF$ for $t>0$,
and extends continuously to $t=0$ with
$F(0,h)=1_{h\le g}$. We verify the uniqueness needed to identify
this solution with the transition probability.

For $N\ge0$, let $E_N$ be the set of configurations obtainable
from $h_0$ by at most $N$ allowed flips, and put
$E=\bigcup_{N\ge0}E_N$.
Every configuration has only finitely many possible flips,
so each $E_N$ is finite. Moreover, $E$ is invariant under the
dynamics and $h\le h_0$ for every $h\in E$.
Write $(h_s)_{s\ge0}$ for the half-space TASEP started from $h_0$,
and define its first exit time from $E_N$ by
\[
\tau_N:=\inf\{s\ge0:h_s\notin E_N\},
\qquad \inf\varnothing:=\infty.
\]
The initial configuration has finitely many particles, injections
occur at rate at most $\alpha$, and each particle jumps at rate
at most one. Thus only finitely many transitions occur on each
bounded time interval, and consequently
$\tau_N\uparrow\infty$ almost surely.

We next establish the bound
\begin{equation}\label{eqn:uniqueness-uniform-bound}
C_T:=\sup_{h\in E}\sup_{0\le t\le T}|F(t,h)|<\infty,
\qquad T<\infty.
\end{equation}
Set
\[
A:=\sup_{x\ge0}(h_0(x)+x)<\infty,
\qquad c:=s_\star-y_\star.
\]
For any $h\in E$, repeatedly delete its rightmost peak whenever
Lemma~\ref{lem:kernelreduction} applies. If a single peak outside
$C_{y_\star,s_\star}$ remains, the triangularity argument in the
proof of that lemma gives determinant one, and hence $|F(t,h)|=1$.
Otherwise, the last remaining peak $(z,b)$ lies inside this cone,
so $b-z>c$. Since $h(x)-x$ is nonincreasing and $h\le h_0$,
every remaining peak $(x,a)$ satisfies
\[
c+x<a\le A-x.
\]
These inequalities allow only finitely many lattice points,
and hence finitely many reduced peak configurations.
The reductions preserve $F(t,h)^2=\det(I+K)$, so continuity in
$t$ for these finitely many configurations proves
\eqref{eqn:uniqueness-uniform-bound}.

Let $q(h,h')$ denote the transition rate from $h$ to $h'\ne h$,
and put
\[q(h)=\sum_{h'\ne h}q(h,h').\]
The generator of the chain killed upon leaving $E_N$ is the
finite matrix
\[
Q_N(h,h')=
\begin{cases}
q(h,h'),&h\ne h',\\
-q(h),&h=h',
\end{cases}
\qquad h,h'\in E_N.
\]
In particular, the diagonal retains the rates of jumps leaving
$E_N$. Writing $F_N(t)=F(t,\cdot)|_{E_N}$, the backward equation
becomes
\[
F_N'(t)=Q_NF_N(t)+b_N(t),
\qquad
b_N(t,h):=\sum_{h'\in E\setminus E_N}q(h,h')F(t,h').
\]
Only finitely many terms occur in this forcing, so it is
continuous. Since $E_N$ is finite and $b_N$ is continuous, the
formula in \cite[Chapter~I, Theorem~2.15]{liggettIPS} therefore gives
\[
F_N(t)=e^{tQ_N}F_N(0)
       +\int_0^t e^{sQ_N}b_N(t-s)\,ds.
\]
The first term represents paths that remain in $E_N$ up to time
$t$, while the integral represents their first exits. Thus,
evaluating at $h_0$,
\[
F(t,h_0)
=
\mathbb E_{h_0}\!\left[F(0,h_t);\,\tau_N>t\right]
+
\mathbb E_{h_0}\!\left[
F(t-\tau_N,h_{\tau_N});\,\tau_N\le t
\right].
\]
Since $F(0,h)=1_{h\le g}$, for $0<t\le T$ this implies
\[
\left|F(t,h_0)-\mathbb P_{h_0}(h_t\le g)\right|
\le (C_T+1)\mathbb P_{h_0}(\tau_N\le t)
\longrightarrow0.
\]
This proves uniqueness, and thus Theorem~\ref{thm:halfmultipoint}.
\end{proof}

\section{Scaling limit and the half-space KPZ fixed point}\label{sec:scalinglimit}
 \subsection{Transformation of the kernel}\label{subsec:transformationofkernel}
 In the kernel \eqref{eqn:halfmultipoint}, all $\S$ operators are explicit, and a standard steepest-descent method can be applied. We will first perform some algebraic transformations of the kernel, which give the operator $W$ a probabilistic interpretation and make it suitable for asymptotic analysis.
 For the random-walk representation in this subsection, let $h_0 := \mathfrak{d}_{\vec{x}}^{\vec{h}}(y) =\begin{cases}
     -2h_i  \text{ if } y=x_i,\\
     \infty \text{ if }  y\neq x_i \text{ for all } i.
 \end{cases}$
 The strict epigraph of a function $f$ is defined to be 
 \begin{equation}\label{eqn:expwalkhiti}
         \mathrm{epi}(f)=\{(m,y): m\in \mathbb{R}, y >f(m)\}.
         \end{equation}

     Recall that $\W$ is defined by 
     \begin{equation*}
    \begin{split}
        &W:= (I-W_0\overline{1}^tW_{1,2}\overline{1}^tW_{2,3}\cdots\overline{1}^t W_{n-1,n}\overline{1}^tW_{n+1}),\quad n>1;\\
&W_{i,i+1} := \mtp^{-u_i}\ptp^{-d_i},\quad  W_0 := \mtp^{u}, \quad W_{n+1} := \ptp^{d}.
    \end{split}
\end{equation*}
    The elementary integral kernels are:
\begin{equation}\ptp^{-1}(x,y) = 1_{x\geq y}\tfrac{1}{2}e^{(y-x)/2}, \quad \mtp^{-1}(x,y)=1_{y\geq x}\tfrac{1}{2}e^{(x-y)/2}.
\end{equation}
Notice that $\ptp^{-1}(x,y)$ is the transition density of a random walk with Exp$(1/2)$ jumps to the left, with mean $2$; $\mtp^{-1}(x,y)$ is the transition density of a random walk with Exp$(1/2)$ jumps to the right, with mean $2$. We define the following mean-zero random walk with independent increments $X_i,i=0,1,\ldots,x_n$, which is parameterized by the configuration $(x_1,h_1;\ldots;x_n,h_n)$: 
\begin{equation}
        \begin{split}
        &\text{For } 0\leq i<x_1, X_{i+1}-X_i \sim e^{-2D}\mtp^{-1},\\
        &\text{For } x_j\leq i<x_j+d_j, X_{i+1}-X_i \sim e^{2D}\ptp^{-1},\quad j=1,\ldots,n-1\\
        &\text{For } x_j+d_j\leq i<x_{j+1}, X_{i+1}-X_i \sim e^{-2D}\mtp^{-1},\quad j=1,\ldots,n-1.
\end{split}
\end{equation}
We want to group the steps between peaks $x_i$ and $x_{i+1}$. Define $X^{i,j}(x,y)$ to be the transition density of the walk between times $x_i$ and $x_j$, starting at $x$ and ending at $y$. Let 
\[W^{\circ}_{i,i+1} = e^{2h_{i}D}W_{i,i+1}e^{-2h_{i+1}D}.\]
Then $X^{i,j} = W^\circ_{i,i+1}W^\circ_{i+1,i+2}\cdots W^\circ_{j-1,j}$ for $i<j$ and $X^{i,j}=I$ for $i=j$. We want to study when the walk $X$ hits the epigraph of $h_0$. We define
 \[
     \tau = \min\{0\leq i\leq x_n: X_i >h_0(i)\},\]
     with the convention that $\tau = \infty$ if the set is empty.
Define the hit operator $X^{\mathrm{hit}_i}$ with the following integral kernel:
     \begin{equation}
         X^{\mathrm{hit}_i}(x,z) :=\int^\infty_{-{2h_i}}\mathbb{P}_x(\tau = x_i, X_{x_i} \in dy)X^{i,n}(y,z).
     \end{equation}
     The integral is well defined because the kernel $X^{i,n}(y,z)$ is a transition probability kernel.

 \begin{proposition}\label{prop:halfWasphit}
  \normalfont   Recall the operator $W$ defined in \eqref{def:Wdef}. For $n\geq 1$,
\begin{equation}
        \W = \mtp^{x_\star }e^{-tD}e^{2(x_1-h_1)D}\big(\sum_{i=1}^{n}X^{\mathrm{hit}_i}\big)e^{(t+2h_n)D}\ptp^{d}.
        \end{equation}
 \end{proposition}
 \begin{proof}
 We first consider \(n>1\). Recall that $x_\star -x_1 = u$. For this random-walk representation only, define $W_{0,1}:=\mtp^{-x_1}, W^{\circ}_{0,1}:=e^{-2x_1D}\mtp^{-x_1}$, replacing the earlier boundary convention for $W_{0,1}$ for this proof. The latter is the transition operator from time $0$ to time $x_1$.
 \[W=\mtp^{x_\star }(\mtp^{-x_\star }\ptp^{-d}-W_{0,1}\overline{1}^{t} W_{1,2}\overline{1}^{t}\cdots\overline{1}^{t}W_{n-1,n}\overline{1}^{t})\ptp^{d}.\]
         Notice that this operator is a combination of $\ptp^{-1},\mtp^{-1}$ and the projection operator. Since all $W_{i,i+1}(x,y)$ only depend on the difference of $x,y$, we can shift the whole operator by $t$ to obtain
\begin{equation*}
W = \mtp^{x_\star }e^{-tD}(\mtp^{-x_\star }\ptp^{-d}-W_{0,1}\overline{1}^{0} W_{1,2}\overline{1}^{0}\cdots\overline{1}^{0}W_{n-1,n}\overline{1}^{0})e^{tD}\ptp^{d}.
\end{equation*}
Then we want to shift the starting point and endpoint of $W_{i,j}$ to center the increments, i.e., write 
\begin{equation}
\begin{split}
&(\mtp^{-x_\star }\ptp^{-d}-W_{0,1}\overline{1}^{0}W_{1,2}\overline{1}^{0}\cdots \overline{1}^{0}W_{n-1,n}\overline{1}^0)\\
&=e^{2(x_1-h_1)D}\bigg(e^{2(h_1-x_1)D}\mtp^{-x_\star }\ptp^{-d}e^{-2h_{n}D}\\
&\qquad-W_{0,1}^\circ\overline{1}^{-2h_1}W^\circ_{1,2}\overline{1}^{-2h_2}\cdots \overline{1}^{-2h_{n-1}}W^\circ_{n-1,n}\overline{1}^{-2h_n}\bigg)e^{2h_{n}D}.
\end{split}
\end{equation}
We use $P^{\mathrm{hit}}(x,y)$ to denote the operator in the bracket; it is the probability density for the random walk $X$ to start at $x$, end at $y$, and exceed $-2h_i$ at time $x_i$ for some $i$.

        We want to further rewrite the probability. We have the following identity, whose proof follows from the probabilistic interpretation of the formula.
        \begin{equation}
                \begin{split}
                P^{\mathrm{hit}}=\sum_{i=1}^n W^\circ_{0,i}1_{-2h_i}\W^\circ_{i,i+1}\W^\circ_{i+1,i+2}\cdots \W^\circ_{n-1,n}.
                \end{split}
                \end{equation}
                Here \[W^\circ_{0,i} =W^\circ_{0,1} \overline{1}^{-2h_1} W^\circ_{1,2}\overline{1}^{-2h_2}\cdots \overline{1}^{-2h_{i-1}}W^\circ_{i-1,i}.\] This formula decomposes $P^{\mathrm{hit}}(x,y)$ according to the first wedge hit by the walk. Each summand corresponds to a walk that avoids the first $i-1$ wedges, hits the $i$-th wedge, and then proceeds freely to the endpoint. Using the notation we defined before the proposition, we have $P^{\mathrm{hit}} = \sum_{i=1}^nX^{\mathrm{hit}_i}$.

                For \(n=1\), \(u=d=0, x_\star = x_1\), and \(P^{\mathrm{hit}} = e^{-2x_1D}\mtp^{-x_1}1_{-2h_1}\). Substitution into the claimed representation gives
                \(\mtp^{x_1}e^{-tD}e^{2(x_1-h_1)D}P^{\mathrm{hit}}e^{(t+2h_1)D} = 1_t =W\).
                Thus the statement is proved.
         \end{proof}
Now, we want to combine $W,W^*$ into $S$. Recall that in \eqref{eqn:ktilde},

\begin{equation}\label{eqn:ktildeasym}
        \begin{split}
            \tilde{K}_{ij} &=1_0\mathcal{S}_i\begin{pmatrix}
           I & I
        \end{pmatrix}
        \normalfont
\begin{pmatrix}
 -(VH^{-1})^* & (VH^{-1})^*H VH^{-1}\\
0 &  V H^{-1}
\end{pmatrix}
\begin{pmatrix}
           I\\
           I
        \end{pmatrix}\mathcal{S}_j^*J1_0\\
        &=1_0\begin{pmatrix}
    \S_H^{l_i',r_i'} \\-\S^{r_i',l_i'}
\end{pmatrix}\begin{pmatrix}
           -(V H^{-1})^* & I
        \end{pmatrix}
        \normalfont
\begin{pmatrix}
I & -H \\
0 &  I
\end{pmatrix}
\begin{pmatrix}
           I\\
           V H^{-1}
        \end{pmatrix}\begin{pmatrix}
    \S^{r_j',l_j'} &\S_H^{l_j',r_j'}
\end{pmatrix}1_0.
        \end{split}
\end{equation}
Now we investigate $\begin{pmatrix}
           I\\
           V H^{-1}
        \end{pmatrix}\begin{pmatrix}
    \S^{r_j',l_j'} &\S_H^{l_j',r_j'}
\end{pmatrix}1_0$. Recall that 

\begin{multline}\label{eqn:local118}
    (\sum_{i=1}^nX^{\mathrm{hit}_i})e^{(t+2h_n)D}\ptp^{d+x_\star }H^{-1}\S^{r_j',l_j'}(x,\cdot)\\ =\sum_{k=1}^n \int^\infty_{-{2h_k}}\mathbb{P}_x(\tau = x_k, X_{x_k} \in dy)X^{k,n}(y,\cdot)e^{(t+2h_n)D}\ptp^{d+x_\star }H^{-1}\S^{r_j',l'_j}.
\end{multline}
Recall that $X^{k,n}=\W^\circ_{k,k+1}\W^\circ_{k+1,k+2}\cdots \W^\circ_{n-1,n}$ is the free mean-zero random walk. Then 
\[e^{-(t+2h_k)D}X^{k,n}e^{(t+2h_n)D}=\W_{k,k+1}\cdots \W_{n-1,n} = \mtp^{-(u_k+\cdots+u_{n-1})}\ptp^{-(d_k+\cdots+d_{n-1})}.\]
Since all $\mtp,\ptp$ commute with $H^{-1}$,
\begin{multline}
1_{-2h_k}X^{k,n}e^{(t+2h_n)D}\ptp^{d+x_\star }H^{-1}\S^{r_j',l_j'}1_0=1_{-2h_k}e^{(t+2h_k)D}H^{-1}\mtp^{-u_{kn}}\ptp^{d+x_\star -d_{kn}}\S^{r_j',l_j'}1_0\\
= 1_{-2h_k}e^{(t+2h_k)D}H^{-1}\S^{\frac{h_k-x_k-s_j-y_j}{2},\frac{h_k+x_k-s_j+y_j}{2}}1_0.
\end{multline}

      The reason why $\ptp^{d+x_\star -d_{kn}}$ can act directly on $S$ without boundary terms is that $(1_{-2h_k}e^{(t+2h_k)D}H^{-1})(x,y)$ is supported on $y\geq t$.

     Further, to use notation similar to that in \cite{MQR21}, we define 
     \begin{multline}\label{def:Sepidef}
             (H^{-1}\S)^{\mathrm{epi},-y_j}(x,z)\\:= \sum_{k=1}^{n}\int_{-2h_k}^\infty \mathbb{P}_x(\tau = x_k, X_{x_k} \in dy)e^{(t+2h_k)D}H^{-1}\S^{\frac{h_k-x_k-s_j-y_j}{2},\frac{h_k+x_k-s_j+y_j}{2}}(y,z).
         \end{multline}
    Also, $S^{\mathrm{epi},-y_j}(x,z)$ is defined similarly, with $H^{-1}$ not present.
     With this notation, \eqref{eqn:local118} is just $(H^{-1}\S)^{\mathrm{epi},-y_j}$. With the same procedure, 
     \[(\sum_{i=1}^nX^{\mathrm{hit}_i})e^{(t+2h_n)D}\ptp^{d+x_\star }H^{-1}\S_H^{l_j',r_j'}=\S^{\mathrm{epi},y_j}.\]

     Performing the same operations on the other side, we have
     \begin{equation}\label{eqn:kernelforscaling1}
     \begin{split}
    &\eqref{eqn:ktildeasym}= 1_0\begin{pmatrix}
            -(\S^{\mathrm{epi},y_i})^*e^{-2(x_1-h_1)D}e^{tD}\ptp^{x_\star }\mtp^{-x_\star } & \S_{H}^{l_i',r_i'}\\
            ((H^{-1}\S)^{\mathrm{epi},-y_i})^*e^{-2(x_1-h_1)D}e^{tD}\ptp^{x_\star }\mtp^{-x_\star } & -\S^{r_i',l_i'}
        \end{pmatrix}\cdot\begin{pmatrix}
 I&-H\\
   0 &I
\end{pmatrix}\\
&\cdot\begin{pmatrix}
            \S^{r_j',l_j'} & \S_{H}^{l_j',r_j'}\\
           \ptp^{-x_\star }\mtp^{x_\star }e^{-tD}e^{2(x_1-h_1)D}(H^{-1}\S)^{\mathrm{epi},-y_j} & \ptp^{-x_\star }\mtp^{x_\star }e^{-tD}e^{2(x_1-h_1)D}\S^{\mathrm{epi},y_j}
        \end{pmatrix}1_0.
     \end{split}
    \end{equation}
    Recall that $x_\star  = (h_n-h_1+x_n+x_1)/2$. For simplicity, we define 
    \[A:=e^{-4x_\star D}\ptp^{x_\star }\mtp^{-x_\star }.\]
     We factor the kernel as follows:
    \begin{equation}\label{eqn:kernelforscaling2}
     \begin{split}
    &\eqref{eqn:ktildeasym}= 1_0\begin{pmatrix}
            -(\S^{\mathrm{epi},y_i})^*e^{tD} & \S_{H}^{l_i',r_i'}e^{-2(h_n+x_n)D}A^*\\
            ((H^{-1}\S)^{\mathrm{epi},-y_i})^*e^{tD} & -\S^{r_i',l_i'}e^{-2(h_n+x_n)D}A^*
        \end{pmatrix}\begin{pmatrix}
 I&-H\\
   0 &I
\end{pmatrix}\\
&\qquad\cdot\begin{pmatrix}
            Ae^{2(h_n+x_n)D} \S^{r_j',l_j'} & Ae^{2(h_n+x_n)D}\S_{H}^{l_j',r_j'}\\
           e^{-tD}(H^{-1}\S)^{\mathrm{epi},-y_j} & e^{-tD}\S^{\mathrm{epi},y_j}
        \end{pmatrix}1_0,
     \end{split}
    \end{equation}
    which is the form in which we will perform an asymptotic analysis. 

    \subsection{Convergence of the kernel}\label{sec:pointwiselimit}
    
    Now we are ready to consider the scaling limit of the TASEP height function. 

    For $\varepsilon>0$, the $1{:}2{:}3$ rescaled TASEP height function  is 
    \begin{equation}\label{eqn:halfheightscale}
            \bh^{\varepsilon}(\bt,\bx):=\varepsilon^{1/2}[h(2\varepsilon^{-3/2} \bt, 2\varepsilon^{-1}\bx)+\varepsilon^{-3/2}\bt],\quad \alpha^\varepsilon =(1+\brho\varepsilon^{1/2})/2,
            \end{equation}
        with the initial condition also scaled as \(\bh^\varepsilon(0,\bx) := \varepsilon^{1/2}h(0,2\varepsilon^{-1}\bx)\). This is the weak scaling around the critical value \(\alpha = 1/2\), under which one still sees the injection rate in the limit. If one fixes $\alpha>1/2$, then one can derive the formula in the symplectic--unitary transition regime.

    We scale all the variables appearing in Theorem~\ref{thm:halfmultipoint} as follows:
\begin{equation}\label{eqn:variablescaling}
\begin{split}
    t^\varepsilon = 2\varepsilon^{-3/2}\bt, \quad &h^\varepsilon_i= \varepsilon^{-1/2}\bh^\varepsilon_i+\varepsilon^{-3/2}\bt,\quad x^\varepsilon_i = 2\varepsilon^{-1}\bx^\varepsilon_i,\\
    &s^\varepsilon_i=\varepsilon^{-1/2}\bs_i,\quad y^\varepsilon_i = 2\varepsilon^{-1}\by_i.
\end{split}
 \end{equation} 
 The bounded lattice corrections will vanish after
rescaling and are thus omitted below.
For each $\varepsilon>0$, we allow the number of peaks $n_\varepsilon$ to depend on $\varepsilon$. The quantities
$\bx_i^\varepsilon$ and
$\bh_i^\varepsilon$ are also allowed to depend on
$\varepsilon$. 

We use the half-line version of the state space $\mathrm{UC}$ and its
local topology from \cite[Section~3.1]{MQR21}:
\begin{multline}
    \mathrm{UC}:= \{\mathfrak{h}:\mathbb{R}_{\geq 0}\to [-\infty,\infty): \mathfrak{h} \text{ is upper semicontinuous, }\\ \mathfrak{h}(\bx)\leq \alpha+\gamma \bx \text{ for some }0<\alpha,\gamma<\infty  \text{ and } \mathfrak{h}(\bx)>-\infty \text{ for some }\bx\},
   \end{multline}
   with the topology of local UC convergence. Here $\alpha$ and $\gamma$ are upper-bound constants, not the boundary injection parameter.
Associated with the peak configuration, define the upper
semicontinuous multiple-wedge profile
\begin{equation}
\label{eqn:rescaled-multiple-wedge-profile}
    \mathfrak d^\varepsilon(\bx)
    :=
    \begin{cases}
        \bh_k^\varepsilon,
        &\bx=\bx_k^\varepsilon
          \text{ for some }1\leq k\leq n_\varepsilon,\\
        -\infty,&\text{otherwise}.
    \end{cases}
\end{equation}
The assumption on the initial data is:
\begin{equation}
\label{eqn:multiple-wedge-UC-convergence}
    \mathfrak d^\varepsilon
    \longrightarrow
    \mathfrak h_0
    \quad\text{locally in UC},\quad \varepsilon^{3/2}x_\star^\varepsilon\to 0,\qquad
   \varepsilon^{3/2}\bigl(h_\star^\varepsilon-\varepsilon^{-3/2}\bt\bigr)\to 0 .
\end{equation}
In particular, the profiles satisfy a common linear upper bound:
there exist $\alpha,\gamma<\infty$ such that
\begin{equation}\label{eqn:common-linear-upper-bound}
    \bh_k^\varepsilon
    \leq
    \alpha+\gamma\bx_k^\varepsilon,
    \qquad
    1\leq k\leq n_\varepsilon,
\end{equation}
for all sufficiently small $\varepsilon$.

No convergence of the individual labeled peaks is assumed in
\eqref{eqn:multiple-wedge-UC-convergence}. Rather, \(\mathrm{UC}\) convergence implies the following: For any sequence $k_\varepsilon$ of peak
indices,
\[
\bx_{k_\varepsilon}^\varepsilon\to\bx \text{ implies }\limsup_{\varepsilon\to0}
\bh_{k_\varepsilon}^\varepsilon
\leq
\mathfrak h_0(\bx).
\]
For every $\bx$ such that $\mathfrak h_0(\bx)>-\infty$, there
exists a sequence $k_\varepsilon$ for which
\[
\bx_{k_\varepsilon}^\varepsilon\to\bx,
\qquad
\bh_{k_\varepsilon}^\varepsilon
\to
\mathfrak h_0(\bx).
\]
Also, notice that the assumption in \eqref{eqn:multiple-wedge-UC-convergence} implies that $\varepsilon^{3/2}(x_n^\varepsilon+h_n^\varepsilon) \to \bt$.

 We define our scaled and conjugated kernel to be  \begin{equation}\label{eqn:scaledkernel}
         K^\varepsilon(\bz_1,\bz_2) =\begin{pmatrix}
    2\varepsilon^{-1/2}K^\varepsilon_{11}(z^\varepsilon_1,z^\varepsilon_2)&4 K^\varepsilon_{12}(z^\varepsilon_1,z^\varepsilon_2)\\
    \varepsilon^{-1} K^\varepsilon_{21}(z^\varepsilon_1,z^\varepsilon_2)&2\varepsilon^{-1/2} K^\varepsilon_{22}(z^\varepsilon_1,z^\varepsilon_2)
\end{pmatrix},
 \end{equation}
 where $z_i^\varepsilon = 2\varepsilon^{-1/2}\bz_i$. On the right-hand side, $K^\varepsilon_{ab}$ denotes the microscopic kernel with scaled parameters; on the left-hand side, $K^\varepsilon$ denotes the rescaled and conjugated kernel. Its Fredholm Pfaffian is
unchanged by Lemma~\ref{lem:appendixconjugationinvariance}.

   Before we state the convergence results, we need to develop some notation for the limiting objects. For $\bt>0$, 
\begin{equation}
    \begin{split}
        \bS^{\bt,\bx}(\bz_1,\bz_2)&:=\tfrac{1}{2\pi i}\int_{C^{\pi/3}_{1}} dw e^{\bt w^3/3+\bx w^2+(\bz_1-\bz_2)w}\\
    &= \bt^{-1/3}e^{\frac{2\bx^3}{3\bt^2}-\frac{(\bz_1-\bz_2)\bx}{\bt}}\Ai(-\bt^{-1/3}(\bz_1-\bz_2)+\bt^{-4/3}\bx^2),
    \end{split}
\end{equation}
    where $C_{a}^{\pi/3} = \{a+re^{\pm i\pi/3}:r\in[0,\infty)\}$ with the orientation going from $\infty e^{-i\pi/3}$ to $\infty e^{i\pi/3}$. This is the integral kernel for the operator $ e^{\bx\bD^2+\bt\bD^3/3}$. One useful property of the operator is $\bS^{-\bt,\bx}= (\bS^{\bt,\bx})^*$. We are going to use a variation of this operator. We define
\begin{equation}
        \begin{split}
             &\bS_{\bH}^{\bt,\bx}(\bz_1,\bz_2)=\int_{C_{a}^{\pi/3}+ (1-\frac{1}{2}1_{\brho=0})C_{\brho}}\frac{dw}{2\pi i}\frac{we^{\bt w^3/3+\bx w^2 +(\bz_1-\bz_2) w}}{(\brho+w)(\brho-w)},
        \end{split}
    \end{equation}
   where $a>|\brho|$ and $C_{\brho}$ is a negatively oriented small loop enclosing only $w=\brho$. $\tfrac{1}{2}C_0$ means that the residue at $w=0$ is multiplied by $\tfrac{1}{2}$. Define the limits of the operators $H,H^{-1}$:
   \begin{equation}
           \begin{split}
               &\bH(\bz_1,\bz_2) = -\frac{1}{2}1_{\bz_1>\bz_2}e^{-\brho(\bz_1-\bz_2)}+\frac{1}{2}1_{\bz_1<\bz_2}e^{-\brho(\bz_2-\bz_1)},\\
               &\bH^{-1} = \bD^{-1}(\brho-\bD)(\brho+\bD),\quad \bD^{-1}(\bx,\by) =1_{\bx>\by}.
           \end{split}
   \end{equation}
   Notice that the definition of \(\bD^{-1}\) differs from its discrete counterpart in \eqref{def:HHinversedef}. This is because a reflection operator is applied on both sides of \(H\), as will become clear below.

    Let $\mathfrak{h}\in \mathrm{UC}$, let $\bB(x)$ be a Brownian motion with diffusion coefficient $2$, and let $\btau = \inf\{s\geq 0:\bB(s)\leq \mathfrak{h}(s)\}$. Define
    \begin{equation}
            (\bH^{-1}\bS)^{\mathrm{hypo}(\mathfrak{h}),\bt,\bx}(\bz_1,\bz_2)= \mathbb{E}_{\bB(0)=\bz_1}[\bH^{-1}\bS^{\bt,\bx-\btau}(\bB(\btau),\bz_2)1_{\btau<\infty}].
    \end{equation}
    $\bS^{\mathrm{hypo}(\mathfrak{h}),\bt,\bx}(\bz_1,\bz_2)$ is defined similarly with $\bH^{-1}$ not present. Now we are ready to state our main convergence theorem.

    \begin{theorem}\label{thm:pointwiseconvergence}
            Let $\mathfrak{h}_0 \in \mathrm{UC},\bt>0$. Let $\bh^\varepsilon(\bt,\cdot)$ be the rescaled TASEP height function defined in \eqref{eqn:halfheightscale}, with the boundary parameter scaled as $\alpha^\varepsilon = (1+\brho\varepsilon^{1/2})/2$. Assume the initial condition satisfies the assumption in \eqref{eqn:multiple-wedge-UC-convergence}. Then for any $\by_1<\cdots<\by_m\in [0,\infty), \bs_1,\ldots,\bs_m \in \mathbb{R}$, 
            \begin{equation}\label{eqn:halffixedpointformula}
                 \lim_{\varepsilon\to0}\mathbb{P}(\bh^\varepsilon(\bt,\by_1)\leq \bs_1,\ldots, \bh^{\varepsilon}(\bt, \by_m)\leq \bs_m) =\Pf(J+ JK^{\mathrm{fp}})_{L^2(\{1,\ldots,m\}\times [0,\infty))}.   
            \end{equation}  
            $K^{\mathrm{fp}}_{ij}=R^{\mathrm{fp}}_{ij}+\tilde{K}^{\mathrm{fp}}_{ij}$, where
            \begin{equation}\label{def:fixedppointRdef}
                    \begin{split}
                   R^{\mathrm{fp}}_{ij} &=
                   \begin{pmatrix}
                           1_{j<i} e^{\bs_i \bD}e^{(\by_i-\by_j)\bD^2}e^{-\bs_j \bD} & -e^{\bs_i\bD+\by_i\bD^2}\bH e^{-\bs_j\bD+\by_j\bD^2}\\
                            0 &  1_{i<j}e^{\bs_i \bD}e^{(\by_j-\by_i)\bD^2}e^{-\bs_j \bD}
                   \end{pmatrix}, 
            \end{split}
            \end{equation}
            and $\tilde{K}^\mathrm{fp}_{ij}$ is
              \begin{equation}\label{def:fixedppointtildeKdef}
           \begin{split}
               \begin{pmatrix}
                   e^{\bs_i \bD} & 0\\
                   0 & e^{\bs_i \bD}
           \end{pmatrix}
                 \begin{pmatrix}
                     -(\bS^{\mathrm{hypo}(\mathfrak{h}_0),\bt,\by_i})^* &     -(\bS_{\bH}^{\bt,\by_i})^*\\
                     (-(\bH^{-1}\bS)^{\mathrm{hypo}(\mathfrak{h}_0),\bt,-\by_i})^*& -(\bS^{\bt,-\by_i})^*
                 \end{pmatrix}
                 \begin{pmatrix}
                     I & \bH\\
                     0 & I
                 \end{pmatrix}\\
                 \begin{pmatrix}
                       \bS^{\bt,-\by_j} & -\bS_{\bH}^{\bt,\by_j}\\
                         -(\bH^{-1}\bS)^{\mathrm{hypo}(\mathfrak{h}_0),\bt,-\by_j} & \bS^{\mathrm{hypo}(\mathfrak{h}_0),\bt,\by_j}
                 \end{pmatrix}
                 \begin{pmatrix}
                   e^{-\bs_j \bD} & 0\\
                   0 & e^{-\bs_j \bD}
           \end{pmatrix}.
           \end{split}
             \end{equation}  
    \end{theorem}
    
 In order to simplify the notation and make the kernel structure clearer, we define the following operators which will be used throughout the rest of the paper. Set \(A_\varepsilon
:=e^{-4x_\star^\varepsilon D}
\ptp^{x_\star^\varepsilon}\mtp^{-x_\star^\varepsilon}\), and let \(\reflection\) be the reflection operator such that \(\reflection f(x)=f(-x)\).
         \begin{equation}\label{eqn:FeGedefinition}
           \begin{split}
             & F_{\varepsilon,-j}
:=
2\varepsilon^{-1/2}\reflection A_\varepsilon e^{2(x_n^\varepsilon+h_n^\varepsilon)D}\S^{r_{j,\varepsilon}',l_{j,\varepsilon}'}
,\quad  F_{\varepsilon,j}
:=
4\reflection A_\varepsilon e^{2(x_n^\varepsilon+h_n^\varepsilon)D}\S_H^{l_{j,\varepsilon}',r_{j,\varepsilon}'}
,\\
&
G_{\varepsilon,-j}
:=
\varepsilon^{-1}\reflection e^{-t^\varepsilon D}
(H^{-1}\S)_{\varepsilon}^{\mathrm{epi},-y_j}, \qquad
G_{\varepsilon,j}
:=
2\varepsilon^{-1/2}\reflection e^{-t^\varepsilon D}
\S_{\varepsilon}^{\mathrm{epi},y_j}.
           \end{split}
         \end{equation}
These are the four kernels in the last matrix in \eqref{eqn:kernelforscaling2} after scaling, with the microscopic kernels evaluated at \(z^\varepsilon = 2\varepsilon^{-1/2}\bz\). By the explicit formula in \eqref{def:HHinversedef}, the rescaling of $H$ is
exact:
\begin{equation}\label{eqn:Hoperatorconvergence}
    4H_{\alpha^\varepsilon}
\left(z_1^\varepsilon,z_2^\varepsilon\right)
=
\bH(\bz_1,\bz_2).
\end{equation} In this notation, one can see that \eqref{eqn:kernelforscaling2} can be written in the following form: 
\begin{equation}\label{eqn:factorization-FG}
\tilde{K}^\varepsilon_{ij}
=1_0
\begin{pmatrix}
-G_{\varepsilon,i}^*&F_{\varepsilon,i}^*\\
G_{\varepsilon,-i}^*&-F_{\varepsilon,-i}^*
\end{pmatrix}
\begin{pmatrix}I&\bH\\0&I\end{pmatrix}
\begin{pmatrix}
F_{\varepsilon,-j}&F_{\varepsilon,j}\\
G_{\varepsilon,-j}&G_{\varepsilon,j}
\end{pmatrix}1_0.
\end{equation}
The sign on \(4H_{\alpha^{\varepsilon}}(z_1^\varepsilon,z_2^\varepsilon)\) changes because \(\reflection H \reflection =-H\).
For the limiting kernel, we denote the product of the fourth and fifth matrices in \eqref{def:fixedppointtildeKdef} by \(\boldsymbol{\mathcal{S}}_j = \begin{pmatrix}
    F_{-j} & F_j\\
    G_{-j} & G_{j}
\end{pmatrix}\); then \(\tilde{K}^{\mathrm{fp}}_{ij} = J\boldsymbol{\mathcal{S}}^*_iJ\begin{pmatrix}
        I & \bH\\
        0 & I
    \end{pmatrix}\boldsymbol{\mathcal{S}}_j\).
    Recalling the definition of \(G_{\varepsilon,\pm j}\) in \eqref{eqn:FeGedefinition}, we need to define the following scaled components:
 \begin{equation}\label{eqn:local-hitting-integrand}
\begin{split}
 g_{\varepsilon,-j,k}(\bz_1,\bz_2)
:={}&
\varepsilon^{-1}
e^{2h_k^\varepsilon D}
H^{-1}
\S_\varepsilon^{\frac{h_k-x_k-s_j-y_j}{2},\frac{h_k+x_k-s_j+y_j}{2}}\left(-
z_1^\varepsilon,
\,z_2^\varepsilon
\right),\\
g_{\varepsilon,j,k}(\bz_1,\bz_2)
:={}&
2\varepsilon^{-1/2}
e^{2h_k^\varepsilon D}
\S_\varepsilon^{\frac{h_k-x_k-s_j+y_j}{2},\frac{h_k+x_k-s_j-y_j}{2}}\left(-
z_1^\varepsilon,
\,z_2^\varepsilon
\right).
\end{split}
\end{equation}
The corresponding limiting kernels are:
\begin{equation}\label{eqn:local-hitting-integrand-limit}
    g_{j,\bx_k^\varepsilon}(\bz_1,\bz_2) = \bS^{\bt,\by_j-\bx_k^\varepsilon}(\bz_1,\bz_2+\bs_j), \quad g_{-j,\bx_k^\varepsilon}(\bz_1,\bz_2) = -(\bH^{-1}\bS)^{\bt,-\by_j-\bx_k^\varepsilon}(\bz_1,\bz_2+\bs_j).
\end{equation}
The subprobability measure in \eqref{def:Sepidef} will also be rescaled. We define the shifted, reflected, diffusively scaled walk by
          \begin{equation}\label{eqn:randomwalkscaling}
                  B^\varepsilon(\bx)= -\frac{\varepsilon^{1/2}}{2}(X_{2\varepsilon^{-1}\bx}+2\varepsilon^{-3/2}\bt),
          \end{equation}
          with linear interpolation between lattice points, and set
          \begin{equation}\label{eqn:hittingtimedef}
                  \btau^\varepsilon = \min\{\bx^\varepsilon_k: B^\varepsilon(\bx^\varepsilon_k)\leq \bh_k^\varepsilon, 1\leq k\leq n_\varepsilon\},
          \end{equation}
with the convention $\min\emptyset=\infty$.
Let
\begin{equation}\label{eqn:prop39-discrete-hitting-measure}
\nu_{\bz}^\varepsilon(d\bx,d\bb)
=
\mathbb P_{\bz}
\left(
\btau^\varepsilon\in d\bx,\,
B^\varepsilon(\btau^\varepsilon)\in d\bb,\,
\btau^\varepsilon<\infty
\right)
\end{equation}
be the joint first-hitting subprobability measure and let \(\nu_{\bz}\) be defined analogously for \(\bB\) and \(\btau\). Using \(\nu_{\bz}^\varepsilon\), we can write 
\begin{equation}\label{eqn:Gdef}
    G_{\varepsilon,\pm j}(\bz_1,\bz_2)=\sum_{k=1}^{n_\varepsilon}
\int_{-\infty}^{\bh_k^\varepsilon}\nu_{\bz_1}^\varepsilon
\bigl(\{\bx_k^\varepsilon\},d\bb\bigr)
 g_{\varepsilon, \pm j,k}(\bb,\bz_2).
\end{equation}
 We first record the following lemma about the hitting measure.
\begin{lemma}\label{lem:revised-hitting-measures}
For $z\ne\mathfrak h_0(0)$,
$\nu_z^\varepsilon\to\nu_z$ vaguely on $[0,\infty)\times\mathbb R$.
\end{lemma}
\begin{proof}
The increments of $B^\varepsilon$ are $\pm \varepsilon^{1/2}(1-E)$, where $E$ has
a mean-one exponential law. The triangular-array
invariance principle therefore gives $B^\varepsilon\Rightarrow\bB$
on compact intervals. For $z>\mathfrak h_0(0)$, applying
\cite[Proposition~3.2]{MQR21}, using locally uniform convergence of paths, gives joint
convergence of the hitting times and positions on bounded time
intervals. For $z<\mathfrak h_0(0)$, local UC convergence forces a hit
at a time tending to zero, with position tending to $z$.
\end{proof}

 We fix \(\theta_1,\ldots,\theta_m\) such that 
\[|\brho| <\theta_1<\cdots<\theta_m.\]
For $f_i=(f_{i,1},f_{i,2})^{\mathsf T}$, define
\[
(\mathcal Mf)_i
=
\begin{pmatrix}
M_{-\theta_i}f_{i,1}\\
M_{\theta_i}f_{i,2}
\end{pmatrix},
\qquad M_cf(x)=e^{cx}f(x).
\]
         To prove the theorem, we will show the following proposition:
         \begin{proposition}\label{prop:mixedSchattenconvergence}
Under the assumptions of Theorem~\ref{thm:pointwiseconvergence},
\begin{equation}\label{eqn:diagonaltracenormconvergence}
\lVert[\mathcal M(K^\varepsilon-K^{\mathrm{fp}})\mathcal M^{-1}]_{11}\rVert_1
+\lVert[\mathcal M(K^\varepsilon-K^{\mathrm{fp}})\mathcal M^{-1}]_{22}\rVert_1
\longrightarrow0,
\end{equation}
and
\begin{equation}\label{eqn:offdiagonalHSconvergence}
\lVert[\mathcal M(K^\varepsilon-K^{\mathrm{fp}})\mathcal M^{-1}]_{12}\rVert_2
+\lVert[\mathcal M(K^\varepsilon-K^{\mathrm{fp}})\mathcal M^{-1}]_{21}\rVert_2
\longrightarrow0.
\end{equation}
\end{proposition}
        
The proposition will be proved in Section~\ref{subsec:Fredholm-convergence}. In the rest of the current section, we will obtain bounds and convergence results for each part of the kernel. For individual functions,
$\lVert\cdot\rVert_2$ denotes the $L^2([0,\infty))$ norm, and $\lVert\cdot\rVert_1$ denotes the trace norm on $[0,\infty)$.


\begin{lemma}\label{lem:firstSbound}
Set
\begin{equation}\label{eqn:free-row-norms}
W_{\varepsilon,i}(\bz)
=
\lVert F_{\varepsilon,-i}(\bz,\cdot)M_{\theta_i}\rVert_2
+
\lVert F_{\varepsilon,i}(\bz,\cdot)M_{-\theta_i}\rVert_2,
\end{equation}
and define $W_i$ by omitting $\varepsilon$.
Fix $0<\delta<\min\{1,\bt/8\}$. There exist
$a,C,\varepsilon_0>0$ such that, for every $i$ and
$0<\varepsilon<\varepsilon_0$,
\begin{equation}\label{eqn:firstSbound-improved}
\begin{aligned}
W_{\varepsilon,i}(\bz)&\leq \begin{cases}
    Ce^{a|\bz|},\quad 
\bz\leq\delta\varepsilon^{-1},\\
C\exp\!\left\{
a|\bz|+C\varepsilon^{-1}\log(2\varepsilon^{-1/2})\right\},\quad 
\bz\geq \delta\varepsilon^{-1}.
\end{cases}\\
W_i(\bz)&\leq Ce^{a|\bz|},
\bz\in\mathbb R.
\end{aligned}
\end{equation}
Moreover, $F_{\varepsilon,\pm i}\to F_{\pm i}$ locally uniformly
on $\mathbb R^2$, and for every compact $K\subset\mathbb R$,
\begin{equation}\label{eqn:firstSbound-L2convergence}
\begin{split}
\sup_{\bz\in K}\Bigl\{
\lVert[F_{\varepsilon,-i}(\bz,\cdot)-F_{-i}(\bz,\cdot)]
M_{\theta_i}\rVert_2+ \lVert[F_{\varepsilon,i}(\bz,\cdot)-F_i(\bz,\cdot)]
M_{-\theta_i}\rVert_2
\Bigr\}\to 0.
\end{split}
\end{equation}
\end{lemma}

\begin{proof}
    \smallskip\noindent\emph{Bound on $W_{\varepsilon,i}(\bz)$.}
We first obtain a uniform
exponential bound when $\bz\leq\delta\varepsilon^{-1}$, and then a
 bound with an explicitly recorded $\varepsilon$-dependent
loss for $\bz\geq \delta\varepsilon^{-1}$. Then we show pointwise convergence. 

Fix $i$ and set
\[
\bu=\bv+\bs_i-\bz,\qquad
T_\varepsilon
=\varepsilon^{3/2}
(x_n^\varepsilon+h_n^\varepsilon-2x_\star^\varepsilon-s_i^\varepsilon).
\]
Here $x_\star^\varepsilon$ is the primordial coordinate of the
$\varepsilon$-dependent configuration. The scaling assumptions give
\[
T_\varepsilon\to \bt>0,\qquad
\varepsilon^{3/2}x_\star^\varepsilon\to 0,\qquad
\varepsilon^{1/2}
(x_n^\varepsilon+h_n^\varepsilon)=O(\varepsilon^{-1}).
\]
Choose $a$ such that
\[
a>\theta_m+|\brho|+1,\qquad
a>8\max_i\{\by_i\}/\bt.
\]
The parameter $a$ is a fixed contour shift, independent of
$\varepsilon,\bz$, and $\bv$. The constants below may depend
on $a$ and $\delta$, but not on \(\bz\) or \(\varepsilon\).

For $a<\varepsilon^{-1/2}$, let $\Gamma_\varepsilon$ be the negatively oriented circle
with center \((a+(\varepsilon a)^{-1})/2\) and radius \(R_\varepsilon=((\varepsilon a)^{-1}-a)/2\). Thus $a$ and $(\varepsilon a)^{-1}$ are the
left and right endpoints of a horizontal diameter. So,
\[
\zeta=a+R_\varepsilon-R_\varepsilon e^{-i\phi}, \quad -\pi \leq \phi \leq \pi.
\]
The circle encloses $\varepsilon^{-1/2}$ but excludes
$-\varepsilon^{-1/2}$ and $\pm\brho$.
Its useful feature is that
\begin{equation}\label{eqn:F-proof-circle-identities}
\left|\frac{1-\varepsilon^{1/2}\zeta}
{1+\varepsilon^{1/2}\zeta}\right|
=\frac{1-a\varepsilon^{1/2}}{1+a\varepsilon^{1/2}},
\qquad
|1-\varepsilon\zeta^2|
=(1-a^2\varepsilon)\frac{\operatorname{Re}\zeta}{a},
\qquad \zeta\in\Gamma_\varepsilon.
\end{equation}
To verify these identities, the equation of the circle gives
\[
|\zeta|^2
=\left(a+\frac1{\varepsilon a}\right)\operatorname{Re}\zeta
-\frac1\varepsilon.
\]
Multiplying by $\varepsilon$ and substituting gives
\[
\begin{split}
|1\pm\varepsilon^{1/2}\zeta|^2
&=1\pm2\varepsilon^{1/2}\operatorname{Re}\zeta
 +\left(\varepsilon a+\frac1a\right)\operatorname{Re}\zeta-1\\
&=\left(\frac1a\pm2\varepsilon^{1/2}+\varepsilon a\right)
  \operatorname{Re}\zeta
 =(1\pm a\varepsilon^{1/2})^2
  \frac{\operatorname{Re}\zeta}{a}.
\end{split}
\]
Taking the ratio and the product proves
\eqref{eqn:F-proof-circle-identities}. 

Whenever $\bu+\varepsilon^{-1}T_\varepsilon>0$, the index-shifting
identities in Lemmas~\ref{lem:aactonS}
and~\ref{lem:ainverseactonS} absorb $A_\varepsilon$, subtracting
$x_\star^\varepsilon$ from both indices, i.e., 
\[F_{\varepsilon,-i}= 2\varepsilon^{-1/2}\reflection  e^{2(x_n^\varepsilon+h_n^\varepsilon-2 x_\star^\varepsilon)D}\S^{r_{i,\varepsilon}'-x_\star^\varepsilon,l_{i,\varepsilon}'-x_\star^\varepsilon}.\]
This is because after the translations and
reflection, the microscopic sum of the two kernel variables is
\[
2\varepsilon^{-1/2}(\bv-\bz)
+2(x_n^\varepsilon+h_n^\varepsilon)-4x_\star^\varepsilon
=2\varepsilon^{-1/2}
 (\bu+\varepsilon^{-1}T_\varepsilon).
\]
When this sum is positive, the derivatives do not meet the support
boundary. After changing variables via $w=\varepsilon^{1/2}\zeta/2$ and shifting the contour to \(\Gamma_\varepsilon\), the
contour formulas become
\begin{equation}\label{eqn:F-proof-contours}
\begin{split}
F_{\varepsilon,-i}(\bz,\bv)
=\int_{\Gamma_\varepsilon}\frac{d\zeta}{2\pi i}
 e^{\Phi_\varepsilon(\zeta;-\by_i)},\quad F_{\varepsilon,i}(\bz,\bv)
=\int_{\Gamma_\varepsilon}\frac{d\zeta}{2\pi i}
 \frac{\zeta e^{\Phi_\varepsilon(\zeta;\by_i)}}
 {\zeta^2-\brho^2}
 -\frac12e^{\Phi_\varepsilon(\brho;\by_i)},
\end{split}
\end{equation}
where
\begin{equation}\label{eqn:F-proof-phase}
\Phi_\varepsilon(\zeta;\pm\by_i)
=\varepsilon^{-3/2}T_\varepsilon
\bigl(\operatorname{arctanh}(\varepsilon^{1/2}\zeta)
-\varepsilon^{1/2}\zeta\bigr)
\mp\varepsilon^{-1}\by_i\log(1-\varepsilon\zeta^2)
-\bu\zeta.
\end{equation}
The two shifted index pairs have
$n+m=\varepsilon^{-3/2}T_\varepsilon$ and
$m-n=\mp2\varepsilon^{-1}\by_i$ for the phase with argument
$\pm\by_i$. Combining these powers with
$e^{-(\varepsilon^{-1}T_\varepsilon+\bu)\zeta}$ gives
\eqref{eqn:F-proof-phase}.  The last term in
\eqref{eqn:F-proof-contours} is the residue at \(\zeta=\brho\), with the same
coefficient $-1/2$ under the prescribed half-residue convention
when $\brho=0$.

We now prove a pointwise bound for $\bz\leq\delta\varepsilon^{-1}$. If $\bz\leq\delta\varepsilon^{-1}$ and $\bv\geq0$, then
$\bu\geq\min_i\bs_i-\delta\varepsilon^{-1}
\geq-2\delta\varepsilon^{-1}$ for small $\varepsilon$. Since
$T_\varepsilon\to\bt$, take $\varepsilon_0$ small enough that
$\bt/2\leq T_\varepsilon\leq2\bt$. The choice
$\delta<\bt/8$ then gives
\[
T_\varepsilon+\varepsilon\bu
\geq\frac\bt2-2\delta\geq\frac\bt4.
\]
In particular, $\bu+\varepsilon^{-1}T_\varepsilon>0$, as required for
\eqref{eqn:F-proof-contours}. 
Recall that $\by_i\geq0$. We have the following estimates:
\begin{equation}\label{eqn:F-proof-circle-decay}
\begin{split}
&\operatorname{Re}\Phi_\varepsilon(\zeta;\pm\by_i)
-\Phi_\varepsilon(a;\pm\by_i)\\
&\quad=-\bigl(\varepsilon^{-1}T_\varepsilon+\bu\bigr)
 (\operatorname{Re}\zeta-a)
 \mp\varepsilon^{-1}\by_i
 \log\frac{\operatorname{Re}\zeta}{a}\\
&\quad\leq-\varepsilon^{-1}
 \left(T_\varepsilon+\varepsilon\bu-\frac{\by_i}{a}\right)
 (\operatorname{Re}\zeta-a)
 \leq-\varepsilon^{-1}\frac{\bt}{8}(\operatorname{Re}\zeta-a).
\end{split}
\end{equation}
The first equality follows from the two identities in \eqref{eqn:F-proof-circle-identities}. The first inequality is because \(\log\frac{\operatorname{Re}\zeta}{a}
=\int_a^{\operatorname{Re}\zeta}\frac{dr}{r}
\leq\frac{\operatorname{Re}\zeta-a}{a}\) for \(\operatorname{Re} \zeta\geq a\). By our second condition on \(a\), we have \(T_\varepsilon+\varepsilon\bu-\frac{\by_i}{a}
\geq\frac\bt4-\frac\bt8=\frac\bt8\).
We also need uniform bounds on the phase at the fixed real points
$a$ and $\brho$. For $\varepsilon r^2<1/2$, Taylor series estimates give
\[\varepsilon^{-3/2}\left|\operatorname{arctanh}(\varepsilon^{1/2}r)-\varepsilon^{1/2}r\right|\leq \frac{2}{3}\left|r\right|^3,\quad \varepsilon^{-1}\left|\log(1-\varepsilon r^2)\right|\leq 2r^2.\] 
Hence,
\[
\Phi_\varepsilon(a;\pm\by_i)\leq C_a-a\bu,
\qquad
\left|e^{\Phi_\varepsilon(\brho;\by_i)}\right|
\leq C_{\brho}e^{-\brho\bu}.
\]
The choice $ a^2\leq\varepsilon^{-1}/2$ also gives \(1/(4a)\leq\varepsilon R_\varepsilon
\leq 1/(2a)\).
In particular, $R_\varepsilon\asymp\varepsilon^{-1}$, and
\[
\begin{split}
\int_{\Gamma_\varepsilon}
 e^{-c\varepsilon^{-1}(\operatorname{Re}\zeta-a)}|d\zeta|
&=R_\varepsilon\int_{-\pi}^{\pi}
 e^{-c\varepsilon^{-1}R_\varepsilon(1-\cos\phi)}\,d\phi\\
&\leq R_\varepsilon\int_{\mathbb R}
 e^{-c'\varepsilon^{-1}R_\varepsilon\phi^2}\,d\phi
\leq C\sqrt{\varepsilon R_\varepsilon}\leq C.
\end{split}
\]
Here we used $1-\cos\phi\geq2\phi^2/\pi^2$ on $[-\pi,\pi]$.
The first condition on $a$ keeps the rational prefactor bounded uniformly on the same contour. Consequently,
\begin{equation}\label{eqn:F-proof-central-pointwise}
\begin{split}
|F_{\varepsilon,-i}(\bz,\bv)|\leq Ce^{-a\bu},\quad |F_{\varepsilon,i}(\bz,\bv)|
\leq C\bigl(e^{-a\bu}+e^{-\brho\bu}\bigr),
\qquad \bu\geq-2\delta\varepsilon^{-1}.
\end{split}
\end{equation}
After multiplying by the weights,
\[
\begin{split}
e^{\theta_i\bv}|F_{\varepsilon,-i}(\bz,\bv)|
&\leq Ce^{a\bz}e^{-(a-\theta_i)\bv},\\
e^{-\theta_i\bv}|F_{\varepsilon,i}(\bz,\bv)|
&\leq C\bigl(e^{a\bz}e^{-(a+\theta_i)\bv}
 +e^{\brho\bz}e^{-(\theta_i+\brho)\bv}\bigr).
\end{split}
\]
The fixed shifts $\bs_i$ are included in $C$. All three decay
rates are positive, since $a>\theta_i>|\brho|$. The bound is also uniform in \(i\) since there are finitely many indices. Thus,
\begin{equation}\label{eqn:F-proof-central-row}
W_{\varepsilon,i}(\bz)\leq Ce^{a|\bz|},
\qquad \bz\leq\delta\varepsilon^{-1}.
\end{equation}

It remains to estimate $W_{\varepsilon,i}(\bz)$ for $\bz>\delta\varepsilon^{-1}$.
For these rows the index shift need not be valid for every
$\bv\geq0$, so we cannot use \eqref{eqn:F-proof-contours}. We will prove that
\begin{equation}\label{eqn:F-proof-global-rough}
W_{\varepsilon,i}(\bz)
\leq C\exp\!\left\{
a|\bz|+C\varepsilon^{-1}\log(2/\varepsilon^{1/2})
\right\},
\end{equation}
which is actually true for all \(\bz \in \mathbb{R}\). To prove
\eqref{eqn:F-proof-global-rough}, we estimate the kernels before applying $A_\varepsilon$, and then
use a weighted operator bound for $A_\varepsilon$.

For either unshifted index pair
$(n,m)=(r_{i,\varepsilon}',l_{i,\varepsilon}')$ or
$(l_{i,\varepsilon}',r_{i,\varepsilon}')$, the scaling gives
\[
n+m=O(\varepsilon^{-3/2}),\qquad
|n-m|=O(\varepsilon^{-1}).
\]
On the same circle, \eqref{eqn:F-proof-circle-identities} implies
\[
\begin{split}
\log\left|
\frac{(1+\varepsilon^{1/2}\zeta)^m}
{(1-\varepsilon^{1/2}\zeta)^n}\right|
&=\frac{n+m}{2}
 \log\frac{1+a\varepsilon^{1/2}}{1-a\varepsilon^{1/2}}
 +\frac{m-n}{2}\log|1-\varepsilon\zeta^2|\\
&\leq C\varepsilon^{-1}\log(2/\varepsilon^{1/2}).
\end{split}
\]
Indeed, for fixed $a$ and $ a^2\leq \varepsilon^{-1}/2$,
$\log((1+\varepsilon^{1/2}a)/(1-\varepsilon^{1/2}a))
\leq C\varepsilon^{1/2}$. Its coefficient has size
$O(\varepsilon^{-3/2})$, so this part costs only
$O(\varepsilon^{-1})$. For the second part,
$a\leq\operatorname{Re}\zeta\leq(\varepsilon a)^{-1}$ gives
\[
1/2\leq 1-\varepsilon a^2\leq |1-\varepsilon\zeta^2|
\leq\frac{1}{\varepsilon a^2}.
\]
 The absolute value of its logarithm is
therefore at most $C\log(2/\varepsilon^{1/2})$; multiplying by
$|m-n|/2\leq C\varepsilon^{-1}$ gives the displayed loss.
This estimate does not use a cancellation between the contour
exponential and the large powers, which is why it remains valid
up to the support boundary.
In the original variable $w=\varepsilon^{1/2}\zeta/2$, the contour
has length $O(\varepsilon^{-1/2})$ and lies in
$\operatorname{Re}w\geq a\varepsilon^{1/2}/2$.
Directly estimating the defining contour integrals therefore gives,
for $r\geq0$,
\begin{equation}\label{eqn:F-proof-unshifted-pointwise}
\begin{split}
|\S^{n,m}(r,0)|
&\leq e^{C\varepsilon^{-1}\log(2/\varepsilon^{1/2})}
 e^{-a\varepsilon^{1/2}r/2},\\
|\S_H^{n,m}(r,0)|
&\leq e^{C\varepsilon^{-1}\log(2/\varepsilon^{1/2})}
 \bigl(e^{-a\varepsilon^{1/2}r/2}
 +e^{-\brho\varepsilon^{1/2}r/2}\bigr).
\end{split}
\end{equation}
For $\S_H$, at the pole \(\zeta=\brho\),
the logarithm of the power factor is
\[
(n+m)\operatorname{arctanh}(\varepsilon^{1/2}\brho)
+\frac{m-n}{2}\log(1-\varepsilon\brho^2).
\]
Its absolute value is at most $C\varepsilon^{-1}$.
Thus the boundary-residue coefficient is at most
$e^{C\varepsilon^{-1}}$. Both kernels retain their defining support $r\geq0$.

Fix $|\brho|<\beta<\theta_1$. Integrating
\eqref{eqn:F-proof-unshifted-pointwise} gives
\begin{equation}\label{eqn:F-proof-unshifted-weighted}
\begin{split}
&\lVert e^{\theta_i\varepsilon^{1/2}r/2}
 \S^{r_{i,\varepsilon}',l_{i,\varepsilon}'}(r,0)
 \rVert_{L^2(dr;\mathbb R_+)}\\
&\quad+
 \lVert e^{-\beta\varepsilon^{1/2}r/2}
 \S_H^{l_{i,\varepsilon}',r_{i,\varepsilon}'}(r,0)
 \rVert_{L^2(dr;\mathbb R_+)}
 \leq e^{C\varepsilon^{-1}\log(2/\varepsilon^{1/2})}.
\end{split}
\end{equation}

We next transfer these bounds to \(F_{\varepsilon,\pm i}\), which include
$A_\varepsilon$. In rescaled coordinates its multiplier is
\[
a_\varepsilon(w)
=e^{-2x_\star^\varepsilon\varepsilon^{1/2}w}
\left(\frac{1+\varepsilon^{1/2}w}
{1-\varepsilon^{1/2}w}\right)^{x_\star^\varepsilon}.
\]
Conjugation by $M_\lambda$ replaces the derivative by
$\bD-\lambda$. Thus, by Plancherel's theorem, the norm of
$M_\lambda A_\varepsilon M_{-\lambda}$ is the supremum of
$|a_\varepsilon(-\lambda+i\xi)|$ over $\xi\in\mathbb R$.
For a real number $\sigma$ with $|\sigma|<\varepsilon^{-1/2}$,
\[
\left|\frac{1+\varepsilon^{1/2}(\sigma+i\xi)}
{1-\varepsilon^{1/2}(\sigma+i\xi)}\right|^2
=\frac{(1+\varepsilon^{1/2}\sigma)^2+\varepsilon\xi^2}
{(1-\varepsilon^{1/2}\sigma)^2+\varepsilon\xi^2}.
\]
For $\sigma=-\theta_i$ this ratio is at most one. For
$\sigma=\beta>0$ it is largest at $\xi=0$. Since
$x_\star^\varepsilon\geq0$, we obtain the following operator bounds on
$L^2(\mathbb R)$:
\begin{equation}\label{eqn:F-proof-weighted-A}
\begin{split}
\lVert M_{\theta_i}A_\varepsilon M_{-\theta_i}\rVert_{L^2}
&\leq e^{2\theta_i\varepsilon^{1/2}x_\star^\varepsilon},\\
\lVert M_{-\beta}A_\varepsilon M_\beta\rVert_{L^2}
&\leq\exp\!\left\{x_\star^\varepsilon\left[
 \log\frac{1+\beta\varepsilon^{1/2}}
 {1-\beta\varepsilon^{1/2}}
 -2\beta\varepsilon^{1/2}\right]\right\}
 \leq e^{C\varepsilon^{3/2}x_\star^\varepsilon}.
\end{split}
\end{equation}
The last inequality is true for $\varepsilon\beta^2\leq1/2$.

Notice that the \(S,S_H\) kernels are Hankel kernels: they
depend only on the sum of their variables, and $A_\varepsilon$
is translation invariant. Taking account of the reflection operator, we have
\[\lVert(\reflection A_\varepsilon S_\varepsilon)(\bz,\cdot)M_{\lambda}\rVert_{L^2(\mathbb{R})} \leq \lVert M_{\lambda}A_\varepsilon M_{-\lambda}\rVert_{L^2(\mathbb{R})}\lVert(M_{\lambda}\reflection S_\varepsilon(\bz,\cdot))\rVert_{L^2(\mathbb{R})}.\]
Here the reflection replaces the first argument by its
negative. Using \eqref{eqn:F-proof-weighted-A} and taking into account the change-of-variable factors gives
\[
\begin{split}
\lVert F_{\varepsilon,-i}(\bz,\cdot)M_{\theta_i}\rVert_{L^2(\mathbb R)}
&\leq\sqrt{2\varepsilon^{-1/2}}\,
 e^{\theta_i(\bz-\varepsilon^{1/2}
(x_n^\varepsilon+h_n^\varepsilon))
 +2\theta_i\varepsilon^{1/2}x_\star^\varepsilon}\\
&\qquad\cdot
 \lVert e^{\theta_i\varepsilon^{1/2}r/2}
 \S^{r_{i,\varepsilon}',l_{i,\varepsilon}'}(r,0)
 \rVert_{L^2(dr;\mathbb R_+)}\\
&\leq
 e^{\theta_i\bz+C\varepsilon^{-1}\log(2/\varepsilon^{1/2})},\\[1mm]
\lVert F_{\varepsilon,i}(\bz,\cdot)M_{-\beta}\rVert_{L^2(\mathbb R)}
&\leq\sqrt{8\varepsilon^{1/2}}\,
 e^{\beta(\varepsilon^{1/2}
(x_n^\varepsilon+h_n^\varepsilon)-\bz)+C\varepsilon^{3/2}x_\star^\varepsilon}\lVert e^{-\beta\varepsilon^{1/2}r/2}
 \S_H^{l_{i,\varepsilon}',r_{i,\varepsilon}'}(r,0)
 \rVert_{L^2(dr;\mathbb R_+)}\\
&\leq
 e^{-\beta\bz+C\varepsilon^{-1}\log(2/\varepsilon^{1/2})}.
\end{split}
\]
The norms of \(\S,\S_H\) are taken in \(L^2(dr;\mathbb{R}_+)\) because \(\S(x,y),\S_H(x,y)\) are Hankel and have the indicator functions \(1_{x+y\geq0}\). The last inequalities follow from \eqref{eqn:F-proof-unshifted-weighted} and the bounds
$|\varepsilon^{1/2}
(x_n^\varepsilon+h_n^\varepsilon)|\leq C\varepsilon^{-1}$ and
$\varepsilon^{1/2} x_\star^\varepsilon=o(\varepsilon^{-1})$.

Restricting to $\bv\geq0$ and using $a>\theta_m>\beta$ in the two preceding inequalities proves
\eqref{eqn:F-proof-global-rough}.

This proves the two discrete bounds in
\eqref{eqn:firstSbound-improved}. 

\smallskip\noindent\emph{The limit and weighted $L^2$ convergence.}
It remains to identify the limiting kernels and justify passage
to the weighted row norms. With $\xi=R_\varepsilon\phi$,
write the circle \(\Gamma_\varepsilon\) as
\[
\zeta_\varepsilon(\xi)
=a+R_\varepsilon(1-\cos(\xi/R_\varepsilon))
+iR_\varepsilon\sin(\xi/R_\varepsilon),
\qquad |\xi|\leq\pi R_\varepsilon.
\]
For each fixed $\xi$,
\[
\zeta_\varepsilon(\xi)\to a+i\xi,
\qquad
\zeta_\varepsilon'(\xi)\to i,
\qquad |\zeta_\varepsilon'(\xi)|=1.
\]
Moreover,
\[
\varepsilon^{-1}
(\operatorname{Re}\zeta_\varepsilon(\xi)-a)
\geq\frac{2\xi^2}{\pi^2\varepsilon R_\varepsilon}
\geq\frac{4a}{\pi^2}\xi^2,
\qquad |\xi|\leq\pi R_\varepsilon.
\]
Thus \eqref{eqn:F-proof-circle-decay} gives a uniform Gaussian upper bound on compact sets of $(\bz,\bv)$. On bounded sets of
$\zeta$, the Taylor expansions are
\[
\begin{split}
\varepsilon^{-3/2}
\bigl(\operatorname{arctanh}(\varepsilon^{1/2}\zeta)
-\varepsilon^{1/2}\zeta\bigr)
&=\frac{\zeta^3}{3}+O(\varepsilon),\\
-\varepsilon^{-1}\log(1-\varepsilon\zeta^2)
&=\zeta^2+O(\varepsilon).
\end{split}
\]
Together with $T_\varepsilon\to\bt$, they give, locally uniformly,
\[
\Phi_\varepsilon(\zeta;\pm\by_i)
\longrightarrow
\frac{\bt}{3}\zeta^3\pm\by_i\zeta^2-\bu\zeta.
\]
One can first restrict $\xi$ to a bounded interval, where the
parametrizations and integrands converge uniformly, and then
remove that restriction using the Gaussian upper bound. This also
justifies convergence uniformly in $(\bz,\bv)$ on compact sets.
Dominated convergence in \eqref{eqn:F-proof-contours}, including
the boundary residue, yields locally uniformly
\[
\begin{split}
F_{\varepsilon,-i}(\bz,\bv)
&\longrightarrow
\bS^{\bt,-\by_i}(\bz,\bv+\bs_i)=F_{-i}(\bz,\bv),\\
F_{\varepsilon,i}(\bz,\bv)
&\longrightarrow
-\bS_{\bH}^{\bt,\by_i}(\bz,\bv+\bs_i)=F_i(\bz,\bv).
\end{split}
\]
Now we derive the bound on \(W_i\). For the limiting contour $a+i\mathbb R$,
\[
\begin{split}
&\operatorname{Re}\left[
\frac{\bt}{3}(a+i\xi)^3
\pm\by_i(a+i\xi)^2-\bu(a+i\xi)\right]=\frac{\bt a^3}{3}\pm\by_i a^2-a\bu
-(\bt a\pm\by_i)\xi^2.
\end{split}
\]
The same parameter choice of \(a\) gives
$\bt a\pm\by_i>7\bt a/8$.
Thus the limiting vertical-contour integral has a uniform Gaussian
upper bound as well. Hence the bounds in
\eqref{eqn:F-proof-central-pointwise} hold for every $\bu\in\mathbb R$. Their weighted
integration proves $W_i(\bz)\leq Ce^{a|\bz|}$.

Finally, for $\bz$ in a fixed compact set $K$, the
pointwise bounds apply for every $\bv\geq0$ once $\varepsilon$
is sufficiently small. Since $a>\theta_m$ and
$\theta_1+\brho>0$, set
\[
c_*:=\min_i\{a-\theta_i,a+\theta_i,\theta_i+\brho\}>0.
\]
Each squared weighted exponential has tail integral at most a
constant times $e^{-2c_*R}$. Consequently, uniformly in such
$\varepsilon$,
\[
\begin{split}
\sup_{\bz\in K}\int_R^\infty
\Bigl[&e^{2\theta_i\bv}
\bigl(|F_{\varepsilon,-i}(\bz,\bv)|^2
+|F_{-i}(\bz,\bv)|^2\bigr)\\
&+e^{-2\theta_i\bv}
\bigl(|F_{\varepsilon,i}(\bz,\bv)|^2
+|F_i(\bz,\bv)|^2\bigr)\Bigr]d\bv
\leq C_K e^{-c_*R}.
\end{split}
\]
On $K\times[0,R]$, the weights are bounded and uniform kernel
convergence gives convergence of the weighted $L^2$ integrals; thus we have proved \eqref{eqn:firstSbound-L2convergence}.
\end{proof}

\begin{lemma}\label{lem:secondSbound}
Recall the kernels \(g_{\varepsilon,\pm j,k}(\bz_1,\bz_2)\) and their limits, defined in \eqref{eqn:local-hitting-integrand} and \eqref{eqn:local-hitting-integrand-limit}.
For every $\mu>0$, there exist $c,C,\varepsilon_0>0$ such that,
for $0<\varepsilon<\varepsilon_0$, every $j,k$, and
$\bz\leq\bh_k^\varepsilon$,
\begin{equation}\label{eqn:secondSbound}
\begin{split}
&\lVert g_{\varepsilon,j,k}(\bz,\cdot)M_{-\theta_j}\rVert_2
+\lVert g_{\varepsilon,-j,k}(\bz,\cdot)M_{\theta_j}\rVert_2 \leq
C\exp\{-c(\bx_k^\varepsilon)^3-\mu\bz_-\}.
\end{split}
\end{equation}
The limiting kernels satisfy
\[
\lVert g_{j,\bx}(\bz,\cdot)M_{-\theta_j}\rVert_2
+\lVert g_{-j,\bx}(\bz,\cdot)M_{\theta_j}\rVert_2
\leq Ce^{-c\bx^3-\mu\bz_-},
\qquad
\bx\geq0,\quad \bz\leq\alpha+\gamma\bx,
\]
where $\alpha,\gamma$ are from
\eqref{eqn:common-linear-upper-bound}.
Moreover, for every compact $K\subset[0,\infty)\times\mathbb R$,
\begin{equation}\label{eqn:hitting-integrand-L2convergence}
\begin{split}
\sup_{\substack{
1\leq k\leq n_\varepsilon,\ \bz\leq\bh_k^\varepsilon\\
(\bx_k^\varepsilon,\bz)\in K}}
\Bigl\{
&\lVert[g_{\varepsilon,j,k}(\bz,\cdot)
-g_{j,\bx_k^\varepsilon}(\bz,\cdot)]M_{-\theta_j}\rVert_2\\
+{}&\lVert[g_{\varepsilon,-j,k}(\bz,\cdot)
-g_{-j,\bx_k^\varepsilon}(\bz,\cdot)]M_{\theta_j}\rVert_2
\Bigr\}
\to 0.
\end{split}
\end{equation}
\end{lemma}

\begin{proof}
In this proof, write $\bx=\bx_k^\varepsilon$, $\bh=\bh_k^\varepsilon$, and set
\[
T_\varepsilon=\bt+\varepsilon(\bh-\bs_j),\qquad
\bu=\bv+\bs_j-\bz.
\]

\smallskip\noindent\emph{The contour and its decay.}
The defining contour formulas, after $w=\varepsilon^{1/2}\,\zeta/2$,
are
\begin{equation}\label{eqn:g-proof-contours}
\begin{split}
g_{\varepsilon,j,k}(\bz,\bv)
&=\int_{C_{\varepsilon^{-1/2}}}\frac{d\zeta}{2\pi i}
 e^{\Phi_\varepsilon(\zeta;\bx-\by_j)},\\
g_{\varepsilon,-j,k}(\bz,\bv)
&=\int_{C_{\varepsilon^{-1/2}}}\frac{d\zeta}{2\pi i}
 \frac{\zeta^2-\brho^2}{\zeta}
 e^{\Phi_\varepsilon(\zeta;\bx+\by_j)},
\end{split}
\end{equation}
where the contour is clockwise and encloses only
$\varepsilon^{-1/2}$, and
\begin{equation}\label{eqn:g-proof-phase}
\Phi_\varepsilon(\zeta;X)
=\frac{T_\varepsilon}{\varepsilon^{3/2}}
 \bigl(\operatorname{arctanh}(\varepsilon^{1/2}\,\zeta)
       -\varepsilon^{1/2}\,\zeta\bigr)
 +\frac X\varepsilon\log(1-\varepsilon\zeta^2)-\bu\zeta.
\end{equation}
The restriction $\bz\leq\bh$ gives
\begin{equation}\label{eqn:g-proof-support}
T_\varepsilon+\varepsilon\bu
=\bt+\varepsilon(\bh-\bz+\bv)
\geq\bt+\varepsilon\bv.
\end{equation}
Thus the support indicator is one and no support-boundary term
arises.

For either choice of $X$, the integral vanishes for $T_\varepsilon\leq 2\varepsilon^{1/2} X$, since there is no pole present. So we assume $T_\varepsilon>2\varepsilon^{1/2} X$.
Using the linear bound $\bh\leq\alpha+\gamma\bx$ and
$|X-\bx|\leq\by_m$, we have
\[
T_\varepsilon\leq\bt+\varepsilon(C+\gamma X),
\qquad
(2\varepsilon^{1/2}-\gamma\varepsilon)X
<\bt+C\varepsilon.
\]
Here $C$ absorbs the fixed observation heights and positions. For $X>0$, take
$\gamma\varepsilon^{1/2}\leq1$ and divide to obtain
\[
\varepsilon^{1/2} X
<\frac{\bt+C\varepsilon}{2-\gamma\varepsilon^{1/2}}
\leq\frac\bt2+C\varepsilon^{1/2}.
\]
For $X\leq0$ the same upper bound is automatic for $X_+=0$.
Next, $\bx\leq X_++\by_m$ implies
$T_\varepsilon\leq\bt+C\varepsilon^{1/2}$ by the linear upper
bound and
$T_\varepsilon>2\varepsilon^{1/2} X\geq-2\by_m\varepsilon^{1/2}$ gives the lower bound. Thus, we have proved
\[
\varepsilon^{1/2} X_+\leq\frac\bt2+C\varepsilon^{1/2},
\qquad
-C\varepsilon^{1/2}<T_\varepsilon\leq\bt+C\varepsilon^{1/2}.
\]
Set $\kappa>\mu+\theta_m+|\brho|+1$ and \(a=\max\left\{\kappa,\frac{3X_+}{2\bt}\right\}\). Using the preceding bounds and picking \(\varepsilon_0\) small enough, we have
$a\varepsilon^{1/2}\leq7/8$ and $T_\varepsilon\leq7\bt/6$ uniformly over the nonzero integrals. This ensures
\[
1-\varepsilon r^2\geq1-\varepsilon a^2\geq\frac{15}{64},
\qquad 0\leq r\leq a.
\]
The coefficient $3/2$ in the definition of $a$ also ensures \(
T_\varepsilon a\leq
\frac{7X}{4}<2X\text{ if }a>\kappa\).

We move the contour to $\Gamma_\varepsilon$, defined as in the preceding proof: the clockwise circle with center
$(a+(a\varepsilon)^{-1})/2$ and radius \(R_\varepsilon=((a\varepsilon)^{-1}-a)/2\).
It encloses $\varepsilon^{-1/2}$ but excludes zero and
$-\varepsilon^{-1/2}$.
On this circle,
\[
\left|\frac{1-\varepsilon^{1/2}\,\zeta}
               {1+\varepsilon^{1/2}\,\zeta}\right|
=\frac{1-\varepsilon^{1/2}a}{1+\varepsilon^{1/2}a},
\qquad
|1-\varepsilon\zeta^2|
=(1-\varepsilon a^2)\frac{\operatorname{Re}\zeta}{a}.
\]
Consequently, using \eqref{eqn:g-proof-support} and
$X_+/a\leq2\bt/3$,
\begin{equation}\label{eqn:g-circle-transverse-decay}
\begin{split}
\operatorname{Re}\Phi_\varepsilon(\zeta;X)
-\Phi_\varepsilon(a;X)
&=-\left(\varepsilon^{-1}T_\varepsilon+\bu\right)
  (\operatorname{Re}\zeta-a)
  +\varepsilon^{-1}X\log\frac{\operatorname{Re}\zeta}{a}\\
&\leq-\varepsilon^{-1}\frac{\bt}{3}
 (\operatorname{Re}\zeta-a).
\end{split}
\end{equation}

At the real point $a$, the part of the phase independent of $\bu$ is
\[
\Phi_\varepsilon(a;X)+a\bu
=\int_0^a\frac{r(T_\varepsilon r-2X)}{1-\varepsilon r^2}\,dr.
\]
If $a=\kappa$, both $X$ and $T_\varepsilon$ are bounded. These bounds and the bound on the denominator show that the
integral is bounded above. Otherwise $a=3X/(2\bt),X>0$ and
$T_\varepsilon a\leq7X/4$; the numerator is nonpositive and hence
\[
\Phi_\varepsilon(a;X)+a\bu
\leq\frac{T_\varepsilon a^3}{3}-Xa^2
\leq-\frac5{12}Xa^2.
\]
Since $|X-\bx|\leq\by_m$, in both cases,
\begin{equation}\label{eqn:g-circle-real-phase}
\Phi_\varepsilon(a;X)\leq C-c\bx^3-a\bu.
\end{equation}

To integrate \eqref{eqn:g-circle-transverse-decay}, parameterize the
circle by
\[
\zeta_\varepsilon(\xi)
=a+R_\varepsilon\bigl(1-\cos(\xi/R_\varepsilon)\bigr)
+iR_\varepsilon\sin(\xi/R_\varepsilon),
\qquad |\xi|\leq\pi R_\varepsilon.
\]
Here $|d\zeta|=d\xi$, $|\zeta|\leq a+|\xi|$, and \(\varepsilon^{-1}(\operatorname{Re}\zeta_\varepsilon(\xi)-a)
\geq c a\xi^2\) for some \(c>0\).
Both prefactors in \eqref{eqn:g-proof-contours} are bounded by
$C(1+|\zeta|)$. Thus the contour integral is bounded by
\[
C e^{\Phi_\varepsilon(a;X)}
\int_{\mathbb R}(1+a+|\xi|)e^{-ca\xi^2}\,d\xi
\leq C(1+a)e^{-c\bx^3-a\bu}.
\]

\smallskip\noindent\emph{Bounds and convergence.}
Since $\bz_+\leq C(1+\bx)$ on the support and
$\mu<\kappa\leq a\leq C(1+\bx)$,
\[
\log(1+a)-a\bu
\leq C(1+\bx)^2-\mu\bz_--\kappa\bv.
\]
Absorbing the quadratic term into the cubic decay proves the
stronger pointwise bound
\begin{equation}\label{eqn:g-proof-pointwise}
|g_{\varepsilon,\pm j,k}(\bz,\bv)|
\leq C e^{-c\bx^3-\mu\bz_--\kappa\bv},
\qquad \bz\leq\bh,\quad\bv\geq0.
\end{equation}
Integrating against either prescribed squared weight gives
\eqref{eqn:secondSbound} because $\kappa>\theta_m$.

For the limiting kernel, use the vertical contour $a+i\mathbb R$, with the same choice of parameter \(a\) and
phase $\Phi_0(\zeta;X)=\bt\zeta^3/3-X\zeta^2-\bu\zeta$.
Its value at $a$ satisfies the same bound as
\eqref{eqn:g-circle-real-phase}, while
\[
\operatorname{Re}\Phi_0(a+i\xi;X)-\Phi_0(a;X)
=-(\bt a-X)\xi^2\leq-\frac{\bt a}{3}\xi^2.
\]
Thus, the preceding Gaussian integration gives the same
pointwise bounds whenever $\bz\leq\alpha+\gamma\bx$.

Finally, on a compact set \(K\) of $(\bx,\bz)$ with $\bz\leq\bh$,
the bounds $\bz\leq\bh\leq\alpha+\gamma\bx$ give
$T_\varepsilon=\bt+O_K(\varepsilon)$ uniformly over all
selected peak indices. No convergence of an individual labeled
peak is needed. 
Also $\zeta_\varepsilon(\xi)\to a+i\xi$ and
$d\zeta_\varepsilon/d\xi\to i$ uniformly for bounded $\xi$.
The phase converges to $\Phi_0$ there, and the Gaussian bound above
supplies an integrable upper bound uniform on compact sets of $\bv$.
Dominated convergence in \eqref{eqn:g-proof-contours} identifies
the limits in \eqref{eqn:local-hitting-integrand-limit}; the minus
sign for $g_{-j}$ follows from
$(\zeta^2-\brho^2)/\zeta=-(\brho^2-\zeta^2)/\zeta$.
This convergence is uniform for $\bv$ in bounded intervals.
The pointwise bound also
gives uniformly on each such compact set
\[
\int_R^\infty e^{2\theta_m\bv}
\bigl(|g_{\varepsilon,\pm j,k}(\bz,\bv)|^2
      +|g_{\pm j,\bx}(\bz,\bv)|^2\bigr)\,d\bv
\leq C_K e^{-2(\kappa-\theta_m)R}.
\]
Combining this tail estimate with convergence on $[0,R]$ proves
\eqref{eqn:hitting-integrand-L2convergence}.
\end{proof}

\begin{lemma}\label{lem:revised-hitting-tail}
For fixed $\eta,\mu>0$, set
\begin{equation}\label{eqn:uniformhittingfunctional}
\begin{split}
Q_\varepsilon^{\eta,\mu}(\bz)
&=\int e^{-\eta\bx^3-\mu(-\bb)_+}
\,\nu_{\bz}^\varepsilon(d\bx,d\bb),\\
Q^{\eta,\mu}(\bz)
&=\int e^{-\eta\bx^3-\mu(-\bb)_+}
\,\nu_{\bz}(d\bx,d\bb),
\end{split}
\end{equation}
where $\nu^\varepsilon$ and $\nu$ are the discrete and Brownian
first-hitting measures, respectively.
For every $0<\lambda<\mu$ and $0<\delta$, there exist $c,C,\varepsilon_0>0$ such
that, for $\bz\in\mathbb R$ and
$0<\varepsilon<\varepsilon_0$,
\begin{equation}\label{eqn:uniformhittingtail}
\begin{split}
Q_\varepsilon^{\eta,\mu}(\bz)
&\leq
C\exp\!\left\{
-\lambda\bz_-
-\lambda \bz_+-c\varepsilon^{-3/2}1_{\bz>\delta \varepsilon^{-1}}
\right\},\\
Q^{\eta,\mu}(\bz)
&\leq
C\exp\{-\lambda\bz_--c\bz_+^{3/2}\}.
\end{split}
\end{equation}
\end{lemma}
\begin{proof}
The walk has time spacing $\varepsilon/2$. 
At every finite hitting time,
\[
 B^\varepsilon(\btau^\varepsilon)
 \leq\alpha+\gamma\btau^\varepsilon,
\]
and the same inequality holds for the Brownian hitting pair.
All constants below are independent of the starting point.

For either possible increment
$\Delta B^\varepsilon=\pm\varepsilon^{1/2}(1-E)$,
where $E$ has the mean-one exponential distribution,
\[
 \log\mathbb E e^{\ell\Delta B^\varepsilon}
 =\pm\varepsilon^{1/2}\ell
  -\log(1\pm\varepsilon^{1/2}\ell)
 \leq\varepsilon\ell^2,
 \qquad |\varepsilon^{1/2}\ell|\leq\tfrac12.
\]
Thus $\exp\{\pm\ell B^\varepsilon(s)-2\ell^2s\}$ is a
nonnegative supermartingale for \(s\in (\varepsilon/2)\mathbb{Z}_{\geq 0}\) whenever
$0\leq\varepsilon^{1/2}\ell\leq1/2$. Applying the supermartingale inequality at the bounded stopping time gives
\begin{equation}\label{eqn:uniform-stopped-martingale}
 \mathbb E_{\bz}\left[
 e^{\pm\ell B^\varepsilon(\btau^\varepsilon)
        -2\ell^2\btau^\varepsilon};\,
 \btau^\varepsilon<\infty\right]
 \leq e^{\pm\ell\bz}.
\end{equation}
For Brownian motion with diffusion coefficient $2$, the analogous
bound holds for every $\ell\geq0$, without the restriction
$\varepsilon^{1/2}\ell\leq1/2$.

First suppose that $\bz\leq0$. Fix the $\lambda\in(0,\mu)$
from the statement and decrease $\varepsilon_0$ so that
$\varepsilon^{1/2}\lambda\leq1/2$. For every $\bb\in\mathbb R$,
$e^{-\mu\bb_-}\leq e^{\lambda\bb}$. Hence
\[
\begin{split}
 Q_\varepsilon^{\eta,\mu}(\bz)
 &\leq
 \sup_{s\geq0}e^{-\eta s^3+2\lambda^2s}
 \mathbb E_{\bz}\left[
 e^{\lambda B^\varepsilon(\btau^\varepsilon)
       -2\lambda^2\btau^\varepsilon};\,
 \btau^\varepsilon<\infty\right] \leq C e^{\lambda\bz}=C e^{-\lambda\bz_-}.
\end{split}
\]
The supremum is finite because of its negative cubic term. This proves the negative-side estimate.

Now suppose that $\bz\geq0$. Dropping the \(-\mu(B^{\varepsilon}(\btau^\varepsilon))_-\) term
and using the upper bound on the hitting position, we have for 
$0\leq\varepsilon^{1/2}\ell\leq1/2$,
\begin{equation}\label{eqn:expanded-tail-optimization}
\begin{split}
 Q_\varepsilon^{\eta,\mu}(\bz)
 &\leq e^{-\ell\bz+\alpha\ell}
 \sup_{s\geq0}
 e^{-\eta s^3+(\gamma\ell+2\ell^2)s}\leq\exp\{-\ell\bz+C(\ell+\ell^3)\}.
\end{split}
\end{equation}
The last inequality is true because
maximizing the cubic polynomial gives
\[
 \sup_{s\geq0}\{-\eta s^3+As\}
 =\frac{2}{3\sqrt{3\eta}}A^{3/2},\qquad A\geq0.
\]
With $A=\gamma\ell+2\ell^2$, this is bounded by
$C(\ell+\ell^3)$, since
$\ell^{3/2}\leq\ell+\ell^3$.

Choose a fixed $d>0$ such that $d\leq1/2$ and $Cd^2\leq\delta/4$,
where $C$ is the constant in \eqref{eqn:expanded-tail-optimization},
and set
\[
 \ell=d\varepsilon^{-1/2}.
\]
For \(\bz\geq \delta\varepsilon^{-1}\),
\[C\ell^3 \leq \ell\bz/4,\quad C\ell \leq C\varepsilon\ell \bz/\delta\leq \ell\bz/4 \text{ by making }\varepsilon_0 \text{ small. }\]
Again, decreasing \(\varepsilon_0\) to ensure \(\varepsilon^{-1/2}d/4\geq \lambda\), we have 
\[Q^{\eta,\mu}_{\varepsilon}(\bz)\leq \exp\{-\varepsilon^{-1/2}d\bz/2\}\leq \exp\{-\lambda\bz-\frac{d\delta}{4}\varepsilon^{-3/2}\}.\]
For $0\leq\bz\leq\delta\varepsilon^{-1}$, taking
$\ell=\lambda$ in \eqref{eqn:expanded-tail-optimization} gives
$Q_\varepsilon^{\eta,\mu}(\bz)\leq
e^{C(\lambda+\lambda^3)}e^{-\lambda\bz}$,
uniformly for sufficiently small $\varepsilon$. For Brownian motion the
same argument allows $\ell=d\sqrt\bz$ for every $\bz\geq0$,
and gives $Ce^{-c\bz^{3/2}}$. Combining the two sides gives \eqref{eqn:uniformhittingtail}.
\end{proof}

        \begin{lemma}\label{lem:randomwalktraceconverge}
            Recall that  $u'_{ij} =(y_{j}-y_i-s_{j}+s_i)/2, d'_{ij} = (y_j-y_i+s_j-s_i)/2$.
             \begin{equation}\label{eqn:randomwalkptwiseconverge}
           \begin{split}
            \normalfont 
                   2\varepsilon^{-1/2}M_{\alpha}(\mtp)^{-u'_{ij}}(\ptp)^{-d'_{ij}}M_{\beta} \to M_{\alpha}e^{\bs_i \bD}e^{(\by_j-\by_i)\bD^2}e^{-\bs_j \bD}M_{\beta} \text{ for } i<j, \\
                   \normalfont 
                   2\varepsilon^{-1/2}M_{\alpha}(\ptp)^{-u'_{ji}}(\mtp)^{-d'_{ji}}M_{\beta} \to M_{\alpha}e^{\bs_i \bD}e^{(\by_i-\by_j)\bD^2}e^{-\bs_j \bD}M_{\beta} \text{ for } i>j.
               \end{split}    
        \end{equation}
        The convergence is in trace norm on $L^2([0,\infty))$.
        \end{lemma}
        \begin{proof}
            Let $P^\varepsilon(x-y)$ be the convolution kernel of $2\varepsilon^{-1/2}(\mtp)^{-u'_{ij}}(\ptp)^{-d'_{ij}}(x,y)$. Let $P$ be the convolution kernel of $e^{\bs_i\bD}
e^{(\by_j-\by_i)\bD^2}
e^{-\bs_j\bD}$.
            Using the Fourier transform, one can see \[
\widehat P^\varepsilon(\xi+i\lambda)
\to \widehat P(\xi+i\lambda)
 \text{ in } L^2.
\]
Thus, Fourier
inversion and Plancherel's theorem give, for every fixed
$\lambda\in\mathbb R$,
\[
\int_{\mathbb R}
|P^\varepsilon(r)-P(r)|^2e^{2\lambda r}\,dr
\to 0.
\]
Now we compute that 
\begin{equation}
                    \lVert M_{\alpha}(P^{\varepsilon}-P)M_{\beta} \rVert_2^2 = \int_{\mathbb{R}}|P^\varepsilon(r)-P(r)|^2e^{2\alpha r}dr\int_{\max\{0,-r\}}^\infty e^{2(\alpha+\beta)x}dx.
            \end{equation}
            The integral in $x$ is finite since $\alpha+\beta<0$, and the whole integral converges to $0$. 

To obtain trace-norm convergence, split each step count into two
integer parts differing by at most one, giving
$P^\varepsilon=P_1^\varepsilon P_2^\varepsilon$, with Gaussian
limits satisfying $P=P_1P_2$. Choose $\beta<\kappa<-\alpha$ and set
\[
U_\varepsilon=1_0M_\alpha P_1^\varepsilon M_\kappa,
\qquad
V_\varepsilon=M_{-\kappa}P_2^\varepsilon M_\beta1_0,
\]
with $U,V$ defined analogously. The preceding weighted Fourier
argument applies to each factor and, since
$\alpha+\kappa<0$ and $\beta-\kappa<0$, gives
$U_\varepsilon\to U$ and $V_\varepsilon\to V$ in Hilbert--Schmidt
norm, as operators from $L^2(\mathbb R)$ to $L^2(\mathbb R_+)$
and from $L^2(\mathbb R_+)$ to $L^2(\mathbb R)$, respectively.
Consequently, on $L^2(\mathbb R_+)$,
\[
\lVert M_\alpha(P^\varepsilon-P)M_\beta\rVert_1
\leq
\lVert U_\varepsilon-U\rVert_2\lVert V_\varepsilon\rVert_2
+\lVert U\rVert_2\lVert V_\varepsilon-V\rVert_2
\longrightarrow0.
\]
        \end{proof}
\begin{proposition}
\label{prop:freeHfreeuniformHSbound}
There exist
$\varepsilon_0>0$ and $C<\infty$ such that
\begin{equation}
\label{eqn:freeHfreeuniformHSbound}
\sup_{0<\varepsilon<\varepsilon_0}
\left\lVert
M_{-\theta_i}
R^\varepsilon_{12}(i,j)
M_{-\theta_j}
\right\rVert_2
\leq C.
\end{equation}
\end{proposition}

\begin{proof}
Let \(c_\varepsilon
:=
\varepsilon^{-1/2}(\by_i+\by_j)\) and
\[
N_-^\varepsilon
:=
l_j'-r_i'
=
\varepsilon^{-1/2} c_\varepsilon
+\frac12\varepsilon^{-1/2}(\bs_i-\bs_j),
\quad
N_+^\varepsilon
:=
l_i'-r_j'
=
\varepsilon^{-1/2} c_\varepsilon
-\frac12\varepsilon^{-1/2}(\bs_i-\bs_j).
\]

Recall that \(R_{12}^\varepsilon(i,j) = -4\mtp^{-N_-^\varepsilon}H_{\alpha^\varepsilon}\ptp^{-N_+^\varepsilon}\).  The explicit formula
for $H_{\alpha^\varepsilon}$ gives
\[
    |\bH(\bz_2,\bz_3)|
    \leq
    C e^{|\brho||\bz_2-\bz_3|},
\]
uniformly in $\varepsilon$. Since \(\mtp^{-N},\ptp^{-N}\) are the transition kernels of a sum of \(N\) exponential random variables, after suitable centering, using the bounds on their moment-generating functions, we have for every fixed $\lambda>0$,
\begin{equation}\label{eqn:free-transition-exponential-moments}
\begin{split}
\sup_{0<\varepsilon<\varepsilon_0}
2\varepsilon^{-1/2}\int_{\mathbb R}
e^{\lambda|\bz_2-\bz_1|}\mtp^{-N_-^\varepsilon}e^{-2\varepsilon^{-1/2}c_\varepsilon D}(z^\varepsilon_1,z^\varepsilon_2)\,d\bz_2 \leq C_\lambda,\\
\sup_{0<\varepsilon<\varepsilon_0}2\varepsilon^{-1/2}
\int_{\mathbb R}
e^{\lambda|\bz_3-\bz_4|}e^{2\varepsilon^{-1/2}c_\varepsilon D}
\ptp^{-N_+^\varepsilon}(z^\varepsilon_3,z^\varepsilon_4)\,d\bz_3 \leq C_\lambda.
\end{split}
\end{equation}
Notice that the operator \(H_{\alpha^\varepsilon}\) is invariant under translation. Thus, taking $\lambda=|\brho|$ and combining this with the bound on \(\bH\), we obtain
\begin{align}
\left|
R_{12}^\varepsilon(i,j)(\bz_1,\bz_4)
\right|
&\leq
C e^{|\brho||\bz_1-\bz_4|},
\qquad
\bz_1,\bz_4\geq0,
\label{eqn:freeHfreekernelbound}
\end{align}
uniformly for all sufficiently small $\varepsilon$. Notice that the same bound is also true when \(N_-^\varepsilon = N_+^\varepsilon=0\).

It follows that
\begin{align*}
\left\lVert
M_{-\theta_i}
R_{12}^\varepsilon(i,j)
M_{-\theta_j}
\right\rVert_2^2
&\leq
C\int_0^\infty\int_0^\infty
e^{-2\theta_i\bz_1-2\theta_j\bz_4}
e^{2|\brho||\bz_1-\bz_4|}
\,d\bz_1\,d\bz_4<\infty.
\end{align*}
\end{proof}

\subsection{Convergence of the Fredholm Pfaffian}
\label{subsec:Fredholm-convergence}
For a two-by-two block operator, write
\[
\lVert K\rVert_{1|2}=\lVert K_{11}\rVert_1+\lVert K_{22}\rVert_1
             +\lVert K_{12}\rVert_2+\lVert K_{21}\rVert_2.
\]
Here the subscripts denote Schatten norms; 
Recall that 
\begin{equation}
\tilde{K}^\varepsilon_{ij}
=1_0
\begin{pmatrix}
-G_{\varepsilon,i}^*&F_{\varepsilon,i}^*\\
G_{\varepsilon,-i}^*&-F_{\varepsilon,-i}^*
\end{pmatrix}
\begin{pmatrix}I&\bH\\0&I\end{pmatrix}
\begin{pmatrix}
F_{\varepsilon,-j}&F_{\varepsilon,j}\\
G_{\varepsilon,-j}&G_{\varepsilon,j}
\end{pmatrix}1_0.
\end{equation}
The same formula without $\varepsilon$ gives
$\tilde{K}^{\mathrm{fp}}_{ij}$. Notice that every term in \eqref{eqn:factorization-FG} is of the form `$F_\varepsilon G_\varepsilon$' or `$G_\varepsilon \bH G_\varepsilon$' or their transpositions.

\begin{proposition}\label{prop:fixedpointtraceclass}
Under the assumptions of Theorem~\ref{thm:pointwiseconvergence},
all four blocks of
\[\mathcal M\tilde{K}^\varepsilon\mathcal M^{-1}\quad\text{and}\quad
\mathcal M\tilde{K}^{\mathrm{fp}}\mathcal M^{-1}\]
are trace class
for sufficiently small $\varepsilon>0$. Moreover, for some
$\varepsilon_0>0$,
\begin{equation}\label{eqn:expanded-uniform-membership}
\sup_{0<\varepsilon<\varepsilon_0}
\lVert\mathcal M K^\varepsilon\mathcal M^{-1}\rVert_{1|2}
+\lVert\mathcal M K^{\mathrm{fp}}\mathcal M^{-1}\rVert_{1|2}<\infty.
\end{equation}
Thus the limiting Fredholm Pfaffian is well defined in the sense of
Appendix~\ref{app:generalizedFredholm}.
\end{proposition}

\begin{proof}
Let $a$ be as in Lemma~\ref{lem:firstSbound}. Choose
$\mu>4(a+|\brho|+1)$, and let $\eta>0$ be the cubic decay constant
in Lemma~\ref{lem:secondSbound}. In this proof, abbreviate
$Q_\varepsilon=Q_\varepsilon^{\eta,\mu}$ and $Q=Q^{\eta,\mu}$.
Recall the definition of \(G_{\varepsilon,\pm j}\) in \eqref{eqn:Gdef}. Lemma~\ref{lem:secondSbound} and Lemma~\ref{lem:revised-hitting-tail} give
\begin{equation}\label{eqn:averaged-hitting-row-bound}
\begin{split}
&\lVert G_{\varepsilon,j}(\bz,\cdot)M_{-\theta_j}\rVert_2
+\lVert G_{\varepsilon,-j}(\bz,\cdot)M_{\theta_j}\rVert_2
\leq C Q_\varepsilon(\bz),\\
&\lVert G_j(\bz,\cdot)M_{-\theta_j}\rVert_2
+\lVert G_{-j}(\bz,\cdot)M_{\theta_j}\rVert_2
\leq C Q(\bz).
\end{split}
\end{equation}
These estimates also justify the defining expectations as weighted
$L^2$-valued integrals.

For every term that is a composition of \(F_\varepsilon\) with \(G_\varepsilon\), using $\lVert f\otimes h\rVert_1=\lVert f\rVert_2\lVert h\rVert_2$, we obtain the following bound: 
\begin{equation}\label{eqn:firstoperatorrankonebound}
\begin{split}
\lVert M_{-\theta_i}F_{\varepsilon,i}^*
G_{\varepsilon,-j}M_{\theta_j}\rVert_1
&\leq\int_{\mathbb R}
\lVert F_{\varepsilon,i}(\bz,\cdot)M_{-\theta_i}\rVert_2
\lVert G_{\varepsilon,-j}(\bz,\cdot)M_{\theta_j}\rVert_2\,d\bz\\
&\leq C\int_{\mathbb R}W_{\varepsilon,i}(\bz)
Q_\varepsilon(\bz)\,d\bz.
\end{split}
\end{equation}
The same is true for any other sign combination or transposition.
Choose \[\max\{a,|\brho|\}<\lambda<\mu\] in
Lemma~\ref{lem:revised-hitting-tail}. By \eqref{eqn:firstSbound-improved} and \eqref{eqn:uniformhittingtail}, \(\int_{\mathbb R}W_{\varepsilon,i}(\bz)
Q_\varepsilon(\bz)\,d\bz<\infty\). In particular, when \(\bz>\delta\varepsilon^{-1}\),
\[W_{\varepsilon,i}(\bz)
Q_\varepsilon(\bz) \leq Ce^{-(\lambda-a)\bz}\exp\{\varepsilon^{-3/2}[C\varepsilon^{1/2}\log(2\varepsilon^{-1/2})-c]\}\leq Ce^{-(\lambda-a)\bz}.\]

Every term that is a composition of \(G_\varepsilon\) with \(G_\varepsilon\) must contain a \(\bH\) term in the middle. Using
$|\bH(\bz,\bz')|\leq\tfrac12e^{|\brho|(|\bz|+|\bz'|)}$ and
\eqref{eqn:averaged-hitting-row-bound}, we have the following bound:
\begin{equation}\label{eqn:doublehittinguniformtracebound}
\lVert M_{\theta_i}G_{\varepsilon,-i}^*\bH
G_{\varepsilon,-j}M_{\theta_j}\rVert_1
\leq C\int_{\mathbb R^2}
e^{|\brho|(|\bz|+|\bz'|)}
Q_\varepsilon(\bz)Q_\varepsilon(\bz')\,d\bz\,d\bz',
\end{equation}
which is finite if we choose \(\lambda \) in Lemma~\ref{lem:revised-hitting-tail} such that \(\lambda>|\brho|\). The same argument applies
to the other combinations of hitting factors and to their limits.
Thus \(\tilde{K}^\varepsilon_{ij}\) and its limit are both trace class after conjugation, with uniform bounds.

It remains to treat $R^\varepsilon$. Lemma~\ref{lem:randomwalktraceconverge} gives uniform trace-norm bounds for the conjugated \(R_{11},R_{22}\) terms, including their limits. Proposition~\ref{prop:freeHfreeuniformHSbound} gives the Hilbert--Schmidt bound for \(R_{12}\). Finally $R_{21}^\varepsilon=R_{21}^{\mathrm{fp}}=0$.
This proves \eqref{eqn:expanded-uniform-membership}.
The factorization \eqref{eqn:factorization-FG}, $\bH^*=-\bH$,
and the explicit extended terms show that $JK^{\mathrm{fp}}$ is
skew-symmetric. The exponential conjugation is symplectic, so the
limiting mixed Fredholm Pfaffian is well defined.
\end{proof}

\begin{remark}
In the proof, every term is trace class except $R_{12}$ when \(\by_i=\by_j =0\).
\end{remark}

\begin{proof}[Proof of Proposition~\ref{prop:mixedSchattenconvergence}]
We first show that, for $\bz\neq\mathfrak h_0(0)$,
\begin{equation}\label{eqn:averaged-hitting-row-convergence}
\begin{split}
&\lVert[G_{\varepsilon,j}(\bz,\cdot)-G_j(\bz,\cdot)]
 M_{-\theta_j}\rVert_2+
\lVert[G_{\varepsilon,-j}(\bz,\cdot)-G_{-j}(\bz,\cdot)]
 M_{\theta_j}\rVert_2\longrightarrow0.
\end{split}
\end{equation}
In \eqref{eqn:Gdef}, insert a continuous compactly supported cutoff
$0\leq\chi_R\leq1$ equal to one on
$[0,R]\times[-R,\alpha+\gamma R]$.
By Lemma~\ref{lem:secondSbound} and the support condition
$\bb\leq\alpha+\gamma\bx$, the discarded part has weighted
$L^2$ norm at most $C(e^{-\eta R^3}+e^{-\mu R})$, uniformly in
$\varepsilon$, and the same bound holds for the limiting integral.
On the support of $\chi_R$, Lemma~\ref{lem:secondSbound} gives
uniform weighted $L^2$ convergence of the integrands.
The limiting map
$(\bx,\bb)\mapsto
\chi_R(\bx,\bb)g_{\pm j,\bx}(\bb,\cdot)M_{\mp\theta_j}$
is continuous and compactly supported with values in
$L^2([0,\infty))$, so it can be uniformly approximated by finite
sums $\sum_{\ell=1}^N\phi_\ell(\bx,\bb)v_\ell$, where
$\phi_\ell\in C_c([0,\infty)\times\mathbb R)$ and
$v_\ell\in L^2([0,\infty))$.
The vague convergence $\nu_{\bz}^\varepsilon\to\nu_{\bz}$ from
Lemma~\ref{lem:revised-hitting-measures} applies to each scalar
coefficient, while the approximation error is uniform because
the measures have mass at most one.
Thus the truncated integrals converge in weighted $L^2$.
Letting $R\to\infty$ proves
\eqref{eqn:averaged-hitting-row-convergence}.

Now write each `$F_\varepsilon G_\varepsilon$' term as an integral of rank-one
operators over the starting variable $\bz$, as in
\eqref{eqn:firstoperatorrankonebound}.
Lemma~\ref{lem:firstSbound} and
\eqref{eqn:averaged-hitting-row-convergence} imply convergence
of the two weighted vectors for almost every $\bz$.
Using
\[
\lVert f_\varepsilon\otimes h_\varepsilon-f\otimes h\rVert_1
\leq\lVert f_\varepsilon-f\rVert_2\lVert h_\varepsilon\rVert_2
    +\lVert f\rVert_2\lVert h_\varepsilon-h\rVert_2,
\]
the rank-one integrands converge in trace norm.
Their norms have the common integrable upper bounds for
$W_{\varepsilon,i}Q_\varepsilon$ and $W_iQ$ obtained in the preceding
proof. Dominated convergence therefore proves trace-norm convergence
of every such term.

For a `$G_\varepsilon \bH G_\varepsilon$' term, use the rank-one integral over
$(\bz,\bz')$. Equation~\eqref{eqn:averaged-hitting-row-convergence}
gives convergence almost everywhere;
\eqref{eqn:doublehittinguniformtracebound} supplies a common
integrable upper bound. Dominated convergence proves trace-norm
convergence here as well. The middle kernel $\bH(\bz,\bz')$ is
independent of $\varepsilon$, so its jump on the diagonal causes no
additional difficulty. Formula~\eqref{eqn:factorization-FG} and
its limiting version identify the resulting kernel with
$\tilde{K}^{\mathrm{fp}}$.

For \(R^\varepsilon_{11}\) and \(R^\varepsilon_{22}\), the trace-norm convergence is proved in Lemma~\ref{lem:randomwalktraceconverge}.

For $R_{12}^\varepsilon$, the centered transition factors in
\eqref{eqn:free-transition-exponential-moments} converge to their Gaussian limits
and have uniformly bounded exponential moments of every fixed order.
It follows that, for almost every $(\bz,\bv)$,
\[
R_{12}^\varepsilon(i,j)(\bz,\bv)
\longrightarrow
-\bigl(e^{\bs_i\bD+\by_i\bD^2}\bH
       e^{-\bs_j\bD+\by_j\bD^2}\bigr)(\bz,\bv).
\]
When $\by_i+\by_j>0$, the limiting displacement has a density and
does not charge the jump of $\bH$. When the sum is zero, necessarily
$i=j$, and the displacement tends to zero, giving convergence for
$\bz\neq\bv$. The square-integrable bound in
\eqref{eqn:freeHfreekernelbound}, after conjugation, proves
Hilbert--Schmidt convergence. Since $R_{21}^\varepsilon=0$,
the proposition follows.
\end{proof}

\begin{proof}[Completion of the proof of
Theorem~\ref{thm:pointwiseconvergence}]
\begingroup\emergencystretch=1em
Proposition~\ref{prop:fixedpointtraceclass} places the discrete and
limiting conjugated kernels in the mixed Schatten class, and
Proposition~\ref{prop:mixedSchattenconvergence} gives convergence
in that class. Their $J$-skew-symmetry and
Corollary~\ref{cor:appendixmixedSchattencontinuity} imply
\[
\Pf(J+JK^\varepsilon)\longrightarrow\Pf(J+JK^{\mathrm{fp}}),
\]
where the Pfaffians are evaluated using the admissible exponential
conjugation, as in Appendix~\ref{app:generalizedFredholm}.\par\endgroup

The estimates above are locally uniform in the thresholds \(\bs_1,\ldots,\bs_m\).
The same dominated-convergence arguments show that the conjugated limiting kernel
is continuous in the mixed Schatten norm. Hence its Fredholm
Pfaffian is continuous in $(\bs_1,\ldots,\bs_m)$.
Sandwiching lattice-rounded thresholds between $\bs_i-\eta$ and
$\bs_i+\eta$, and then letting $\eta\downarrow0$, removes the bounded
lattice corrections. Spatial interpolation changes the rescaled
height by $O(\varepsilon^{1/2})$ over one lattice interval and is
handled by the same sandwich. This proves
\eqref{eqn:halffixedpointformula}.
\end{proof}

\appendix
\section{Generalized Fredholm determinants and Pfaffians}
\label{app:generalizedFredholm}

We record the definitions and properties used in Sections 2--3.
For the classical determinant and Pfaffian theory, see
\cite{Simon1979TraceIA} and \cite[Appendix~B]{OQR17}.

\subsection{Definitions and continuity}
Let $\mathcal H$ be a separable Hilbert space and let
$\mathfrak S_1$ and $\mathfrak S_2$ denote the trace-class and
Hilbert--Schmidt ideals, with norms $\lVert\cdot\rVert_1$ and $\lVert\cdot\rVert_2$.
On $\mathcal H\oplus\mathcal H$, set
\begin{equation}\label{def:mixedschatternclass}
    J=\begin{pmatrix}0&I\\-I&0\end{pmatrix},\qquad
\mathfrak S_{1|2}(\mathcal H)
=\{A=(A_{ab})_{a,b=1}^2:
A_{11},A_{22}\in\mathfrak S_1,\quad
A_{12},A_{21}\in\mathfrak S_2\}.
\end{equation}
We use the norm and diagonal trace
\[
\lVert A\rVert_{1|2}
=\lVert A_{11}\rVert_1+\lVert A_{22}\rVert_1+\lVert A_{12}\rVert_2+\lVert A_{21}\rVert_2,
\qquad
\tr_{\mathrm d}A
=\tr A_{11}+\tr A_{22}.
\]

For Hilbert--Schmidt operators, the regularized determinant and the regularized Pfaffian are
\[
\det\nolimits_2(I+A)=\det\bigl((I+A)e^{-A}\bigr),\qquad
\Pf_2(J+B)=\Pf\bigl((J-JBJ)e^{BJ}\bigr).
\]
The operators on the right are trace-class perturbations of $I$ and
$J$, respectively. For $A\in\mathfrak S_{1|2}$, define
\begin{equation}
\label{eqn:appendixmixedFredholmdeterminant}
\det(I+A):=e^{\tr_{\mathrm d}A}\det\nolimits_2(I+A).
\end{equation}
If, in addition, $(JA)^*=-JA$, define
\begin{equation}
\label{eqn:appendixmixedFredholmPfaffian}
\Pf(J+JA):=e^{\frac12\tr_{\mathrm d}A}\Pf_2(J+JA).
\end{equation}
These definitions agree with the ordinary Fredholm determinant and
Pfaffian when $A$ is trace class, and satisfy
\begin{equation}
\label{eqn:appendixmixedPfsquare}
\Pf(J+JA)^2=\det(I+A),\qquad \Pf(J)=1.
\end{equation}
The Pfaffian is analytic on the subspace $(JA)^*=-JA$.
Thus $\sqrt{\det(I+A)}$ in the main text denotes this analytic
Pfaffian, not the pointwise nonnegative square root.

\begin{lemma}[Continuity and zeros]
\label{lem:appendixmixeddeterminantcontinuity}
If $A_n,A\in\mathfrak S_{1|2}$ and
$\lVert A_n-A\rVert_{1|2}\to0$, then $\det(I+A_n)\to\det(I+A)$.
Moreover,
\[
\det(I+A)=0\quad\Longleftrightarrow\quad
I+A\text{ is not invertible}.
\]
\end{lemma}

\begin{corollary}[Mixed Schatten continuity]
\label{cor:appendixmixedSchattencontinuity}
Under the convergence assumptions of the preceding lemma, if
$(JA_n)^*=-JA_n$ and $(JA)^*=-JA$, then
\[
\Pf(J+JA_n)\longrightarrow\Pf(J+JA).
\]
\end{corollary}

These assertions follow from the Hilbert--Schmidt continuity of
$\det_2$ and $\Pf_2$, the zero criterion for $\det_2$, and continuity
of $\tr_{\mathrm d}$ in the mixed norm. In particular,
only the two diagonal blocks require trace-norm convergence.

\subsection{Differentiation and finite-rank identities}
\begin{lemma}[Differentiation]
\label{lem:appendixmixeddifferentiation}
Suppose that $A(t)\in\mathfrak S_{1|2}$ is differentiable in the mixed
norm, $(JA(t))^*=-JA(t)$, and $A'(t)$ is trace class.
Where $I+A(t)$ is invertible,
\begin{equation}
\label{eqn:appendixmixedPfaffianderivative}
\frac{d}{dt}\Pf(J+JA(t))
=\frac12\Pf(J+JA(t))\tr
\bigl((I+A(t))^{-1}A'(t)\bigr).
\end{equation}
\end{lemma}
This is the usual differentiation formula, applied to
\eqref{eqn:appendixmixedFredholmPfaffian}; the trace is ordinary
because $A'(t)$ is trace class. In the Kolmogorov calculation,
the derivatives are finite rank.

\begin{lemma}[Finite-rank perturbations]
\label{lem:appendixfiniterankperturbation}
Let $A\in\mathfrak S_{1|2}$ with $I+A$ invertible, and let
$U:\mathbb C^r\to\mathcal H\oplus\mathcal H$ and
$V:\mathcal H\oplus\mathcal H\to\mathbb C^r$ be bounded.
Writing $R=(I+A)^{-1}$, one has
\begin{equation}
\label{eqn:appendixmatrixdeterminantlemma}
\det(I+A+UV)=\det(I+A)\det_{\mathbb C^r}(I_r+VRU),
\qquad
\tr(RUV)=\tr_{\mathbb C^r}(VRU).
\end{equation}
If $JA$ and $J(A+UV)$ are skew-symmetric and
\[
\det_{\mathbb C^r}(I_r+\lambda VRU)=(1+\lambda c)^2
\quad\text{for all }\lambda\in\mathbb C,
\]
then
\begin{equation}
\label{eqn:appendixperfectsquarePfupdate}
\Pf(J+J(A+UV))=\Pf(J+JA)(1+c).
\end{equation}
\end{lemma}

\begin{proof}
The determinant and trace identities follow from their finite-dimensional
versions by finite-rank approximation. Set
\[
f(\lambda)=\frac{\Pf(J+J(A+\lambda UV))}{\Pf(J+JA)}.
\]
By \eqref{eqn:appendixmixedPfsquare}, $f$ is entire,
$f(\lambda)^2=(1+\lambda c)^2$, and $f(0)=1$.
Hence $f(\lambda)=1+\lambda c$.
\end{proof}

\begin{corollary}[Reduction and vanishing]
\label{cor:appendixinitialconditionproperties}
Let $A\in\mathfrak S_{1|2}$. Whenever a Pfaffian is mentioned,
assume that $JA$ is skew-symmetric.
\begin{enumerate}
\item For a finite block-triangular decomposition compatible with the
two components, the mixed determinant factors over its diagonal blocks.
In particular, an identity diagonal block of $I+A$ can be deleted.
\item If $A$ is strictly block triangular in such a decomposition,
then $\det(I+A)=1$ and $\Pf(J+JA)=1$.
\item If $Af=-f$ for some nonzero $f\in\mathcal H\oplus\mathcal H$,
then $\det(I+A)=0$ and $\Pf(J+JA)=0$.
\end{enumerate}
\end{corollary}

The triangular identities follow by finite-rank approximation preserving
the block decomposition; the Pfaffian sign in (2) is fixed along
$\lambda A$ by its value at $\lambda=0$.
Assertion (3) follows from the zero criterion and
\eqref{eqn:appendixmixedPfsquare}.

\subsection{Conjugation}
\begin{lemma}[Conjugation invariance]
\label{lem:appendixconjugationinvariance}
Let $A\in\mathfrak S_{1|2}$ and let $S=\operatorname{diag}(S_1,S_2)$
be bounded and invertible. Then
\[
\det(I+SAS^{-1})=\det(I+A).
\]
If also $S^*JS=J$ and $(JA)^*=-JA$, then
\begin{equation}
\label{eqn:appendixPfconjugation}
\Pf\bigl(J+J(SAS^{-1})\bigr)=\Pf(J+JA).
\end{equation}
For integral kernels on
$\mathcal H=L^2(\{1,\ldots,m\}\times[0,\infty))$,
the same conclusions hold for diagonal multiplication operators
$S,S^{-1}$ that are bounded on every bounded spatial interval,
provided both $A$ and the kernel $SAS^{-1}$ belong to
$\mathfrak S_{1|2}$.
\end{lemma}

\begin{proof}
For bounded $S,S^{-1}$, both $\det_2$ and the diagonal trace are
invariant under block-diagonal similarity. The Pfaffian identity follows
by squaring and analytic continuation along $\lambda A$, normalized
at $\lambda=0$.

For the multiplication operators, let $P_L$ project onto
$\{1,\ldots,m\}\times[0,L]$ and put $Q_L=P_L\oplus P_L$.
On $\operatorname{ran}Q_L$, the two compressed kernels are related by
a bounded similarity, since $Q_L$ commutes with $S$.
Moreover, for every $B\in\mathfrak S_{1|2}$,
\[
\lVert Q_LBQ_L-B\rVert_{1|2}\longrightarrow0.
\]
Apply this to $B=A$ and $B=SAS^{-1}$ and use mixed-norm continuity.
\end{proof}

A common real unitary change of variables in the two components, including
the dilation in \eqref{eqn:scaledkernel}, also preserves the determinant
and Pfaffian. Consequently, if an unweighted kernel $K$ is not known
to be in $\mathfrak S_{1|2}$ but
$K^{(\theta)}=\mathcal M_\theta K\mathcal M_\theta^{-1}$ is, we set
\[
\det(I+K):=\det(I+K^{(\theta)}),\qquad
\Pf(J+JK):=\Pf(J+JK^{(\theta)}).
\]
Here $\mathcal M_\theta$ is a diagonal exponential
conjugation with $\mathcal M_\theta^*J\mathcal M_\theta=J$.
The preceding lemma shows that these values are independent of the weights.

\bibliographystyle{alpha} 
\bibliography{halfspaceTASEPVersion2}
\end{document}